\documentclass[reqno,11pt]{amsart}

\usepackage[
  a4paper,
  textwidth=15cm,
  textheight=21cm,
  centering
]{geometry}

\usepackage{dsfont}
\usepackage{times,amsmath,cancel,stmaryrd,graphicx}
\usepackage{amsfonts,enumitem,amssymb,color,mathrsfs}
\usepackage[colorlinks=true]{hyperref}
\hypersetup{
citecolor=blue
}
\usepackage{cleveref,enumerate}
\usepackage{mathtools}
\usepackage{lmodern}
\usepackage[dvipsnames]{xcolor}
\usepackage{comment}
\usepackage{tikz}
\usetikzlibrary{arrows.meta,calc}
\usepackage{cite}

\usepackage{float}

\newtheorem{thm}{Theorem}[section]
\newtheorem{lem}[thm]{Lemma}
\newtheorem{cor}[thm]{Corollary}
\newtheorem{prop}[thm]{Proposition}

\newtheorem{deff}[thm]{Definition}
\theoremstyle{definition}
\newtheorem{rem}[thm]{Remark}

\numberwithin{equation}{section}

\newcommand{\R}{\mathbb{R}}

\newcommand{\rd}{\mathrm{d}}

\newcommand{\dhr}{\mathrel{\lhook\joinrel\relbar\kern-.8ex\joinrel\lhook\joinrel\rightarrow}}

\makeatletter
\renewcommand{\@captionfont}{\small}
\makeatother

\allowdisplaybreaks

\begin{document}

\title[Completion of DNA replication is constrained by origin activity]
{Completion of DNA replication is constrained by the spatiotemporal organisation of origin firing}



%
\author{Ahmad Alkhaled}
\address{King Abdullah University of Science and Technology (KAUST)\\
CEMSE Division\\
Thuwal 23955-6900\\
Saudi Arabia}
\email{ahmad.alkhaled@kaust.edu.sa}

\author{Francisco Berkemeier}
\address{%
University of Cambridge, Department of Pathology,
Cambridge CB2 1QP, United Kingdom\newline
\hspace*{\parindent}University of Cambridge, Department of Genetics,
Cambridge CB2 3EH, United Kingdom}
\email{fp409@cam.ac.uk}

\author{Michael A.\ Boemo}
\address{%
University of Cambridge, Department of Pathology,
Cambridge CB2 1QP, United Kingdom\newline
\hspace*{\parindent}University of Cambridge, Department of Genetics,
Cambridge CB2 3EH, United Kingdom}
\email{mb915@cam.ac.uk}

\author{Katerina Nik}
\address{King Abdullah University of Science and Technology (KAUST)\\
CEMSE Division\\
Thuwal 23955-6900\\
Saudi Arabia}
\email{katerina.nik@kaust.edu.sa}
%

\begin{abstract}
DNA replication requires the coordination of origin firing and fork progression to ensure the entire genome is timely duplicated before cell division. Yet origin firing is stochastic, giving rise to the classical random completion problem of how probabilistic local events can nevertheless ensure reliable genome duplication. Although several biological mechanisms have been proposed to resolve this problem, a quantitative account of how heterogeneous initiation and fork speed govern the persistence of the final unreplicated regions is still lacking. To address this gap, we introduce a population-level kinetic framework that extends KJMA nucleation-and-growth models by tracking unreplicated intervals over size, genomic position and time. We establish well-posedness of the resulting mean-field system and global existence for compatible data. Notably, by introducing a local initiation mass function, we quantify how the density and spatial organisation of origin firing constrain replication completion, yielding novel and sharp upper bounds on both the worst-locus unreplicated fraction and locuswise near-completion time. These results provide a rigorous and computable foundation for mapping vulnerabilities in replication completion and relating persistent unreplicated regions to replication stress and genome instability. 
\end{abstract}
%

\maketitle
\section{Introduction}
\label{Sec1}

DNA replication is the process by which a cell duplicates its genome before division, and its accurate regulation is essential for preserving genome integrity across cell generations. In eukaryotes, replication takes place during the S phase of the cell cycle and is initiated at multiple genomic loci known as origins of replication~\cite{leonard2013dna}. When an origin fires, two replication forks are formed and move in opposite directions along the DNA molecule, copying the template as they progress. These forks continue until they meet forks from neighbouring origins, reach chromosome ends, or encounter barriers that cause them to slow or stop moving~\cite{mirkin2007replication}. Together, these mechanisms define the replication programme, understood here as the highly regulated, ordered sequence in which different sections of a genome are duplicated within S phase~\cite{rhind2013dna,kelly2000regulation,sclafani2007cell}. Accurate replication therefore requires the timely coordination of origin firing and fork progression, so that replicated domains expand and merge until the genome has been copied once and completely prior to cell division. When this coordination fails, replication can be delayed or incomplete, with consequences for genome instability, replication stress, ageing and cancer~\cite{zeman2014causes,flach2014replication,gaillard2015replication,macheret2015dna}.

Despite this regulation, replication is inherently stochastic. Origins fire with probabilities that vary across the genome and through S phase, shaped by chromatin state, origin accessibility and the local regulatory environment~\cite{wang2021genome,bechhoefer2012replication}. Fork progression is likewise heterogeneous, with local sequence composition, chromatin organisation, and replication-transcription conflicts among the many factors that can alter fork speed~\cite{yousefi2019stochasticity,conti2007replication}. In an individual cell, the time at which a locus is replicated is therefore a random outcome of origin firing and fork progression~\cite{dileep2018single}. At the population level, however, genomic loci exhibit reproducible replication-timing profiles, indicating that these stochastic local events are constrained by a regulated genome-wide programme~\cite{hansen2010sequencing,zhao2020high}. Nonetheless, a reproducible average timing profile does not ensure that every realisation of the replication programme completes on time, since stochastic origin firing may occasionally leave large intervals to be replicated late in S phase. This is the essence of the classical ``random completion problem'', which asks how stochastic initiation can nevertheless ensure timely and reliable genome duplication~\cite{hyrien2003paradoxes}.

Biological resolutions to this problem include excess licensed origins that provide dormant initiation capacity~\cite{blow2011dormant}, increasing initiation probability as S phase progresses~\cite{jun2008just,goldar2008dynamic,yang2008xenopus}, regulation of cellular replication capacity to control S phase duration~\cite{pennycook2020e2f,bertoli2021control}, and spatial organisation of potential origins that limits large origin-poor regions~\cite{hyrien2003paradoxes}. Experimentally, replication completion is difficult to assess because rare, cell-specific regions that persist late into S phase can be obscured by population-averaged timing measurements, such as Repli-seq~\cite{hansen2010sequencing,zhao2020high}. Single-molecule approaches can resolve origin usage and fork progression more directly~\cite{conti2007replication}, but sample only partial realisations of the replication programme. Studies of dormant origins and under-replicated DNA have already established that origin availability and spacing influence faithful completion~\cite{karschau2012optimal,ge2007dormant,ibarra2008excess,almamun2016inevitability,moreno2016unreplicated}.

These observations recast completion as a quantitative constraint on the persistence of the final unreplicated regions, rather than simply the endpoint of an average timing profile~\cite{wendel2014completion}. They also expose a fundamental gap in how the random completion problem is formulated for heterogeneous genomes. Although the importance of origin availability and organisation is well established, it remains less clear how spatially and temporally varying initiation and fork-speed landscapes determine the persistence and distribution of unreplicated regions, and hence the locuswise probability of remaining unreplicated by a given time. This motivates a coarse-grained kinetic description in which the many molecular determinants of initiation and elongation are captured by effective origin-firing rates and fork speeds~\cite{Herrick2002,PhysRevE.71.011908,PhysRevE.71.011909,Nieto2022}, without the need to model every protein and enzymatic step explicitly~\cite{yeeles2015regulated,yeeles2017replisome,gauthier2012modeling}. Such a description retains the spatial and temporal organisation relevant to replication kinetics.

In this work, we develop a mathematical framework to determine how initiation and fork progression govern the persistence of unreplicated regions, providing a new quantitative view of the random completion problem in heterogeneous genomes. We first establish that, for prescribed initiation and fork-speed profiles, the reduced system generates a unique replication programme that depends continuously on its initial data. This analytical control makes it possible to derive completion estimates in heterogeneous settings. We then prove an explicit upper bound on locuswise near-completion time \(T_\varepsilon\), defined as the time by which almost every locus has replicated in all but an arbitrarily small fraction $\varepsilon$ of cells, together with bounds on expected replication-timing observables. These estimates are subsequently evaluated against stochastic simulations driven by Repli-seq-derived initiation profiles~\cite{hansen2010sequencing}. Importantly, we show that completion depends not only on the overall density of initiation, but also on how initiation is distributed across the genome. This dependence is captured by a local initiation ``mass'', which reflects the initiation available in the least active genomic regions and explains why landscapes with similar overall activity can differ in their susceptibility to persistent unreplicated regions. Overall, this work establishes a rigorous and computable connection between heterogeneous initiation landscapes, fork progression and genome-wide near-completion dynamics, allowing experimentally informed replication programmes to reveal completion bottlenecks and yield conservative, conditional guarantees on the resolution of residual unreplicated DNA.

\section{Derivation of the Model}
\label{Sec2}
The most natural setting for DNA replication kinetics is nucleation-and-growth theory. The classical Kolmogorov--Johnson--Mehl--Avrami (KJMA) model, developed to describe phase transformations such as crystallisation~\cite{kolmogorov-crystallization, johnson1939reaction, Avrami1939KineticsOP, Avrami1940KineticsOP, Avrami1941GranulationPC}, casts origin firing and bidirectional fork progression as a one-dimensional analogue of nucleation and growth~\cite{Herrick2002,PhysRevE.71.011908,PhysRevE.71.011909}. In this analogy, nucleation corresponds to origin firing, the expansion of crystallised domains to passive replication by moving forks, and domain coalescence corresponds to the merging of converging forks (Figure \ref{fig:kjma_replication_analogy}).

Replication kinetics are encoded by two prescribed fields over genomic position (locus) $x$ and time $t$. The \textit{initiation rate} $I(x,t)$ specifies where and when new origins fire within unreplicated DNA, while the \textit{fork speed} $v(x,t)$ determines how rapidly established replication forks move the genome. Formally, $I(x, t)$ is the initiation rate per unit time and per unit length of unreplicated DNA, whereas $v(x, t)$ is the fork speed at position $x$ and time $t$.

Assuming that the initial unreplicated set is \(\mathbb R\), that potential
initiation events form a homogeneous Poisson point process on
\(\mathbb R\times(0,\infty)\) with intensity
\(I(x,t)\equiv I_0>0\), and that \(v(x,t)\equiv v>0\), the
replicated fraction is independent of \(x\in\mathbb R\) and satisfies
\begin{equation}
\label{eq:replicated_fraction_homogeneous_case}
f(t)=1-e^{-I_0vt^2},
\qquad t\geq0.
\end{equation}
Indeed, for every \(x\in\mathbb R\) and \(t\geq0\), the backward causal
cone has intensity
\[
\int_0^t2I_0v(t-\tau)\,\rd\tau=I_0vt^2,
\]
and the Poisson void probability gives
\eqref{eq:replicated_fraction_homogeneous_case} 
~\cite{PhysRevE.71.011908}.

\begin{figure}[t]
    \centering
    \includegraphics[width=.82\textwidth]{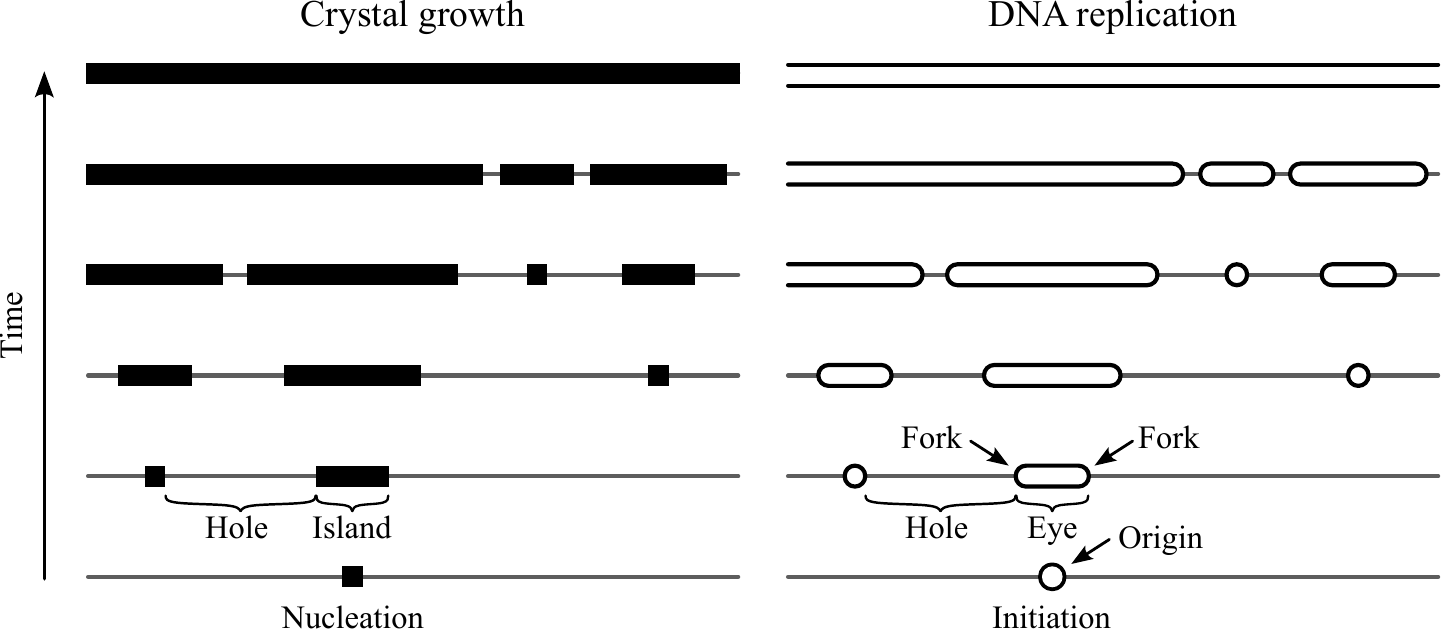}
    \caption{
    Analogy between KJMA nucleation and growth and DNA replication.
    In crystal growth, stochastic nucleation creates expanding islands that consume holes until the domain is fully crystallised, whereas in DNA replication, origin firing creates replication bubbles (or ``eyes'') whose bidirectional forks copy unreplicated DNA until neighbouring forks coalesce, completing genome duplication into two identical DNA molecules.
    }
    \label{fig:kjma_replication_analogy}
\end{figure}

Relation \eqref{eq:replicated_fraction_homogeneous_case} summarises the average progress of genome duplication and shows how initiation and fork speed jointly determine the temporal increase of replicated DNA. It is closely tied to replication timing, since $1-f(t)$ gives the probability that a locus remains unreplicated at time $t$ in the homogeneous setting~\cite{berkemeier2025dna}. However, this global fraction does not describe how the remaining unreplicated DNA is organised. The same value of $f(t)$ may correspond to many small unreplicated gaps or to a smaller number of large gaps, with different implications for fork density, coalescence and completion. This distinction is biologically important because genome duplication depends not only on the total amount of DNA left to copy, but also on how that DNA is arranged. Short, dispersed gaps can be replicated by nearby converging forks, whereas long initiation-poor intervals require forks to travel farther and may remain vulnerable late in S phase~\cite{ge2007dormant,blow2011dormant,yang2008xenopus}. The size and distribution of unreplicated regions therefore carry information about the robustness of completion that is not captured by the global replicated fraction alone.

To capture this finer structure, we track the distribution of unreplicated intervals rather than only their total mass. The number density of holes of size $\ell$ at time $t$, $\rho_{h}(\ell, t)$, can be derived by extending the argument used to obtain \eqref{eq:replicated_fraction_homogeneous_case}, as shown in~\cite{Nieto2022}, yielding
\begin{equation}
\rho_{h}(\ell, t)=(I_0 t)^2 e^{-I_0 t ( \ell +v t)}, \qquad t>0, \quad \ell > 0.
\end{equation} 
Jun et al.~\cite{PhysRevE.71.011908} further generalised the one-dimensional KJMA model to allow a time-dependent initiation rate $I(t)$, obtaining an evolution equation for the hole density,
\begin{equation}
\label{eq:evolution_eq_rho(l,t)}
\partial_t\rho_h(\ell,t)
=
2v\partial_\ell\rho_h(\ell,t)
-I(t)\ell\rho_h(\ell,t)
+2I(t)\int_\ell^\infty \rho_h(y,t)\,\rd y, \qquad t>0, \quad \ell > 0.
\end{equation}
Here the first term describes shrinkage, the second the annihilation of holes by initiation, and the last the creation of two child holes after initiation in a parent hole.

For \(x\in[0,L)\) and \(t\geq0\), Gauthier et al.~\cite{gauthier2012modeling} later introduced a mean-field rate-equation model linking the replicated fraction to fork-density fields, while allowing both the initiation rate and fork speed to vary with space and time. In this formulation, $f(x,t)$ denotes the population-level replicated fraction, namely the fraction of cells in which locus $x$ has replicated by time $t$. Introducing fork-density fields $\rho_+(x,t)$ and $\rho_-(x,t)$ for right- and left-moving forks, respectively, with local fork speeds $v_+(x,t)$ and $v_-(x,t)$, they obtained
\begin{equation}
\label{eq:sys_1}
\partial_t f(x,t)=v_+(x,t)\rho_+(x,t)+v_-(x,t)\rho_-(x,t),
\end{equation}
and
\begin{align}
\label{eq:sys_2}
\partial_t\rho_\pm(x,t) &\pm \partial_x\big(v_\pm(x,t)\rho_\pm(x,t)\big) \nonumber \\
&=
I(x,t)s(x,t)
-
\frac{\big(v_+(x,t) + v_-(x,t)\big)\rho_+(x,t)\rho_-(x,t)}{s(x,t)},
\end{align}
where $s(x,t)=1-f(x,t)$ is the unreplicated fraction, and the second term on the right-hand side of~\eqref{eq:sys_2} represents fork loss by coalescence under the
mean-field approximation and is defined on the set on which
\(s(x,t)>0\).

Nieto and V\'asquez~\cite{Nieto2022} studied a size-only hole-density
model related to~\eqref{eq:evolution_eq_rho(l,t)}, proving local existence and
uniqueness of weak solutions on \((0,\infty)\times(0,T)\), for some \(T>0\),
first for constant fork speed \(v>0\) and a prescribed bounded initiation rate
\(I(t)\), and then for size-dependent fork speed \(v(\ell)\) and nonlinear
initiation rate \(I(t,\ell,\rho_h)\).

For the derivation below and the subsequent analysis of the reduced
system, we therefore consider a generalised equation for the hole
density that depends on size \(\ell>0\), position \(x\in\mathbb R\),
and time \(t>0\), and do not include boundary effects at chromosome
ends. Although such a model might be more detailed than what is needed
for a given observable, it provides a unified description of the
models introduced above. Let \(T\in(0,\infty)\) be fixed.

In particular, as shown below,
integration with respect to \(\ell\) over \((0,\infty)\), together with
the mean-field approximation in~\eqref{eq:mean-field-Jh}, yields on
\(\mathbb R\times(0,T)\) the symmetric-speed form
of~\eqref{eq:sys_1}--\eqref{eq:sys_2} under
\[
v_+(x,t)=v_-(x,t)=v(x,t)
\quad\text{for a.e. }(x,t)\in\mathbb R\times(0,T),
\]
whereas omitting the dependence on \(x\) in \(I\) and \(\rho_h\) and
taking a constant fork speed \(v>0\) yields
\eqref{eq:evolution_eq_rho(l,t)} on
\((0,\infty)\times(0,T)\).

For \((x,t)\in\mathbb R\times(0,T)\), let \(I(x,t)\) and \(v(x,t)\)
denote the space- and time-dependent initiation rate and fork speed,
respectively. For a.e. \(t\in(0,T)\), assume that the unreplicated set is
almost surely a disjoint union of bounded intervals of positive length
and that the first-moment measure of these holes has density
\(\rho_h(\ell,x,t)\geq0\) with respect to
\(\rd\ell\,\rd x\), so that, for every non-negative measurable function
\(\psi:(0,\infty)\times\mathbb R\to[0,\infty]\),
\begin{equation}
\label{eq:hole-density-definition}
\mathbb E\left[
\sum_{H\text{ a hole at time }t}
\psi(|H|,\inf H)
\right]
=
\int_{\mathbb R}\int_0^\infty
\psi(\ell,x)\rho_h(\ell,x,t)\,\rd\ell\,\rd x\, .    
\end{equation}

Following the same mechanism as in the derivation
of~\eqref{eq:evolution_eq_rho(l,t)}, the evolution equation for
\(\rho_h(\ell,x,t)\) contains three contributions on the right-hand
side, namely hole shrinkage, hole annihilation, and hole creation, while
the motion of the left endpoint gives the transport term in \(x\) on the
left-hand side.

\begin{itemize}
    \item \textbf{Shrinkage.}
    Since the two endpoints of a hole move inward with speeds $v(x,t)$ and $v(x+\ell,t)$, respectively, the total shrinking speed is
    $$
    w(\ell,x,t) \coloneqq v(x,t)+v(x+\ell,t) \qquad \text{for }(\ell,x,t)\in (0,\infty)\times \mathbb{R}\times (0,T)\,,
    $$
    and the corresponding transport term in the size variable is
    $$
    \partial_\ell\big(w(\ell,x,t)\rho_h(\ell,x,t)\big) \qquad \text{for }(\ell,x,t)\in (0,\infty)\times \mathbb{R}\times (0,T)\,.
    $$
    If the dependence of $v(x,t)$ on $x$ is removed, then $w(\ell,t)=2v(t)$ and this reduces to $2v(t)\partial_\ell\rho_h(\ell,x,t)$.

    \item \textbf{Annihilation.}
    Initiation may occur anywhere inside the hole. Hence the loss rate is obtained by integrating the hazard over the whole interval
    $$
    -\left(\int_0^\ell I(x+s,t)\,\rd s\right)\rho_h(\ell,x,t) \qquad \text{for }(\ell,x,t)\in (0,\infty)\times \mathbb{R}\times (0,T)\,.
    $$
    If $I(x,t) \equiv I(t)$, this reduces to $-I(t)\ell\,\rho_h(\ell,x,t)$.

    \item \textbf{Creation.}
    There are two creation channels.

    \begin{enumerate}
        \item \emph{Left-subhole term.}
        If initiation occurs at $x+\ell$ inside a parent hole with the same left endpoint $x$ and size $m\ge \ell$, then a child hole of size $\ell$ is created. This contributes
        $$
        I(x+\ell,t)\int_\ell^\infty \rho_h(m,x,t)\,\rd m \qquad \text{for }(\ell,x,t)\in (0,\infty)\times \mathbb{R}\times (0,T)\,.
        $$

        \item \emph{Right-subhole term.}
        If initiation occurs at $x$, then a child hole of size $\ell$ is created from a parent with left endpoint $x-s$ and size $m=\ell+s$, where $s\ge 0$. This contributes
        $$
        I(x,t)\int_0^\infty \rho_h(\ell+s,x-s,t)\,\rd s \qquad \text{for }(\ell,x,t)\in (0,\infty)\times \mathbb{R}\times (0,T)\,.
        $$
    \end{enumerate}

    When the dependence on $x$ is removed, the two creation terms reduce to
    $$
    2I(t)\int_\ell^\infty \rho_h(m,t)\,\rd m \qquad \text{for }(\ell,t)\in (0,\infty)\times (0,T)\,.
    $$

    \item \textbf{Transport in $x$.}
    Since $\rho_h(\ell,x,t)$ is indexed by the left endpoint position, the left endpoint itself moves with velocity $v(x,t)$. Therefore the left-hand side contains, in addition to the time derivative, the advection term
    $$
    \partial_x\!\bigl(v(x,t)\rho_h(\ell,x,t)\bigr), \qquad \text{for }(\ell,x,t)\in (0,\infty)\times \mathbb{R}\times (0,T).
    $$
\end{itemize}

Combining these contributions, we obtain
\begin{align}
\label{eq:joint-hole}
\partial_t \rho_h(\ell,x,t)
&+ \partial_x\big(v(x,t)\rho_h(\ell,x,t)\big)\nonumber \\
&=
\partial_\ell\big(w(\ell,x,t)\rho_h(\ell,x,t)\big)
- \Bigl(\int_0^\ell I(x+s,t)\,\rd s\Bigr)\rho_h(\ell,x,t) \nonumber \\
&\quad + I(x+\ell,t)\int_\ell^\infty \rho_h(m,x,t)\,\rd m 
 + I(x,t)\int_0^\infty \rho_h(\ell+s,x-s,t)\,\rd s\,, 
\end{align}
in $\mathcal{D}^{\prime}((0, \infty) \times \mathbb{R} \times(0, T))$. 

To integrate~\eqref{eq:joint-hole} with respect to
\(\ell\in(0,\infty)\), we assume
\begin{subequations}
\label{eq:assumptions}
\begin{equation}
I\in L^\infty\bigl(\mathbb R\times(0,T)\bigr)\,,
\qquad
I(x,t)\geq0
\text{ for a.e. }(x,t)\in\mathbb R\times(0,T)\,,
\label{eq:assumption-I}
\end{equation}
and 
\begin{equation}
v\in L^\infty\bigl(0,T;W^{1,\infty}(\mathbb R)\bigr)\,,
\qquad
v(x,t)\geq0 \text{ for a.e. }(x,t)\in\mathbb R\times(0,T)\,,
\label{eq:assumption-v}
\end{equation}
and 
\begin{equation}
(1+\ell)\rho_h
\in
L^\infty(0,T;L^1((0,\infty);L^\infty(\mathbb R)))\,,
\label{eq:assumption-rhoh}
\end{equation}
and 
\begin{equation}
\label{eq:additional_assumptions}
 w\rho_h
\in L^1_{\mathrm{loc}}(\mathbb R\times(0,T);W^{1,1}(0,\infty))\,.   
\end{equation}
\end{subequations}

For a.e. \((x,t)\in\mathbb R\times(0,T)\), we define
\begin{equation}
\label{eq:rho_h_integral}
\rho_+(x,t)\coloneqq \int_0^\infty \rho_h(\ell,x,t)\,\rd \ell\, ,
\qquad
\rho_-(x,t)\coloneqq \int_0^\infty \rho_h(\ell,x-\ell,t)\,\rd \ell\, ,
\end{equation}
where the first and second quantities count right- and left-moving
forks at the left and right endpoints of the holes, respectively.

For the shrinkage term, \eqref{eq:additional_assumptions} allows us to set
\begin{equation}
\label{eq:J_h-definition}
J_h(x,t):=
\lim_{\ell\downarrow0}(w\rho_h)(\ell,x,t)
\end{equation}
for a.e. \((x,t)\in\mathbb R\times(0,T)\), where \(J_h(x,t)\) is the flux of holes through zero size and represents the annihilation of a
hole when two opposing forks meet. Moreover,
\[
\lim_{\ell\to\infty}(w\rho_h)(\ell,x,t)=0\,,
\]
and hence
\begin{equation}
\label{eq:size-flux-identity}
\int_0^\infty
\partial_\ell(w\rho_h)(\ell,x,t)\,\rd\ell
=
\bigl[(w\rho_h)(\ell,x,t)\bigr]_{\ell=0}^{\ell=\infty}
=
-J_h(x,t)\,.
\end{equation}
To obtain the reduced system, we use the mean-field approximation
from~\cite{gauthier2012modeling} and, assuming that
\(s(x,t)>0\) for a.e.
\((x,t)\in\mathbb R\times(0,T)\), impose the relation
\begin{equation}
\label{eq:mean-field-Jh}
J_h(x,t)
=
\frac{2v(x,t)}{s(x,t)}
\rho_+(x,t)\rho_-(x,t)
\end{equation}
for a.e.
\((x,t)\in\mathbb R\times(0,T)\). The subsequent analysis concerns
the reduced system obtained under~\eqref{eq:mean-field-Jh}.

We next consider the annihilation term and the left-subhole creation
term and, for a.e. \((x,t)\in\mathbb R\times(0,T)\), set
\begin{equation}
\label{eq:H_and_F_def}
F(\ell,x,t) \coloneqq  \int_0^\ell I(x+s,t)\,\rd s\, ,
\quad
H(\ell,x,t)\coloneqq \int_\ell^\infty\rho_h(m,x,t)\,\rd m \, ,
\qquad \ell>0.
\end{equation}
The following lemma shows that these terms cancel after integration
with respect to \(\ell\).

\begin{lem}
\label{lem:cancellation_of_annihilation_and_left_creation}
Assume~\eqref{eq:assumption-I} and~\eqref{eq:assumption-rhoh}. For a.e. $(x,t)\in \mathbb{R}\times(0,T)$,
$$
-\int_0^\infty F(\ell,x,t)\rho_h(\ell,x,t)\,\rd \ell
+ \int_0^\infty I(x+\ell,t)H(\ell,x,t)\,\rd \ell =0\,.
$$
\end{lem}

\begin{proof}
By~\eqref{eq:assumption-I} and~\eqref{eq:assumption-rhoh}, for a.e. \((x,t)\in\mathbb R\times(0,T)\),
\begin{align*}
\int_0^\infty\int_0^m
|I(x+\ell,t)|\rho_h(m,x,t)\,\rd\ell\,\rd m
\leq \|I\|_{L^\infty(\mathbb R\times(0,T))}
\int_0^\infty m\rho_h(m,x,t)\,\rd m
<\infty\,.
\end{align*}
Hence Fubini's theorem applies and gives
\begin{align*}
\int_0^\infty I(x+\ell,t)H(\ell,x,t)\,\rd \ell
&=\int_0^\infty \int_\ell^\infty I(x+\ell,t)\rho_h(m,x,t)\,\rd m\,\rd \ell \\
&=\int_0^\infty \int_0^m I(x+\ell,t)\,\rd \ell \rho_h(m,x,t)\,\rd m \\
&=\int_0^\infty F(m,x,t)\rho_h(m,x,t)\,\rd m\,,
\end{align*}
which proves the assertion.
\end{proof}

We now consider the right-subhole creation term and, for a.e. \((x,t)\in\mathbb R\times(0,T)\), define
\[
\mathcal B(x,t)\coloneqq 
\int_0^\infty\int_0^\infty
\rho_h(\ell+s,x-s,t)\,\rd s\,\rd\ell\,,
\]
which is finite by~\eqref{eq:assumption-rhoh}. The following lemma identifies the integrated right-subhole creation term
with the unreplicated fraction.

\begin{lem}
\label{lem:coverage-identity}
Assume~\eqref{eq:assumption-rhoh}. For a.e. $(x,t)\in \mathbb{R}\times(0,T)$,
$$
\mathcal{B}(x,t)
= \int_0^\infty \int_{\mathbb{R}} \mathbf{1}_{[y,y+r)}(x)\rho_h(r,y,t)\,\rd y\,\rd r\,.
$$
In particular,
\[
\mathcal B(x,t)=s(x,t)=1-f(x,t)
\]
for a.e. \((x,t)\in\mathbb R\times(0,T)\).
\end{lem}

\begin{proof}
For a.e. \((x,t)\in\mathbb R\times(0,T)\), the condition
\(x\in[y,y+r)\) means that \(y\in(x-r,x]\), and the change of variables \(y=x-s\) gives \(s\in[0,r)\). Since \(\rho_h\geq0\), Tonelli's theorem
and the change of variables \(r=\ell+s\) give
\begin{align*}
\int_0^\infty\int_{\mathbb R}
\mathbf{1}_{[y,y+r)}(x)\rho_h(r,y,t)\,\rd y\,\rd r &=
\int_0^\infty\int_0^r
\rho_h(r,x-s,t)\,\rd s\,\rd r\\
&=
\int_0^\infty\int_0^\infty
\rho_h(\ell+s,x-s,t)\,\rd s\,\rd\ell \\
& =\mathcal B(x,t)\,.
\end{align*}
Finally, the choice
\(\psi(r,y)=\mathbf{1}_{[y,y+r)}(x)\)
in~\eqref{eq:hole-density-definition}, together with the disjointness of
the holes, gives the second assertion.
\end{proof}

The following lemma relates \(\partial_xs\) to
\(\rho_+-\rho_-\).

\begin{lem}
\label{lem:geometric-identity}
Assume~\eqref{eq:assumption-rhoh}. Then, for a.e. \(t\in(0,T)\), we have
\begin{equation}
\label{eq:geometric_identity}
\partial_xs(\cdot,t)
=
\rho_+(\cdot,t)-\rho_-(\cdot,t)
\qquad\text{in }\mathcal D'(\mathbb R)\,.
\end{equation}
\end{lem}

\begin{proof}
Fix \(\varphi\in C_c^\infty(\mathbb R)\). For a.e. \(t\in(0,T)\), Lemma~\ref{lem:coverage-identity}, \eqref{eq:assumption-rhoh}, and Fubini's theorem give

\begin{align*}
\left\langle\partial_xs(\cdot,t),\varphi\right\rangle &=
-\int_{\mathbb R}s(x,t)\varphi'(x)\,\rd x\\
&= -\int_0^\infty\int_{\mathbb R}
\rho_h(r,y,t)
\left(\int_y^{y+r}\varphi'(x)\,\rd x\right)
\, \rd y\,\rd r\\
&=
\int_0^\infty\int_{\mathbb R}
\rho_h(r,y,t)
\bigl(\varphi(y)-\varphi(y+r)\bigr)
\,\rd y\,\rd r\\
&=
\int_{\mathbb R}
\bigl(\rho_+(x,t)-\rho_-(x,t)\bigr)
\varphi(x)\,\rd x\,.
\end{align*}
\end{proof}

The following lemma gives the fork-density equations obtained from
\eqref{eq:joint-hole} by integrating over the hole size according to~\eqref{eq:rho_h_integral}.

\begin{lem}
\label{lem:size-collapsed-fork-equations}
Assume~\eqref{eq:assumptions}. Then
\begin{equation}
\label{eq:rho-plus-collapsed-exact}
\partial_t\rho_+(x,t)
+\partial_x\bigl(v(x,t)\rho_+(x,t)\bigr)
= I(x,t)s(x,t)-J_h(x,t)
\end{equation}
and
\begin{equation}
\label{eq:rho-minus-collapsed-exact}
\partial_t\rho_-(x,t)
-\partial_x\!\bigl(v(x,t)\rho_-(x,t)\bigr)
=
I(x,t)s(x,t)-J_h(x,t)
\end{equation}
in \(\mathcal D'(\mathbb R\times(0,T))\). Inserting
\eqref{eq:mean-field-Jh} into
\eqref{eq:rho-plus-collapsed-exact}
and~\eqref{eq:rho-minus-collapsed-exact}, we obtain
\begin{equation}
\label{eq:rho-plus-collapsed}
\partial_t\rho_+(x,t)
+\partial_x\bigl(v(x,t)\rho_+(x,t)\bigr)
= I(x,t)s(x,t) -\frac{2v(x,t)}{s(x,t)} \rho_+(x,t)\rho_-(x,t)
\end{equation}
and
\begin{equation}
\label{eq:rho-minus-collapsed}
\partial_t\rho_-(x,t)
-\partial_x\bigl(v(x,t)\rho_-(x,t)\bigr)
= I(x,t)s(x,t)
-\frac{2v(x,t)}{s(x,t)}
\rho_+(x,t)\rho_-(x,t)
\end{equation}
in \(\mathcal D'(\mathbb R\times(0,T))\).
\end{lem}

\begin{proof}
By the definition of \(w\), \eqref{eq:assumption-v}, and
\eqref{eq:assumption-rhoh}, the functions \(\rho_h\), \(v\rho_h\), and
\(w\rho_h\) belong to
\(L^1\bigl((0,\infty);L^1_{\mathrm{loc}}(\mathbb R\times(0,T))\bigr)\).
Let \(\varphi\in C_c^\infty(\mathbb R\times(0,T))\), choose
\(\chi\in C_c^\infty([0,\infty))\) and a non-decreasing function
\(\eta\in C^\infty([0,\infty))\), both with values in \([0,1]\),
such that
\[
\chi(r)=1 \quad\text{for }0\leq r\leq1\,,
\qquad
\chi(r)=0 \quad\text{for }r\geq2\,,
\]
and 
\[
\eta(r)=0 \quad\text{for }0\leq r\leq1\,,
\qquad
\eta(r)=1 \quad\text{for }r\geq2\,.
\]
For \(R>1\) and \(0<\varepsilon<R/2\), we set
\(\chi_R(\ell):=\chi(\ell/R)\) and
\(\eta_\varepsilon(\ell):=\eta(\ell/\varepsilon)\).
Then 
\[
\chi_R(\ell)=1 \quad \text{for }0\leq\ell\leq R\,, 
\quad 
\chi_R(\ell)=0 \quad \text{for } \ell\geq2R\,,
\quad 
\operatorname{supp}\chi_R'\subset[R,2R]\,,
\]
with $|\chi_R'(\ell)|\leq
\frac{1}{R}\|\chi'\|_{L^\infty(0,\infty)}$, and 
\[
\eta_\varepsilon(\ell)=0 \quad \text{for }0\leq\ell\leq \varepsilon\,, 
\quad 
\eta_\varepsilon(\ell)=1 \quad \text{for } \ell\geq 2 \varepsilon\,,
\quad 
\operatorname{supp}\eta_\varepsilon'\subset[\varepsilon,2\varepsilon]\,,
\]
with $|\eta_\varepsilon'(\ell)|\leq\frac{1}{\varepsilon}\|\eta'\|_{L^\infty(0,\infty)}$. It follows that
$\eta_\varepsilon(\ell)\chi_R(\ell)\varphi(y,t)$ and
$\eta_\varepsilon(\ell)\chi_R(\ell)\varphi(y+\ell,t)$ lie in
$C_c^\infty\bigl((0,\infty)\times\mathbb R\times(0,T)\bigr)$.
For a.e. \((y,t)\in\mathbb R\times(0,T)\),
\eqref{eq:additional_assumptions},
\eqref{eq:J_h-definition}, and the fundamental theorem of calculus give
\begin{align*}
(w\rho_h)(\ell,y,t)=
J_h(y,t)
+ \int_0^\ell
\partial_r\bigl((w\rho_h)(r,y,t)\bigr)\,\rd r\, .
\end{align*}
Moreover, since
\begin{equation*}
\int_\varepsilon^{2\varepsilon}
\eta_\varepsilon'(\ell)\,\rd\ell
=
\eta_\varepsilon(2\varepsilon)
-\eta_\varepsilon(\varepsilon)
=1\, ,
\end{equation*}
and, for \(2\varepsilon<R\),
\(\chi_R(\ell)=1\) for every
\(\ell\in\operatorname{supp}\eta_\varepsilon'\), we obtain
\begin{align*}
&\int_0^\infty
(w\rho_h)(\ell,y,t)
\eta_\varepsilon'(\ell)\chi_R(\ell)
\varphi(y+\ell,t)\,\rd\ell
-J_h(y,t)\varphi(y,t)\\
&\; =
\int_\varepsilon^{2\varepsilon}
\eta_\varepsilon'(\ell)
\Big(\int_0^\ell\partial_r\bigl((w\rho_h)(r,y,t)\bigr)\,\rd r\Big)
\varphi(y+\ell,t)\,\rd\ell 
+ J_h(y,t)
\int_\varepsilon^{2\varepsilon}
\eta_\varepsilon'(\ell)
\bigl(
\varphi(y+\ell,t)-\varphi(y,t)
\bigr)\,\rd\ell.
\end{align*}
Therefore,
\begin{align}
\label{eq:lcse}
&\left|
\int_0^\infty
(w\rho_h)(\ell,y,t)
\eta_\varepsilon'(\ell)\chi_R(\ell)
\varphi(y+\ell,t)\,\rd\ell
-J_h(y,t)\varphi(y,t)
\right| \nonumber\\
&\; \leq
\|\eta'\|_{L^\infty(0,\infty)}
\|\varphi\|_{L^\infty(\mathbb R\times(0,T))}
\int_0^{2\varepsilon}
\left| \partial_r\bigl((w\rho_h)(r,y,t)\bigr)
\right|\,\rd r + |J_h(y,t)|
\sup_{0\leq r\leq2\varepsilon}
|\varphi(y+r,t)-\varphi(y,t)|
\nonumber \\
& \; \longrightarrow0 \, 
\end{align}
as \(\varepsilon\downarrow0\).  
Moreover, replacing
\(\varphi(y+\ell,t)\) by \(\varphi(y,t)\) in the preceding calculation
gives
\begin{align}
\label{eq:lcue}
&\left|
\int_0^\infty
(w\rho_h)(\ell,y,t)
\eta_\varepsilon'(\ell)\chi_R(\ell)
\varphi(y,t)\,\rd\ell
-J_h(y,t)\varphi(y,t)
\right| \nonumber\\
&\quad\leq
\|\eta'\|_{L^\infty(0,\infty)}
\|\varphi\|_{L^\infty(\mathbb R\times(0,T))}
\int_0^{2\varepsilon}
\left|
\partial_r\bigl((w\rho_h)(r,y,t)\bigr)
\right|\,\rd r
\longrightarrow0 \quad \text{ as } \varepsilon\downarrow0\,.
\end{align}
Furthermore,
\eqref{eq:size-flux-identity} and
\eqref{eq:additional_assumptions} give, for a.e.
\((y,t)\in\mathbb R\times(0,T)\),
\begin{align*}
|J_h(y,t)|
\leq
\int_0^\infty
\left|
\partial_r\bigl((w\rho_h)(r,y,t)\bigr)
\right|\,\rd r
\leq
\|(w\rho_h)(\cdot,y,t)\|_{W^{1,1}(0,\infty)},
\end{align*}
and consequently, for every compact set
\(K\subset\mathbb R\times(0,T)\),
\begin{align*}
\int_K|J_h(y,t)|\,\rd y\,\rd t
&\leq
\int_K
\|(w\rho_h)(\cdot,y,t)\|_{W^{1,1}(0,\infty)}
\,\rd y\,\rd t
<\infty\,,
\end{align*}
which proves that
\(J_h\in L^1_{\mathrm{loc}}(\mathbb R\times(0,T))\). 
For every fixed
\(R>1\), the expressions on the left-hand sides
of~\eqref{eq:lcse}
and~\eqref{eq:lcue} vanish outside a compact subset of \(\mathbb R\times(0,T)\) that is independent of
\(\varepsilon\) and, by these estimates and the preceding bound for
\(J_h\), satisfy, for a.e.
\((y,t)\in\mathbb R\times(0,T)\),
\begin{align*}
\Big|
\int_0^\infty
(w\rho_h)(\ell,y,t)
&\eta_\varepsilon'(\ell)\chi_R(\ell)
\varphi(y+\ell,t)\,\rd\ell
-J_h(y,t)\varphi(y,t)
\Big|
\\
&\quad\leq
\Bigl(
\|\eta'\|_{L^\infty(0,\infty)}+2
\Bigr)
\|\varphi\|_{L^\infty(\mathbb R\times(0,T))}
\|(w\rho_h)(\cdot,y,t)\|_{W^{1,1}(0,\infty)}
\end{align*}
and
\begin{align*}
\Big|
\int_0^\infty
(w\rho_h)(\ell,y,t)
&\eta_\varepsilon'(\ell)\chi_R(\ell)
\varphi(y,t)\,\rd\ell
-J_h(y,t)\varphi(y,t)
\Big|
\\
&\quad\leq
\|\eta'\|_{L^\infty(0,\infty)}
\|\varphi\|_{L^\infty(\mathbb R\times(0,T))}
\|(w\rho_h)(\cdot,y,t)\|_{W^{1,1}(0,\infty)}\,.
\end{align*}
Since both upper bounds are integrable with respect to \((y,t)\) on
this compact subset by~\eqref{eq:additional_assumptions}, it follows
from Lebesgue's dominated convergence theorem that
\begin{align}
\label{eq:lcrm}
\lim_{\varepsilon\downarrow0}
\int_0^T\int_{\mathbb R}\int_0^\infty
(w\rho_h)(\ell,y,t)
\eta_\varepsilon'(\ell)\chi_R(\ell)
\varphi(y+\ell,t)
\,\rd\ell\,\rd y\,\rd t
=
\int_0^T\int_{\mathbb R}
J_h(y,t)\varphi(y,t)\,\rd y\,\rd t.
\end{align}
and
\begin{align}
\label{eq:lcrp}
\lim_{\varepsilon\downarrow0}
\int_0^T\int_{\mathbb R}\int_0^\infty
(w\rho_h)(\ell,y,t)
\eta_\varepsilon'(\ell)\chi_R(\ell)
\varphi(y,t)
\,\rd\ell\,\rd y\,\rd t
=
\int_0^T\int_{\mathbb R}
J_h(y,t)\varphi(y,t)\,\rd y\,\rd t.
\end{align}
In addition, since \(1+\ell\geq1+R\) for
\(\ell\in[R,2R]\), it follows from the definition of \(w\), the
properties of \(\chi_R\), \eqref{eq:assumption-v}, and
\eqref{eq:assumption-rhoh} that
\begin{align}
\label{eq:ucl}
&\Big|
\int_0^T\int_R^{2R}\int_{\mathbb R}
\chi_R'(\ell)w(\ell,y,t)\rho_h(\ell,y,t)
\varphi(y+\ell,t)
\,\rd y\,\rd\ell\,\rd t
\Big|\nonumber\\
&\quad\leq
\frac{
2}{R(1+R)}\|\chi'\|_{L^\infty(0,\infty)}
\|v\|_{L^\infty(\mathbb R\times(0,T))}
\int_0^T
\|\varphi(\cdot,t)\|_{L^1(\mathbb R)}
\int_R^{2R}
(1+\ell)
\|\rho_h(\ell,\cdot,t)\|_{L^\infty(\mathbb R)}
\,\rd\ell\,\rd t \nonumber\\
&\quad\leq
\frac{2}{R(1+R)}\|\chi'\|_{L^\infty(0,\infty)}
\|v\|_{L^\infty(\mathbb R\times(0,T))}
\|\varphi\|_{L^1(\mathbb R\times(0,T))}
\nonumber\\
&\qquad\times
\operatorname*{ess\,sup}_{0<t<T}
\int_0^\infty
(1+\ell)
\|\rho_h(\ell,\cdot,t)\|_{L^\infty(\mathbb R)}
\,\rd\ell
\longrightarrow0
\end{align}
as \(R\to\infty\), and the same estimate holds with
\(\varphi(y+\ell,t)\) replaced by \(\varphi(y,t)\).
For \(0\leq q\leq m\),
\begin{align*}
\int_{\mathbb R}
\rho_h(m,y,t)
\Bigl(|\varphi(y,t)|+|\varphi(y+q,t)|+|\varphi(y+m,t)|
\Bigr)\,\rd y \leq
3\|\rho_h(m,\cdot,t)\|_{L^\infty(\mathbb R)}
\|\varphi(\cdot,t)\|_{L^1(\mathbb R)}\,,
\end{align*}
and therefore~\eqref{eq:assumption-I}
and~\eqref{eq:assumption-rhoh} imply
\begin{align}
\label{eq:fii}
&\int_0^T\int_0^\infty\int_{\mathbb R}
\rho_h(m,y,t)
\int_0^m|I(y+q,t)|
\Bigl(
|\varphi(y,t)|
+|\varphi(y+q,t)|
+|\varphi(y+m,t)|
\Bigr)
\,\rd q\,\rd y\,\rd m\,\rd t \nonumber\\
&\quad \leq
3\|I\|_{L^\infty(\mathbb R\times(0,T))}
\|\varphi\|_{L^1(\mathbb R\times(0,T))}
\operatorname*{ess\,sup}_{0<t<T}
\int_0^\infty
m\|\rho_h(m,\cdot,t)\|_{L^\infty(\mathbb R)}
\,\rd m
<\infty,
\end{align}
so that Fubini's theorem applies to the three initiation terms on the
right-hand side of~\eqref{eq:joint-hole} and, by Lebesgue's
dominated convergence theorem, we may first let
\(\varepsilon\downarrow0\) for each fixed \(R>1\) and then let
\(R\to\infty\) in these terms.
To prove~\eqref{eq:rho-plus-collapsed-exact}, we first
test~\eqref{eq:joint-hole} with
\(\eta_\varepsilon(\ell)\chi_R(\ell)\varphi(y,t)\), which gives
\begin{align}
\label{eq:rpcwf}
&-\int_0^T\int_{\mathbb R}\int_0^\infty
\rho_h(\ell,y,t)\eta_\varepsilon(\ell)\chi_R(\ell)
\partial_t\varphi(y,t)
\,\rd\ell\,\rd y\,\rd t \nonumber \\
&-
\int_0^T\int_{\mathbb R}\int_0^\infty
v(y,t)\rho_h(\ell,y,t)\eta_\varepsilon(\ell)\chi_R(\ell)
\partial_x\varphi(y,t)
\,\rd\ell\,\rd y\,\rd t \nonumber\\
&+
\int_0^T\int_{\mathbb R}\int_0^\infty
(w\rho_h)(\ell,y,t)
\Bigl(
\eta_\varepsilon'(\ell)\chi_R(\ell)
+\eta_\varepsilon(\ell)\chi_R'(\ell)
\Bigr)
\varphi(y,t)
\,\rd\ell\,\rd y\,\rd t \nonumber \\
&\quad =
\int_0^T\int_{\mathbb R}\int_0^\infty
\Big[-F(\ell,y,t)\rho_h(\ell,y,t)
+I(y+\ell,t)H(\ell,y,t) \nonumber\\
&\hspace{3cm}
+I(y,t)\int_0^\infty
\rho_h(\ell+q,y-q,t)\,\rd q
\Big]
\eta_\varepsilon(\ell)\chi_R(\ell)\varphi(y,t)
\,\rd\ell\,\rd y\,\rd t\,.
\end{align}
It follows from~\eqref{eq:rho_h_integral},
\eqref{eq:assumption-v}, \eqref{eq:assumption-rhoh}, and Lebesgue's
dominated convergence theorem that
\begin{align*}
\lim_{R\to\infty}\lim_{\varepsilon\downarrow0}
\int_0^T\int_{\mathbb R}\int_0^\infty
\rho_h(\ell,y,t)\eta_\varepsilon(\ell)\chi_R(\ell)
\partial_t\varphi(y,t)
\,\rd\ell\,\rd y\,\rd t=
\int_0^T\int_{\mathbb R}
\rho_+(x,t)\partial_t\varphi(x,t)\,\rd x\,\rd t
\end{align*}
and
\begin{align*}
&\lim_{R\to\infty}\lim_{\varepsilon\downarrow0}
\int_0^T\int_{\mathbb R}\int_0^\infty
v(y,t)\rho_h(\ell,y,t)
\eta_\varepsilon(\ell)\chi_R(\ell)
\partial_x\varphi(y,t)
\,\rd\ell\,\rd y\,\rd t\\
&\qquad  =
\int_0^T\int_{\mathbb R}
v(x,t)\rho_+(x,t)\partial_x\varphi(x,t)\,\rd x\,\rd t,
\end{align*}
while~\eqref{eq:lcrp}
and~\eqref{eq:ucl} give
\begin{align*}
&\lim_{R\to\infty}\lim_{\varepsilon\downarrow0}
\int_0^T\int_{\mathbb R}\int_0^\infty
(w\rho_h)(\ell,y,t)
\Bigl(
\eta_\varepsilon'(\ell)\chi_R(\ell)
+\eta_\varepsilon(\ell)\chi_R'(\ell)
\Bigr)
\varphi(y,t)
\,\rd\ell\,\rd y\,\rd t\\
&\qquad=
\int_0^T\int_{\mathbb R}
J_h(x,t)\varphi(x,t)\,\rd x\,\rd t\,.
\end{align*}
For the right-hand side
of~\eqref{eq:rpcwf}, it follows
from~\eqref{eq:fii}, Lebesgue's dominated
convergence theorem, and
Lemmas~\ref{lem:cancellation_of_annihilation_and_left_creation}
and~\ref{lem:coverage-identity} that
\begin{align*}
&\lim_{R\to\infty}\lim_{\varepsilon\downarrow0}
\int_0^T\int_{\mathbb R}\int_0^\infty
\Big[
-F(\ell,y,t)\rho_h(\ell,y,t)
+I(y+\ell,t)H(\ell,y,t)\\
&\hspace{12em}
+I(y,t)\int_0^\infty
\rho_h(\ell+q,y-q,t)\,\rd q
\Big]
\eta_\varepsilon(\ell)\chi_R(\ell)\varphi(y,t)
\,\rd\ell\,\rd y\,\rd t\\
&\quad=
\int_0^T\int_{\mathbb R}
\Big[
-\int_0^\infty
F(\ell,y,t)\rho_h(\ell,y,t)\,\rd\ell
+\int_0^\infty
I(y+\ell,t)H(\ell,y,t)\,\rd\ell\\
&\hspace{12em}
+I(y,t)\mathcal B(y,t)
\Big]\varphi(y,t)\,\rd y\,\rd t\\
&\quad=
\int_0^T\int_{\mathbb R}
I(x,t)s(x,t)\varphi(x,t)\,\rd x\,\rd t\,,
\end{align*}
and, substituting these four limits
into~\eqref{eq:rpcwf}, we infer
\begin{align*}
&-\int_0^T\int_{\mathbb R}
\rho_+(x,t)\partial_t\varphi(x,t)\,\rd x\,\rd t
-\int_0^T\int_{\mathbb R}
v(x,t)\rho_+(x,t)\partial_x\varphi(x,t)\,\rd x\,\rd t
+
\int_0^T\int_{\mathbb R}
J_h(x,t)\varphi(x,t)\,\rd x\,\rd t\\
& \qquad=
\int_0^T\int_{\mathbb R}
I(x,t)s(x,t)\varphi(x,t)\,\rd x\,\rd t \,,
\end{align*}
which proves~\eqref{eq:rho-plus-collapsed-exact}.
To prove~\eqref{eq:rho-minus-collapsed-exact}, we
test~\eqref{eq:joint-hole} with
\(\eta_\varepsilon(\ell)\chi_R(\ell)\varphi(y+\ell,t)\). Since
$
\partial_y\varphi(y+\ell,t) = \partial_\ell\varphi(y+\ell,t)
=\partial_x\varphi(y+\ell,t)
$
and $w(\ell,y,t)-v(y,t)=v(y+\ell,t)$,
the corresponding cutoff weak formulation is
\begin{align}
\label{eq:rmcwf}
&-\int_0^T\int_{\mathbb R}\int_0^\infty
\rho_h(\ell,y,t)\eta_\varepsilon(\ell)\chi_R(\ell)
\partial_t\varphi(y+\ell,t)
\,\rd\ell\,\rd y\,\rd t \nonumber \\
&+
\int_0^T\int_{\mathbb R}\int_0^\infty
v(y+\ell,t)\rho_h(\ell,y,t)
\eta_\varepsilon(\ell)\chi_R(\ell)
\partial_x\varphi(y+\ell,t)
\,\rd\ell\,\rd y\,\rd t \nonumber\\
&+
\int_0^T\int_{\mathbb R}\int_0^\infty
(w\rho_h)(\ell,y,t)
\Bigl(
\eta_\varepsilon'(\ell)\chi_R(\ell)
+\eta_\varepsilon(\ell)\chi_R'(\ell)
\Bigr)\varphi(y+\ell,t)
\,\rd\ell\,\rd y\,\rd t \nonumber\\
&=
\int_0^T\int_{\mathbb R}\int_0^\infty
\Big[
-F(\ell,y,t)\rho_h(\ell,y,t)
+I(y+\ell,t)H(\ell,y,t) \nonumber \\
&\hspace{2.5cm}
+I(y,t)\int_0^\infty
\rho_h(\ell+q,y-q,t)\,\rd q
\Big]
\eta_\varepsilon(\ell)\chi_R(\ell)
\varphi(y+\ell,t)
\,\rd\ell\,\rd y\,\rd t\,.
\end{align}
For the left-hand side
of~\eqref{eq:rmcwf}, it follows from the
change of variables \(x=y+\ell\), \eqref{eq:rho_h_integral},
\eqref{eq:assumption-v}, \eqref{eq:assumption-rhoh}, Lebesgue's
dominated convergence theorem,
\eqref{eq:lcrm}, and
\eqref{eq:ucl} that
\begin{align}
\label{eq:rmlhsl}
\lim_{R\to\infty}\lim_{\varepsilon\downarrow0}
\Big[
&-\int_0^T\int_{\mathbb R}\int_0^\infty
\rho_h(\ell,y,t)\eta_\varepsilon(\ell)\chi_R(\ell)
\partial_t\varphi(y+\ell,t)
\,\rd\ell\,\rd y\,\rd t \nonumber\\
&+ \int_0^T\int_{\mathbb R}\int_0^\infty
v(y+\ell,t)\rho_h(\ell,y,t)
\eta_\varepsilon(\ell)\chi_R(\ell)
\partial_x\varphi(y+\ell,t)
\,\rd\ell\,\rd y\,\rd t \nonumber \\
&+
\int_0^T\int_{\mathbb R}\int_0^\infty
(w\rho_h)(\ell,y,t)
\Big(
\eta_\varepsilon'(\ell)\chi_R(\ell)
+\eta_\varepsilon(\ell)\chi_R'(\ell)
\Big)\varphi(y+\ell,t)
\,\rd\ell\,\rd y\,\rd t
\Big] \nonumber \\
&\hspace{-1.5cm}=
-\int_0^T\int_{\mathbb R}
\rho_-(x,t)\partial_t\varphi(x,t)\,\rd x\,\rd t
+\int_0^T\int_{\mathbb R}
v(x,t)\rho_-(x,t)\partial_x\varphi(x,t)
\,\rd x\,\rd t \nonumber\\
&\hspace{-1.0cm}+
\int_0^T\int_{\mathbb R}
J_h(x,t)\varphi(x,t)\,\rd x\,\rd t \,.
\end{align}
For the right-hand side
of~\eqref{eq:rmcwf}, it follows
from~\eqref{eq:fii}, Lebesgue's dominated
convergence theorem, Fubini's theorem, and the changes of variables
\(\widetilde y=y-q\), \(m=\ell+q\), and then \(x=y+q\) that
\begin{align}
\label{eq:rmrl}
&\lim_{R\to\infty}\lim_{\varepsilon\downarrow0}
\int_0^T\int_{\mathbb R}\int_0^\infty
\Big[
-F(\ell,y,t)\rho_h(\ell,y,t)
+I(y+\ell,t)H(\ell,y,t)\nonumber\\
&\hspace{12em}
+I(y,t)\int_0^\infty
\rho_h(\ell+q,y-q,t)\,\rd q
\Big]
\eta_\varepsilon(\ell)\chi_R(\ell)
\varphi(y+\ell,t)
\,\rd\ell\,\rd y\,\rd t \nonumber \\
&\quad=
\int_0^T\int_{\mathbb R}\int_0^\infty
\Big[
-F(\ell,y,t)\rho_h(\ell,y,t)
+I(y+\ell,t)H(\ell,y,t) \nonumber \\
&\hspace{12em}
+I(y,t)\int_0^\infty
\rho_h(\ell+q,y-q,t)\,\rd q
\Big]
\varphi(y+\ell,t)
\,\rd\ell\,\rd y\,\rd t\nonumber \\
&\quad=
\int_0^T\int_{\mathbb R}\int_0^\infty
\rho_h(m,y,t)
\int_0^m I(y+q,t)
\Bigl(
-\varphi(y+m,t)
+\varphi(y+q,t)
+\varphi(y+m,t)
\Bigr)
\,\rd q\,\rd m\,\rd y\,\rd t \nonumber\\
&\quad=
\int_0^T\int_{\mathbb R}
I(x,t)\varphi(x,t)
\left(
\int_0^\infty\int_0^m
\rho_h(m,x-q,t)\,\rd q\,\rd m
\right)\rd x\,\rd t \nonumber\\
&\quad=
\int_0^T\int_{\mathbb R}
I(x,t)\mathcal B(x,t)\varphi(x,t)\,\rd x\,\rd t \nonumber\\
&\quad=
\int_0^T\int_{\mathbb R}
I(x,t)s(x,t)\varphi(x,t)\,\rd x\,\rd t\, .
\end{align}
Substituting~\eqref{eq:rmlhsl}
and~\eqref{eq:rmrl}
into~\eqref{eq:rmcwf}, we infer
\begin{align*}
&-\int_0^T\int_{\mathbb R}
\rho_-(x,t)\partial_t\varphi(x,t)\,\rd x\,\rd t
+\int_0^T\int_{\mathbb R}
v(x,t)\rho_-(x,t)\partial_x\varphi(x,t)\,\rd x\,\rd t\\
&\quad+
\int_0^T\int_{\mathbb R}
J_h(x,t)\varphi(x,t)\,\rd x\,\rd t
=
\int_0^T\int_{\mathbb R}
I(x,t)s(x,t)\varphi(x,t)\,\rd x\,\rd t,
\end{align*}
which proves~\eqref{eq:rho-minus-collapsed-exact}. 

Finally, inserting~\eqref{eq:mean-field-Jh}
into~\eqref{eq:rho-plus-collapsed-exact}
and~\eqref{eq:rho-minus-collapsed-exact} proves
\eqref{eq:rho-plus-collapsed}
and~\eqref{eq:rho-minus-collapsed}.
\end{proof}

Finally, we derive the evolution equation for the replicated fraction, $f(x,t)$.

\begin{lem}
\label{lem:replicated_frac_eq}
Assume~\eqref{eq:assumptions}. Assume additionally that, for every
$0<\tau_1<\tau_2<T$ and every bounded interval
$[a,b]\subset\mathbb R$,
\begin{equation}
\label{eq:2nd_moment_rho}
\int_{\tau_1}^{\tau_2}
\int_0^\infty
\int_{a-m}^{b}
\rho_h(m,y,t)
\left(
\int_0^m |I(y+q,t)|\,\rd q
\right)
\,\rd y\,\rd m\,\rd t
<\infty.
\end{equation}
Then
\begin{equation}
\label{eq:f-definition}
\partial_t f(x,t)
=
v(x,t)\bigl(\rho_+(x,t)+\rho_-(x,t)\bigr)
\qquad
\text{in }\mathcal D'(\mathbb R\times(0,T)).
\end{equation}
\end{lem}

\begin{proof}
Let
$\varphi\in C_c^\infty(\mathbb R\times(0,T))$, and choose
$a<b$ and $0<\tau_1<\tau_2<T$ such that
\[
\operatorname{supp}\varphi
\subset
[a,b]\times[\tau_1,\tau_2].
\]
For \((\ell,y,t)\in[0,\infty)\times\mathbb R\times(0,T)\), we define
\[
\Psi(\ell,y,t) \coloneqq \int_y^{y+\ell}\varphi(x,t)\,\rd x \, ,
\] 
and differentiation gives
\begin{align}
\label{eq:Phi-derivatives}
\partial_t\Psi(\ell,y,t)
&= \int_y^{y+\ell}\partial_t\varphi(x,t)\,\rd x\, , \nonumber\\
\partial_y\Psi(\ell,y,t)
&=
\varphi(y+\ell,t)-\varphi(y,t)\,, \nonumber \\
\partial_\ell\Psi(\ell,y,t)
&= \varphi(y+\ell,t)\,.
\end{align}
Moreover, Fubini's theorem and Lemma~\ref{lem:coverage-identity}
give
\begin{equation}
\label{eq:coverage-tested}
\int_0^T\int_{\mathbb R}
s(x,t)\varphi(x,t)\,\rd x\,\rd t
=
\int_0^T\int_{\mathbb R}\int_0^\infty
\rho_h(\ell,y,t)\Psi(\ell,y,t)
\,\rd\ell\,\rd y\,\rd t\,.
\end{equation}
Let $\eta_\varepsilon,\chi_R\in C^\infty([0,\infty))$, with values
in $[0,1]$, satisfy
\[
\eta_\varepsilon(\ell)=0
\quad\text{for }0\leq\ell\leq\varepsilon \, ,
\qquad
\eta_\varepsilon(\ell)=1
\quad\text{for }\ell\geq2\varepsilon\, ,
\]
and
\[
\chi_R(\ell)=1
\quad\text{for }0\leq\ell\leq R\, ,
\qquad
\chi_R(\ell)=0
\quad\text{for }\ell\geq2R\, ,
\]
where
\[
|\eta_\varepsilon'(\ell)|
\leq\frac{C}{\varepsilon}\, ,
\qquad
|\chi_R'(\ell)|
\leq\frac{C}{R}\, .
\]
Here $C>0$ is fixed independently of $\varepsilon$ and $R$. For $0<2\varepsilon<R$, the function
\(
\eta_\varepsilon(\ell)\chi_R(\ell)\Psi(\ell,y,t)
\)
belongs to
$C_c^\infty((0,\infty)\times\mathbb R\times(0,T))$.
Testing~\eqref{eq:joint-hole} with this function gives
\begin{align}
&-\int_0^T\int_{\mathbb R}\int_0^\infty
\rho_h(\ell,y,t)
\eta_\varepsilon(\ell)\chi_R(\ell)
\partial_t\Psi(\ell,y,t)
\,\rd\ell\,\rd y\,\rd t
\nonumber\\
&
-\int_0^T\int_{\mathbb R}\int_0^\infty
v(y,t)\rho_h(\ell,y,t)
\eta_\varepsilon(\ell)\chi_R(\ell)
\partial_y\Psi(\ell,y,t)
\,\rd\ell\,\rd y\,\rd t
\nonumber\\
&
+\int_0^T\int_{\mathbb R}\int_0^\infty
w(\ell,y,t)\rho_h(\ell,y,t)
\eta_\varepsilon(\ell)\chi_R(\ell)
\partial_\ell\Psi(\ell,y,t)
\,\rd\ell\,\rd y\,\rd t
\nonumber\\
&
+\int_0^T\int_{\mathbb R}\int_0^\infty
w(\ell,y,t)\rho_h(\ell,y,t)
\Bigl(
\eta_\varepsilon'(\ell)\chi_R(\ell)
+
\eta_\varepsilon(\ell)\chi_R'(\ell)
\Bigr)
\Psi(\ell,y,t)
\,\rd\ell\,\rd y\,\rd t
\nonumber\\
& \quad=
\int_0^T\int_{\mathbb R}\int_0^\infty
\Big[
-F(\ell,y,t)\rho_h(\ell,y,t)
+
I(y+\ell,t)H(\ell,y,t)
\nonumber\\
&\hspace{3cm}
+
I(y,t)
\int_0^\infty
\rho_h(\ell+q,y-q,t)\,\rd q
\Big]
\eta_\varepsilon(\ell)\chi_R(\ell)
\Psi(\ell,y,t)
\,\rd\ell\,\rd y\,\rd t.
\label{eq:f-cutoff-weak-form}
\end{align}
We first consider the two terms containing derivatives of the cutoff
functions on the left-hand side of~\eqref{eq:f-cutoff-weak-form}.
Since
\[
|\Psi(\ell,y,t)|
\leq
\ell\|\varphi\|_{L^\infty(\mathbb R\times(0,T))}
\]
and
\[
\operatorname{supp}_y\Psi(\ell,\cdot,t)
\subset[a-\ell,b],
\]
we have, for $0<\varepsilon<\min\{1,R/2\}$,
\begin{align*}
&\left|
\int_0^T\int_{\mathbb R}\int_0^\infty
w(\ell,y,t)\rho_h(\ell,y,t)
\eta_\varepsilon'(\ell)\chi_R(\ell)
\Psi(\ell,y,t)
\,\rd\ell\,\rd y\,\rd t
\right|
\\
&\quad\leq
\frac{
2C\|v\|_{L^\infty(\mathbb R\times(0,T))}
\|\varphi\|_{L^\infty(\mathbb R\times(0,T))}
}{\varepsilon}
\int_{\tau_1}^{\tau_2}
\int_0^{2\varepsilon}
\ell(\ell+b-a)
\|\rho_h(\ell,\cdot,t)\|_{L^\infty(\mathbb R)}
\,\rd\ell\,\rd t
\\
&\quad\leq
4C(b-a+2)
\|v\|_{L^\infty(\mathbb R\times(0,T))}
\|\varphi\|_{L^\infty(\mathbb R\times(0,T))}
\\
&\qquad\qquad\times
\int_{\tau_1}^{\tau_2}
\int_0^{2\varepsilon}
\|\rho_h(\ell,\cdot,t)\|_{L^\infty(\mathbb R)}
\,\rd\ell\,\rd t
\longrightarrow0
\end{align*}
as $\varepsilon\downarrow0$, for every fixed $R>1$. Indeed,
\eqref{eq:assumption-rhoh} gives
\[
\int_{\tau_1}^{\tau_2}
\int_0^\infty
\|\rho_h(\ell,\cdot,t)\|_{L^\infty(\mathbb R)}
\,\rd\ell\,\rd t
<\infty,
\]
so the last convergence follows from Lebesgue's dominated convergence
theorem.

For the upper cutoff term, we use
\[
|\Psi(\ell,y,t)|
\leq
\|\varphi(\cdot,t)\|_{L^1(\mathbb R)}
\]
and
\[
\operatorname{supp}_y\Psi(\ell,\cdot,t)
\subset[a-\ell,b].
\]
Since
$\operatorname{supp}\chi_R'\subset[R,2R]$, we obtain, for $R>1$,
\begin{align*}
&\left|
\int_0^T\int_{\mathbb R}\int_0^\infty
w(\ell,y,t)\rho_h(\ell,y,t)
\eta_\varepsilon(\ell)\chi_R'(\ell)
\Psi(\ell,y,t)
\,\rd\ell\,\rd y\,\rd t
\right|
\\
&\quad\leq
\frac{
2C\|v\|_{L^\infty(\mathbb R\times(0,T))}
}{R}
\int_{\tau_1}^{\tau_2}
\|\varphi(\cdot,t)\|_{L^1(\mathbb R)}
\int_R^{2R}
(\ell+b-a)
\|\rho_h(\ell,\cdot,t)\|_{L^\infty(\mathbb R)}
\,\rd\ell\,\rd t
\\
&\quad\leq
\frac{
2C(b-a+2)
\|v\|_{L^\infty(\mathbb R\times(0,T))}
}{1+R}
\|\varphi\|_{L^1(\mathbb R\times(0,T))}
\\
&\qquad\qquad\times
\operatorname*{ess\,sup}_{0<t<T}
\int_0^\infty
(1+\ell)
\|\rho_h(\ell,\cdot,t)\|_{L^\infty(\mathbb R)}
\,\rd\ell
\longrightarrow0
\end{align*}
as $R\to\infty$, uniformly in $\varepsilon$. We next consider the first three terms on the left-hand side
of~\eqref{eq:f-cutoff-weak-form}. For a.e.\
$(\ell,t)\in(0,\infty)\times(0,T)$, Fubini's theorem and
\eqref{eq:Phi-derivatives} give
\begin{align*}
\int_{\mathbb R}
\rho_h(\ell,y,t)
|\partial_t\Psi(\ell,y,t)|
\,\rd y
&\leq
\ell
\|\rho_h(\ell,\cdot,t)\|_{L^\infty(\mathbb R)}
\|\partial_t\varphi(\cdot,t)\|_{L^1(\mathbb R)},\\
\int_{\mathbb R}
|v(y,t)|\rho_h(\ell,y,t)
|\partial_y\Psi(\ell,y,t)|
\,\rd y
&\leq
2\|v\|_{L^\infty(\mathbb R\times(0,T))}
\|\rho_h(\ell,\cdot,t)\|_{L^\infty(\mathbb R)}
\|\varphi(\cdot,t)\|_{L^1(\mathbb R)},\\
\int_{\mathbb R}
|w(\ell,y,t)|\rho_h(\ell,y,t)
|\partial_\ell\Psi(\ell,y,t)|
\,\rd y
&\leq
2\|v\|_{L^\infty(\mathbb R\times(0,T))}
\|\rho_h(\ell,\cdot,t)\|_{L^\infty(\mathbb R)}
\|\varphi(\cdot,t)\|_{L^1(\mathbb R)}.
\end{align*}
The right-hand sides are integrable with respect to
$(\ell,t)\in(0,\infty)\times(0,T)$ by
\eqref{eq:assumption-rhoh}. Therefore, Lebesgue's dominated convergence
theorem gives
\begin{align*}
&\lim_{R\to\infty}\lim_{\varepsilon\downarrow0}
\int_0^T\int_{\mathbb R}\int_0^\infty
\rho_h(\ell,y,t)
\eta_\varepsilon(\ell)\chi_R(\ell)
\partial_t\Psi(\ell,y,t)
\,\rd\ell\,\rd y\,\rd t
\\
&\qquad=
\int_0^T\int_{\mathbb R}\int_0^\infty
\rho_h(\ell,y,t)
\partial_t\Psi(\ell,y,t)
\,\rd\ell\,\rd y\,\rd t,
\end{align*}
\begin{align*}
&\lim_{R\to\infty}\lim_{\varepsilon\downarrow0}
\int_0^T\int_{\mathbb R}\int_0^\infty
v(y,t)\rho_h(\ell,y,t)
\eta_\varepsilon(\ell)\chi_R(\ell)
\partial_y\Psi(\ell,y,t)
\,\rd\ell\,\rd y\,\rd t
\\
&\qquad=
\int_0^T\int_{\mathbb R}\int_0^\infty
v(y,t)\rho_h(\ell,y,t)
\partial_y\Psi(\ell,y,t)
\,\rd\ell\,\rd y\,\rd t,
\end{align*}
and
\begin{align*}
&\lim_{R\to\infty}\lim_{\varepsilon\downarrow0}
\int_0^T\int_{\mathbb R}\int_0^\infty
w(\ell,y,t)\rho_h(\ell,y,t)
\eta_\varepsilon(\ell)\chi_R(\ell)
\partial_\ell\Psi(\ell,y,t)
\,\rd\ell\,\rd y\,\rd t
\\
&\qquad=
\int_0^T\int_{\mathbb R}\int_0^\infty
w(\ell,y,t)\rho_h(\ell,y,t)
\partial_\ell\Psi(\ell,y,t)
\,\rd\ell\,\rd y\,\rd t.
\end{align*}
We next consider the three initiation terms on the right-hand side
of~\eqref{eq:f-cutoff-weak-form}. For $0\leq q\leq m$, we have
\begin{align*}
&|\Psi(q,y,t)|
+
|\Psi(m-q,y+q,t)|
+
|\Psi(m,y,t)|
\\
&\qquad\leq
2\|\varphi(\cdot,t)\|_{L^1(\mathbb R)}
\mathbf 1_{[a-m,b]}(y).
\end{align*}
Consequently,
\begin{align*}
&\int_{\tau_1}^{\tau_2}
\int_0^\infty
\int_{\mathbb R}
\rho_h(m,y,t)
\int_0^m |I(y+q,t)|
\Bigl(
|\Psi(q,y,t)|
\\
&\hspace{6cm}
+
|\Psi(m-q,y+q,t)|
+
|\Psi(m,y,t)|
\Bigr)
\,\rd q\,\rd y\,\rd m\,\rd t
\\
&\quad\leq
2\sup_{\tau_1\leq t\leq\tau_2}
\|\varphi(\cdot,t)\|_{L^1(\mathbb R)}
\int_{\tau_1}^{\tau_2}
\int_0^\infty
\int_{a-m}^{b}
\rho_h(m,y,t)
\left(
\int_0^m|I(y+q,t)|\,\rd q
\right)
\,\rd y\,\rd m\,\rd t
\\
&\quad<\infty
\end{align*}
by~\eqref{eq:2nd_moment_rho}. Thus Fubini's theorem, the changes of
variables below, and Lebesgue's dominated convergence theorem are
applicable.

After first letting $\varepsilon\downarrow0$ and then letting
$R\to\infty$, the annihilation term on the right-hand side
of~\eqref{eq:f-cutoff-weak-form} becomes
\begin{align*}
&-\int_{\tau_1}^{\tau_2}
\int_{\mathbb R}\int_0^\infty
F(\ell,y,t)\rho_h(\ell,y,t)
\Psi(\ell,y,t)
\,\rd\ell\,\rd y\,\rd t
\\
&\quad=
-\int_{\tau_1}^{\tau_2}
\int_{\mathbb R}\int_0^\infty
\rho_h(m,y,t)
\int_0^m
I(y+q,t)\Psi(m,y,t)
\,\rd q\,\rd m\,\rd y\,\rd t.
\end{align*}
Using the definition of $H$ from~\eqref{eq:H_and_F_def} and then changing the order of integration,
the left-subhole creation term becomes
\begin{align*}
&\int_{\tau_1}^{\tau_2}
\int_{\mathbb R}\int_0^\infty
I(y+\ell,t)H(\ell,y,t)
\Psi(\ell,y,t)
\,\rd\ell\,\rd y\,\rd t
\\
&\quad=
\int_{\tau_1}^{\tau_2}
\int_{\mathbb R}\int_0^\infty
\rho_h(m,y,t)
\int_0^m
I(y+q,t)\Psi(q,y,t)
\,\rd q\,\rd m\,\rd y\,\rd t.
\end{align*}
Finally, in the right-subhole creation term, the changes of variables
\[
\widetilde y=y-q,
\qquad
m=\ell+q
\]
give, after renaming $\widetilde y$ as $y$,
\begin{align*}
&\int_{\tau_1}^{\tau_2}
\int_{\mathbb R}\int_0^\infty
I(y,t)
\left(
\int_0^\infty
\rho_h(\ell+q,y-q,t)\,\rd q
\right)
\Psi(\ell,y,t)
\,\rd\ell\,\rd y\,\rd t
\\
&\quad=
\int_{\tau_1}^{\tau_2}
\int_{\mathbb R}\int_0^\infty
\rho_h(m,y,t)
\int_0^m
I(y+q,t)
\Psi(m-q,y+q,t)
\,\rd q\,\rd m\,\rd y\,\rd t.
\end{align*}
It follows that the limit of the entire right-hand side
of~\eqref{eq:f-cutoff-weak-form} is
\begin{align*}
&\int_{\tau_1}^{\tau_2}
\int_{\mathbb R}\int_0^\infty
\rho_h(m,y,t)
\int_0^m
I(y+q,t)
\Bigl[
\Psi(q,y,t)
\\
&\hspace{5cm}
+
\Psi(m-q,y+q,t)
-
\Psi(m,y,t)
\Bigr]
\,\rd q\,\rd m\,\rd y\,\rd t.
\end{align*}
Since
\[
\Psi(q,y,t)
+
\Psi(m-q,y+q,t)
=
\Psi(m,y,t),
\qquad
0\leq q\leq m,
\]
this limit is zero. Combining the preceding limits in~\eqref{eq:f-cutoff-weak-form}
and using~\eqref{eq:Phi-derivatives}, we obtain
\begin{align*}
0
&=
-\int_0^T\int_{\mathbb R}\int_0^\infty
\rho_h(\ell,y,t)
\partial_t\Psi(\ell,y,t)
\,\rd\ell\,\rd y\,\rd t
\\
&\quad
-\int_0^T\int_{\mathbb R}\int_0^\infty
v(y,t)\rho_h(\ell,y,t)
\bigl[
\varphi(y+\ell,t)-\varphi(y,t)
\bigr]
\,\rd\ell\,\rd y\,\rd t
\\
&\quad
+\int_0^T\int_{\mathbb R}\int_0^\infty
\bigl[v(y,t)+v(y+\ell,t)\bigr]
\rho_h(\ell,y,t)\varphi(y+\ell,t)
\,\rd\ell\,\rd y\,\rd t
\\
&=
-\int_0^T\int_{\mathbb R}
s(x,t)\partial_t\varphi(x,t)\,\rd x\,\rd t
\\
&\quad
+\int_0^T\int_{\mathbb R}
v(x,t)
\bigl(\rho_+(x,t)+\rho_-(x,t)\bigr)
\varphi(x,t)
\,\rd x\,\rd t.
\end{align*}
In the last equality, we used~\eqref{eq:coverage-tested},
the definitions of $\rho_\pm$, and the change of variables
$x=y+\ell$ in the term corresponding to $\rho_-$. Thus
\[
\partial_t s(x,t)
=
-v(x,t)(\rho_+(x,t)+\rho_-(x,t))
\qquad
\text{in }\mathcal D'(\mathbb R\times(0,T)).
\]
Since $f(x,t)=1-s(x,t)$, it follows that
\[
\partial_t f(x,t)
=
v(x,t)(\rho_+(x,t)+\rho_-(x,t))
\qquad
\text{in }\mathcal D'(\mathbb R\times(0,T)).
\]
\end{proof}

Lemma~\ref{lem:replicated_frac_eq} expresses the fact that each fork advances the replicated region at speed $v(x,t)$. Equations~\eqref{eq:rho-plus-collapsed},~\eqref{eq:rho-minus-collapsed}, and~\eqref{eq:f-definition} give the reduced replication-fork system analysed in Section~\ref{Sec3}.

The compatibility relation inherited from the hole-density description is
preserved by the reduced dynamics as shown below.

\begin{lem}
\label{lem:compatibility-propagation}
Let $(\rho_+,\rho_-,f)$ be a distributional solution
of~\eqref{eq:rho-plus-collapsed},
\eqref{eq:rho-minus-collapsed}, and~\eqref{eq:f-definition}, and set
\[
s(x,t):=1-f(x,t).
\]
Assume that
\begin{equation}
\label{eq:temporal-traces-compatibility}
\rho_\pm(\cdot,t)
\stackrel{*}{\rightharpoonup}
\rho_{\pm,0}
\quad\text{in }L^\infty(\mathbb R),
\qquad
s(\cdot,t)\stackrel{*}{\rightharpoonup}s_0
\quad\text{in }L^\infty(\mathbb R)
\end{equation}
as $t\downarrow0$, and that
\begin{equation}
\label{eq:compatible_data}
\partial_xs_0
=
\rho_{+,0}-\rho_{-,0}
\qquad
\text{in }\mathcal D'(\mathbb R).
\end{equation}
Then
\begin{equation}
\label{eq:spatial_derivative_of_s}
\partial_xs(\cdot,t)
=
\rho_+(\cdot,t)-\rho_-(\cdot,t)
\qquad
\text{in }\mathcal D'(\mathbb R)
\end{equation}
for a.e.\ $t\in(0,T)$.
\end{lem}

\begin{proof}
Subtracting~\eqref{eq:rho-minus-collapsed}
from~\eqref{eq:rho-plus-collapsed} and, using \(s=1-f\), differentiating~\eqref{eq:f-definition} with respect to \(x\), respectively, we obtain
\begin{align*}
\partial_t(\rho_+(x,t)-\rho_-(x,t))
&=-\partial_x\big(v(x,t)(\rho_ +(x,t)+\rho_-(x,t))\big)\,, \\
\partial_t\partial_xs(x,t)&=
\partial_x\partial_ts(x,t)=
-\partial_x\big(v(x,t)(\rho_+(x,t)+\rho_-(x,t))\big)
\end{align*}
in \(\mathcal D'(\mathbb R\times(0,T))\),  where we have used the
commutativity of distributional derivatives, and hence
\begin{equation}
\label{eq:compatibility-time-derivative}
\partial_t\big(
\partial_xs(x,t)-\rho_+(x,t)+\rho_-(x,t)
\big)=0
\qquad
\text{in }\mathcal D'(\mathbb R\times(0,T))\,.
\end{equation}
Since \(s\), \(\rho_+\), \(\rho_- \in L^1_{\mathrm{loc}}(\mathbb R\times(0,T))\), Fubini's theorem allows us to choose \(E_0\subset(0,T)\) such that
\((0,T)\setminus E_0\) has Lebesgue measure zero and
\[
s(\cdot,t)\,,\,
\rho_+(\cdot,t)\,,\,
\rho_-(\cdot,t)
\in L^1_{\mathrm{loc}}(\mathbb R)
\quad
\text{for every }t\in E_0\,.
\]
Consequently, for every \(\zeta\in C_c^\infty(\mathbb R)\), the scalar map
\[
t\longmapsto 
\left\langle
\partial_xs(\cdot,t)-\rho_+(\cdot,t)+\rho_-(\cdot,t),
\zeta
\right\rangle
=
-\int_{\mathbb R}s(x,t)\partial_x\zeta(x)\,\rd x
-\int_{\mathbb R} (\rho_+(x,t)-\rho_-(x,t))\zeta(x)\,\rd x
\]
belongs to \(L^1_{\mathrm{loc}}(0,T)\), and testing
\eqref{eq:compatibility-time-derivative} with
\(\zeta(x)\vartheta(t)\), where
\(\vartheta\in C_c^\infty(0,T)\), shows that
\[
\frac{\rd}{\rd t}
\left\langle
\partial_xs(\cdot,t)-\rho_+(\cdot,t)+\rho_-(\cdot,t),
\zeta
\right\rangle
=0
\qquad
\text{in }\mathcal D'(0,T)\,.
\]
Thus we may choose \(E_\zeta\subset E_0\) such that
\((0,T)\setminus E_\zeta\) has Lebesgue measure zero and this scalar
map is constant on \(E_\zeta\). Since \(s=1-f\),
\eqref{eq:temporal-traces-compatibility} and the continuity of
distributional differentiation imply
\[
s(\cdot,t)\longrightarrow s_0\,,
\qquad
\partial_xs(\cdot,t)\longrightarrow\partial_xs_0
\quad\text{in }\mathcal D'(\mathbb R)
\quad\text{as }t\downarrow0\,.
\]
Choosing \(t_k\in E_\zeta\) with \(t_k\downarrow0\),
\eqref{eq:temporal-traces-compatibility}
and~\eqref{eq:compatible_data} show that the constant above is
\begin{align*}
\lim_{k\to\infty}
\left\langle
\partial_xs(\cdot,t_k)-\rho_+(\cdot,t_k)+\rho_-(\cdot,t_k),\zeta
\right\rangle
=\left\langle \partial_xs_0-\rho_{+,0}+\rho_{-,0},\zeta \right\rangle
=0\,,
\end{align*}
and therefore
\begin{align*}
\left\langle \partial_xs(\cdot,t)-\rho_+(\cdot,t)+\rho_-(\cdot,t), \zeta\right\rangle
=0 \qquad \text{for every }t\in E_\zeta\,.
\end{align*}
It remains to choose a common set of times, independent of
\(\zeta\), on which the identity holds for every test function. For
every \(m\in\mathbb N\), choose a countable family
\[
\{\zeta_{m,n}:n\in\mathbb N\} \subset C_c^\infty((-m,m))
\]
which is dense in \(C_c^\infty((-m,m))\) with respect to the \(C^1([-m,m])\)-norm, and set
\[
E:= E_0 \cap \bigcap_{m=1}^{\infty} \bigcap_{n=1}^{\infty}E_{\zeta_{m,n}}\,.
\]
Since \((0,T)\setminus E\) is a countable union of sets of Lebesgue
measure zero, it has Lebesgue measure zero, and
\[
\left\langle \partial_xs(\cdot,t)-\rho_+(\cdot,t)+\rho_-(\cdot,t), \zeta_{m,n} \right\rangle
=0 \qquad \text{for every } t\in E \text{ and all } m,n \in\mathbb N\,.
\] 
Now, let \(t\in E\) and \(\zeta\in C_c^\infty(\mathbb R)\), choose \(m\in\mathbb N\) such that
\(\operatorname{supp}\zeta\subset(-m,m)\), and choose a sequence \((\zeta_{m,n_k})_{k\in\mathbb N}\) satisfying
\[
\zeta_{m,n_k}\longrightarrow\zeta \quad \text{in }C^1([-m,m])\,.
\]
Then
\begin{align*}
| &\left\langle \partial_xs(\cdot,t)-\rho_+(\cdot,t)+\rho_-(\cdot,t), \zeta \right\rangle|\\
&\quad =|\left\langle\partial_xs(\cdot,t)-\rho_+(\cdot,t)+\rho_-(\cdot,t),\zeta-\zeta_{m,n_k}\right\rangle|\\
&\quad =\left|-\int_{-m}^{m} s(x,t)\partial_x (\zeta-\zeta_{m,n_k})(x)\,\rd x
-\int_{-m}^{m} (\rho_+(x,t)-\rho_-(x,t))(\zeta-\zeta_{m,n_k})(x)\,\rd x
\right|\\
&\quad \leq \|s(\cdot,t)\|_{L^1(-m,m)}\|\partial_x\zeta-\partial_x\zeta_{m,n_k}\|_{L^\infty(-m,m)}
+
\|\rho_+(\cdot,t)-\rho_-(\cdot,t)\|_{L^1(-m,m)}
\|\zeta-\zeta_{m,n_k}\|_{L^\infty(-m,m)}\\
&\quad \longrightarrow0\,.
\end{align*}
It follows that
\[
\partial_xs(\cdot,t)
=
\rho_+(\cdot,t)-\rho_-(\cdot,t)
\qquad
\text{in }\mathcal D'(\mathbb R)
\]
for every \(t\in E\), and thus for a.e. \(t\in(0,T)\), which
proves~\eqref{eq:spatial_derivative_of_s}.
\end{proof}
\section{Well-Posedness}
\label{Sec3} 

We now study the Cauchy problem associated with the size-collapsed system derived in Lemmas~\ref{lem:size-collapsed-fork-equations}
and~\ref{lem:replicated_frac_eq}.
For an arbitrary \(T\in(0,\infty)\), we seek the right- and left-moving fork densities
\(\rho_+\) and \(\rho_-\) and the replicated fraction \(f\) on
\(\mathbb R\times[0,T]\), and we define the unreplicated fraction, as
before, by
\[
s(x,t) \coloneqq 1-f(x,t),
\]
so that the system takes the form
\begin{equation}
\left\{
\begin{aligned}
\partial_t \rho_+(x,t)+\partial_x\!\big(v(x,t)\rho_+(x,t)\big)&=I(x,t)s(x,t)-\frac{2v(x,t)}{s(x,t)}\rho_+(x,t)\rho_-(x,t), \\
\partial_t \rho_-(x,t)-\partial_x\!\big(v(x,t)\rho_-(x,t)\big)&=I(x,t)s(x,t)-\frac{2v(x,t)}{s(x,t)}\rho_+(x,t)\rho_-(x,t),\\
\partial_t f(x,t)
&=v(x,t)\big(\rho_+(x,t)+\rho_-(x,t)\big),
\end{aligned}
\right.
\label{eq:collapsed_system}
\end{equation}
in \(\mathcal D'(\mathbb R\times(0,T))\), subject to the initial
conditions
\begin{equation}
\rho_\pm(x,0)=\rho_{\pm,0}(x), \qquad f(x,0)=f_0(x), \qquad \text{for a.e. } x \in \R.
\label{eq:initial_data}
\end{equation}
Throughout this section, for every \(T\in(0,\infty)\), we assume
\begin{subequations}
\label{eq:wp-assumptions}
\begin{equation}
\label{eq:wp-initiation}
I \in L^\infty(\mathbb{R}\times(0,T)), \qquad
I(x,t) \geq 0
\quad\text{for a.e. }(x,t)\in\mathbb{R}\times(0,T),
\end{equation}
and
\begin{equation}
\label{eq:wp-velocity-regularity}
v\in L^\infty\bigl(0,T;W^{1,\infty}(\mathbb R)\bigr)\,,
\qquad
v(x,t)\geq0 \text{ for a.e. }(x,t)\in\mathbb R\times(0,T)\,.
\end{equation}
Moreover, the initial data satisfy
\begin{equation}
\label{eq:wp-initial-forks}
\rho_{\pm,0} \in L^\infty(\mathbb{R})\,, \qquad
\rho_{\pm,0}(x) \geq 0
\quad\text{for a.e. }x\in\mathbb{R}\,,
\end{equation} 
and
\begin{equation}
\label{eq:wp-initial-fraction}
f_0 \in L^\infty(\mathbb{R})\,, \qquad
f_0(x)\geq 0
\quad\text{for a.e. }x\in\mathbb{R}\,.
\end{equation}
Finally, setting \(s_0 \coloneqq 1-f_0\), we assume that there exists
\(s_*>0\) such that
\begin{equation}
\label{eq:wp-initial-gap}
s_0(x) \geq s_*
\qquad \text{for a.e. } x \in \mathbb{R}\,.
\end{equation}
\end{subequations}
Note that, for every finite \(T>0\), the assumptions on \(I\) and \(v\)
in~\eqref{eq:wp-initiation}
and~\eqref{eq:wp-velocity-regularity} coincide with
\eqref{eq:assumption-I} and~\eqref{eq:assumption-v}, respectively,
while~\eqref{eq:wp-initial-forks}--\eqref{eq:wp-initial-gap}
require the initial fork densities and replicated fraction to be
bounded and non-negative and the initial unreplicated fraction to be
bounded away from zero.
For each fixed \(T\in(0,\infty)\), we use the notation
\[
M_v\coloneqq 
\|v\|_{L^\infty(\mathbb R\times(0,T))}\,,
\qquad
M_{\partial v}\coloneqq 
\|\partial_xv\|_{L^\infty(\mathbb R\times(0,T))}\,,
\qquad
M_I\coloneqq 
\|I\|_{L^\infty(\mathbb R\times(0,T))}\,,
\]
and
\[
R_0\coloneqq 
\max\left\{
\|\rho_{+,0}\|_{L^\infty(\mathbb R)},
\|\rho_{-,0}\|_{L^\infty(\mathbb R)}
\right\}\,.
\]
\begin{rem}
\label{rem:biological_interpretation_of_the_assumptions}
The assumptions above arise naturally from DNA replication kinetics. 
The non-negativity and boundedness of $I(x,t)$ describe origin firing through
a finite effective initiation rate whose spatial and temporal variation may
reflect origin licensing, chromatin organisation, checkpoint control and
limiting replication factors~\cite{Herrick2002,fragkos2015DNA}.
The non-negativity and boundedness of \(v(x,t)\) allow heterogeneous fork
progression with a finite effective local speed, including severe slowing
and effective local arrest of fork progression on space-time regions where
\(v=0\)~\cite{gauthier2012modeling,gauthier2010defects}. Its spatial Lipschitz regularity is imposed to ensure uniqueness of the
characteristics introduced in the following subsection.
\end{rem}

\subsection{Characteristics and mild formulation}
\label{subsec:characteristics-mild-formulation}
In the first two equations of~\eqref{eq:collapsed_system},
\(\rho_+\) is transported with velocity \(v\), while \(\rho_-\) is
transported with velocity \(-v\). Accordingly, for every
\((x,t)\in\mathbb R\times(0,T]\), we consider the terminal-value
problems
\begin{equation}
\label{eq:char}
\frac{\rd}{\rd\tau}X_\pm(\tau;x,t)
=
\pm v\bigl(X_\pm(\tau;x,t),\tau\bigr)\,,
\qquad
X_\pm(t;x,t)=x\,.
\end{equation}
By~\eqref{eq:wp-velocity-regularity} and Carath\'eodory's
existence and uniqueness
theorem~\cite[Chapter~III, \S~10, Supplement~II,
Theorem~XVIII]{Walter1998}, the two problems
in~\eqref{eq:char} admit unique solutions
\[
X_\pm(\cdot;x,t)\in W^{1,\infty}(0,t)\,,
\]
and the differential equations hold for a.e.\ \(\tau\in(0,t)\). These solutions are the
backward characteristics ending at \(x\) at time \(t\). For
\(x_1,x_2\in\mathbb R\), applying Gronwall's inequality
to~\eqref{eq:char} in both time directions gives
\begin{equation}
\label{eq:char_lipschitz}
e^{-M_{\partial v}(t-\tau)}|x_1-x_2|
\leq
\left|X_\pm(\tau;x_1,t)-X_\pm(\tau;x_2,t)\right|
\leq
e^{M_{\partial v}(t-\tau)}|x_1-x_2|
\end{equation}
for \(0\leq\tau\leq t\leq T\). Consequently,
\(x\mapsto X_\pm(\tau;x,t)\) is a bi-Lipschitz bijection of $\R$ onto $\R$. 
Moreover,  the integral formulation
of~\eqref{eq:char}, uniqueness and~\eqref{eq:char_lipschitz} show that
the map 
\[
(\tau,x,t)\mapsto X_\pm(\tau;x,t)
\] 
is continuous, and hence measurable, on
\(\bigl\{(\tau,x,t)\in[0,T]\times\mathbb R\times[0,T]:
0\leq\tau\leq t\leq T\bigr\}\)\,. 
Together with the bi-Lipschitz property and Fubini's theorem, this shows that compositions with \(X_\pm(\tau;\cdot,t)\) and
\((X_\pm(\tau;\cdot,t))^{-1}\) are measurable and well defined on \(L^\infty(\mathbb R)\)- and
\(L^\infty(\mathbb R\times(0,T))\)-spaces occurring below and preserve a.e. equalities and inequalities.\\

To rewrite the first two equations
of~\eqref{eq:collapsed_system} along the characteristics, we formally
expand the transport terms on \(\mathbb R\times(0,T)\) as
$$
\partial_t\rho_+(x,t) + v(x,t)\partial_x\rho_+(x,t) + (\partial_x v)(x,t)\rho_+(x,t)
=
I(x,t)s(x,t)-\frac{2v(x,t)}{s(x,t)}\rho_+(x,t)\rho_-(x,t),
$$
and
$$
\partial_t\rho_-(x,t) - v(x,t)\partial_x\rho_-(x,t) - (\partial_x v)(x,t)\rho_-(x,t)
=
I(x,t)s(x,t)-\frac{2v(x,t)}{s(x,t)}\rho_+(x,t)\rho_-(x,t)\, , 
$$
which, for every \(t\in(0,T]\) and for a.e.\ \((\tau,x)\in(0,t)\times\mathbb R\), reduce along
\(X_+(\cdot;x,t)\) and \(X_-(\cdot;x,t)\), respectively, to
\begin{subequations}
\begin{equation}
\frac{\rd}{\rd\tau}\rho_+\big(X_+(\tau;x,t),\tau\big)
+A_+(\tau;x,t)\rho_+\big(X_+(\tau;x,t),\tau\big)
=G_+(\tau;x,t)\,,
\label{eq:ode_plus}
\end{equation}
and
\begin{equation}
\frac{\rd}{\rd\tau}\rho_-(X_-(\tau;x,t),\tau)
+A_-(\tau;x,t)\rho_-(X_-(\tau;x,t),\tau)
=G_-(\tau;x,t)\,,
\label{eq:ode_minus}
\end{equation}
\end{subequations}
where
\begin{equation*}
A_+(\tau;x,t)\coloneqq (\partial_x v)(X_+(\tau;x,t),\tau)
+\frac{2v(X_+(\tau;x,t),\tau)}{s(X_+(\tau;x,t),\tau)}
\rho_-(X_+(\tau;x,t),\tau)\,,
\end{equation*}
and
\begin{equation*}
A_-(\tau;x,t)\coloneqq -(\partial_x v)(X_-(\tau;x,t),\tau)
+\frac{2v(X_-(\tau;x,t),\tau)}{s(X_-(\tau;x,t),\tau)}
\rho_+(X_-(\tau;x,t),\tau)\,
\end{equation*}
while
\begin{equation*}
G_\pm(\tau;x,t)\coloneqq I(X_\pm(\tau;x,t),\tau)s(X_\pm(\tau;x,t),\tau)\,.
\end{equation*}
Solving~\eqref{eq:ode_plus} and~\eqref{eq:ode_minus} by the
integrating-factor method and using~\eqref{eq:char}
and~\eqref{eq:initial_data}, we obtain
\begin{subequations}
\begin{align}
\rho_+(x,t)&= \rho_{+,0}(X_+(0;x,t))
\exp\left(-\int_0^t A_+(\eta;x,t)\,\rd \eta\right) \nonumber \\
&\qquad +\int_0^t G_+(\sigma;x,t)
\exp\left(-\int_\sigma^t A_+(\eta;x,t)\,\rd \eta\right)\rd \sigma,
\label{eq:duhamel_plus}
\end{align}
and 
\begin{align}
\rho_-(x,t)&=\rho_{-,0}(X_-(0;x,t))\exp\left(-\int_0^t A_-(\eta;x,t)\,\rd \eta\right) \nonumber \\
&\qquad +\int_0^t G_-(\sigma;x,t)
\exp\left(-\int_\sigma^t A_-(\eta;x,t)\,\rd \eta\right)\rd \sigma
\label{eq:duhamel_minus}
\end{align}
for a.e.\ $(x,t)\in \mathbb R\times(0,T)$, while the third equation in~\eqref{eq:collapsed_system} yields
\begin{equation}
f(x,t)=f_0(x)+\int_0^t v(x,\tau)\big(\rho_+(x,\tau)+\rho_-(x,\tau)\big)\,\rd \tau, \qquad \text{for a.e. }(x,t)\in \mathbb R\times(0,T)\,.
\label{eq:duhamel_f}
\end{equation}
\end{subequations}
We use the following notion of a mild solution.

\begin{deff}
\label{def:mild_solution}
We call a triple
\[
(\rho_+,\rho_-,f)
\in
\bigl(L^\infty(\mathbb R\times(0,T))\bigr)^3
\]
a \emph{mild solution}
of~\eqref{eq:collapsed_system}--\eqref{eq:initial_data}
on \(\mathbb R\times[0,T]\) if
\(\rho_\pm(x,t)\geq0\) and \(0\leq f(x,t)\leq1\) for a.e. \((x,t)\in\mathbb R\times(0,T)\), if \(s=1-f\) satisfies
\[
\operatorname*{ess\,inf}_{(x,t)\in\mathbb R\times(0,T)}
s(x,t)>0\,,
\]
and if identities~\eqref{eq:duhamel_plus}--\eqref{eq:duhamel_f}
hold for a.e.\((x,t)\in\mathbb R\times(0,T)\).
\end{deff}
The mild and distributional formulations are related as follows.
\begin{prop}
\label{prop:mild-implies-distributional}
Assume~\eqref{eq:wp-initiation}--\eqref{eq:wp-initial-fraction}
and let \((\rho_+,\rho_-,f)\) be a mild solution on
\(\mathbb R\times[0,T]\). Then~\eqref{eq:collapsed_system} holds in
\(\mathcal D'(\mathbb R\times(0,T))\). Moreover, the initial
conditions in~\eqref{eq:initial_data} are attained in the sense that
\begin{align*}
\rho_\pm(\cdot,t)\stackrel{*}{\rightharpoonup}\rho_{\pm,0} \quad  \text{in }L^\infty(\mathbb R) \qquad \text{and} \qquad f(\cdot,t)&\longrightarrow f_0 \quad \text{in }L^\infty(\mathbb R)
\end{align*}
as $t\downarrow0$.
\end{prop}

\begin{proof}
By Definition~\ref{def:mild_solution} and
\eqref{eq:wp-initiation}--\eqref{eq:wp-initial-fraction},
\[
\tfrac{1}{s},\ Is,\ v\rho_\pm,\
\frac{v\rho_+\rho_-}{s},\
v(\rho_++\rho_-)
\in L^\infty(\mathbb R\times(0,T))\subset L^1_{\mathrm{loc}}(\mathbb R\times(0,T))\,,
\]
and hence the right-hand sides of~\eqref{eq:collapsed_system} are locally integrable. 
Next, for every \(y\in\mathbb R\), the same Carath\'eodory theorem used
for~\eqref{eq:char} gives unique forward characteristics
\(Y_\pm(\cdot;y)\in W^{1,\infty}(0,T)\) for 
\begin{equation}
\label{eq:forward_characteristic_equation}
\frac{\rd}{\rd t} Y_\pm(t;y) = \pm v\bigl(Y_\pm(t;y),t \bigr)
\quad \text{for a.e. } t \in (0,T), \qquad Y_\pm(0;y) = y,
\end{equation}
satisfying
\[
Y_\pm(t;y)
=y\pm\int_0^t v(Y_\pm(r;y),r)\,\rd r \qquad \text{for every }t\in[0,T]\,,
\]
and, since \(r\mapsto X_\pm(r;Y_\pm(t;y),t)\) and \(r\mapsto Y_\pm(r;y)\) solve the same characteristic equation on
\([0,t]\) and agree at \(r=t\), uniqueness gives
\begin{equation}
\label{eq:flow_identity}
X_\pm(\tau;Y_\pm(t;y),t)
= Y_\pm(\tau;y) \qquad
\text{for }0\leq\tau\leq t\leq T\,. 
\end{equation}
In particular,
\(X_\pm(0;Y_\pm(t;y),t)=y\), and therefore the bijectivity of
\(x\mapsto X_\pm(0;x,t)\) shows that
\[
Y_\pm(t;\cdot)^{-1}(x)
= X_\pm(0;x,t) \qquad \text{for every }(x,t)\in\mathbb R\times[0,T].
\]
Moreover, the inverse derivative formula and the variational formula
for \(X_\pm\) give, for every \(t\in[0,T]\) and for a.e. \(y\in\mathbb R\),
\begin{equation}
\label{eq:forward_jacobian}
\partial_yY_\pm(t;y)
=
\exp\left(\pm\int_0^t
(\partial_xv)(Y_\pm(r;y),r)\,\rd r
\right)>0,
\end{equation}
and, since the exponent in this formula is absolutely continuous in
\(t\), the fundamental theorem 
and the chain rule yield 
\begin{equation}
\label{eq:forward_variational_equation}
\frac{\rd}{\rd t}\partial_yY_\pm(t;y) =\pm(\partial_xv)(Y_\pm(t;y),t) \partial_yY_\pm(t;y) 
\end{equation}
for a.e. \((y,t)\in\mathbb R\times(0,T)\). Consequently,
\begin{equation}
\label{eq:forward_displacement}
|Y_\pm(t;y)-y| \leq \int_0^t |v(Y_\pm(r;y),r)|\,\rd r
\leq M_vt \quad \text{for every }(y,t)\in\mathbb R\times[0,T]\,,
\end{equation}
whereas~\eqref{eq:forward_jacobian} implies
\begin{equation}
\label{eq:forward_jacobian_bound}
\|\partial_yY_\pm(t;\cdot)-1\|_{L^\infty(\mathbb R)}
\leq
\left\|
\exp\Big(
\int_0^t |(\partial_xv)(Y_\pm(r;\cdot),r)|\,\rd r\Big)-1
\right\|_{L^\infty(\mathbb R)}\leq e^{M_{\partial v}t}-1
\end{equation}
for every $t\in[0,T]$.
Next, pulling back
\eqref{eq:duhamel_plus}-\eqref{eq:duhamel_minus} by
\(x=Y_\pm(t;y)\) and using~\eqref{eq:flow_identity}, we obtain
\begin{align}
\label{eq:duhamel_forward}
\rho_\pm(Y_\pm(t;y),t)
&=\rho_{\pm,0}(y) \exp\left(
-\int_0^t\Big[\pm(\partial_xv)+\frac{2v\rho_\mp}{s}
\Big](Y_\pm(\eta;y),\eta)\,\rd\eta\right) 
\nonumber \\
&\quad+
\int_0^t(Is)(Y_\pm(\sigma;y),\sigma)
\exp\left(-\int_\sigma^t
\Big[\pm(\partial_xv)+\frac{2v\rho_\mp}{s}
\Big](Y_\pm(\eta;y),\eta)\,\rd\eta
\right)\rd\sigma
\end{align}
for a.e. \((y,t)\in\mathbb R\times(0,T)\).
Moreover, Definition~\ref{def:mild_solution} and
\eqref{eq:wp-initiation}--\eqref{eq:wp-initial-fraction} imply
\begin{align*}
\Big|\Big[\pm(\partial_xv)&+\frac{2v\rho_\mp}{s}
\Big](Y_\pm(t;y),t)\Big| + |(Is)(Y_\pm(t;y),t)|\\
& \leq
M_{\partial v}+2M_v\|\tfrac{1}{s}\|_{L^\infty(\mathbb R\times(0,T))}\|\rho_\mp\|_{L^\infty(\mathbb R\times(0,T))}+M_I<\infty
\end{align*}
for a.e. \((y,t)\in\mathbb R\times(0,T)\), and therefore
\eqref{eq:duhamel_forward} implies that, for a.e. \(y\in\mathbb R\),
\begin{equation}
\label{eq:forward_pullback_continuity}
\rho_\pm\bigl(Y_\pm(\cdot;y),\cdot\bigr)
\in W^{1,\infty}(0,T)\,.
\end{equation}
Differentiating~\eqref{eq:duhamel_forward} then yields
\begin{align*}
\frac{\rd}{\rd t} \rho_\pm(Y_\pm(t;y),t)
=\Big(Is-\frac{2v\rho_+\rho_-}{s}\mp(\partial_xv)\rho_\pm\Big)(Y_\pm(t;y),t) \qquad \text{for a.e. } t \in (0,T)\,,
\end{align*}
and, combining this identity with the variational equation~\eqref{eq:forward_variational_equation} for \(\partial_yY_\pm\), we infer, for a.e. \((y,t)\in\mathbb R\times(0,T)\), that
\begin{align*}
\frac{\rd}{\rd t}
\big(\rho_\pm(Y_\pm(t;y),t)\partial_yY_\pm(t;y)\big)
= \Big(Is-\frac{2v\rho_+\rho_-}{s}\Big)(Y_\pm(t;y),t)\partial_yY_\pm(t;y)\,.
\end{align*}
Now, for every \(\varphi\in C_c^\infty(\mathbb R\times(0,T))\),  differentiation along the forward characteristics \(Y_\pm\) and the preceding identity give
\begin{align*}
\frac{\rd}{\rd t}
&\big[\rho_\pm (Y_\pm(t;y),t)\partial_yY_\pm(t;y)
\varphi(Y_\pm(t;y),t)\big] \\
&=\Big[
\Big(Is-\frac{2v\rho_+\rho_-}{s}\Big)\varphi
+\rho_\pm(\partial_t\varphi\pm v\partial_x\varphi)
\Big](Y_\pm(t;y),t) \partial_yY_\pm(t;y)
\end{align*}
for a.e. \((y,t)\in\mathbb R\times(0,T)\), and, since
\eqref{eq:forward_displacement} shows that the expression on the
right vanishes outside a fixed compact interval in \(y\) and is
bounded, integration with respect to \(t\), Fubini's theorem, and the
change of variables \(x=Y_\pm(t;y)\),
\(\rd x=\partial_yY_\pm(t;y)\,\rd y\), yield 
\begin{align*}
\int_0^T\int_{\mathbb R}
\Big[
\rho_\pm(x,t)\partial_t\varphi(x,t)
&\pm v(x,t)\rho_\pm(x,t)\partial_x\varphi(x,t)
\\
&+\Big(I(x,t)s(x,t)-\frac{2v(x,t)\rho_+(x,t)\rho_-(x,t)}{s(x,t)}
\Big)\varphi(x,t)
\Big]\,\rd x\,\rd t =0\,.
\end{align*}
Thus the first two equations of~\eqref{eq:collapsed_system} hold in
\(\mathcal D'(\mathbb R\times(0,T))\).

Next, for a.e. \(x\in\mathbb R\), we have \(v(x,\cdot)(\rho_+(x,\cdot)+\rho_-(x,\cdot))\in L^\infty(0,T)\), and hence~\eqref{eq:duhamel_f} implies that
\[
f(x,\cdot)\in W^{1,\infty}(0,T)\,,
\quad
\partial_tf(x,t)=v(x,t)(\rho_+(x,t)+\rho_-(x,t))
\qquad\text{for a.e. }t\in(0,T)\,,
\]
which proves that the third equation of \eqref{eq:collapsed_system} holds in
\(\mathcal D'(\mathbb R\times(0,T))\), while \eqref{eq:duhamel_f} further gives
\begin{align*}
\|f(\cdot,t)-f_0\|_{L^\infty(\mathbb R)}\leq tM_v\big(
\|\rho_+\|_{L^\infty(\mathbb R\times(0,T))}+\|\rho_-\|_{L^\infty(\mathbb R\times(0,T))}\big) \longrightarrow0
\qquad\text{as }t\downarrow0\,.
\end{align*}
It remains to prove the weak-star traces of the fork densities. For
every \(\psi\in C_c^\infty(\mathbb R)\), choose \(R>0\) such that
\(\operatorname{supp}\psi\subset[-R,R]\). 
Then
\eqref{eq:forward_displacement} and the mean value theorem imply
\begin{align*}
\|\psi\bigl(Y_\pm(t;\cdot)\bigr)-\psi\|_{L^1(\mathbb R)}
&=\int_{-R-M_vT}^{R+M_vT}| \psi(Y_\pm(t;y)\bigr)-\psi(y)|\,\rd y
\leq \int_{-R-M_vT}^{R+M_vT} M_vt\|\partial_x\psi\|_{L^\infty(\mathbb R)}\,\rd y
\longrightarrow0\,,
\end{align*}
and, together with~\eqref{eq:forward_jacobian_bound}, this yields
\begin{align*}
&\|\psi(Y_\pm(t;\cdot))\partial_yY_\pm(t;\cdot)-\psi\|_{L^1(\mathbb R)}
\\
&\quad\leq \|\partial_yY_\pm(t;\cdot)\|_{L^\infty(\mathbb R)}
\|\psi(Y_\pm(t;\cdot))-\psi\|_{L^1(\mathbb R)}
+\|\partial_yY_\pm(t;\cdot)-1\|_{L^\infty(\mathbb R)}
\|\psi\|_{L^1(\mathbb R)}
\\
&\quad\leq e^{M_{\partial v}t} \|\psi(Y_\pm(t;\cdot))-\psi\|_{L^1(\mathbb R)}
+ \|\psi\|_{L^1(\mathbb R)}(e^{M_{\partial v}t}-1)
\longrightarrow0\,,
\end{align*}
as \(t\downarrow0\). Moreover, the change of variables \(x=Y_\pm(t;y)\) shows that
\[
\|\psi(Y_\pm(t;\cdot))\partial_yY_\pm(t;\cdot)\|_{L^1(\mathbb R)} =\|\psi\|_{L^1(\mathbb R)}
\qquad
\text{for every }t\in[0,T],
\]
and hence, by density,
\(\psi(Y_\pm(t;\cdot)) \partial_yY_\pm(t;\cdot)\) converges to \(\psi\) in \(L^1(\mathbb R)\) as \(t\downarrow0\) for
every \(\psi\in L^1(\mathbb R)\). Since
\(X_\pm(0;Y_\pm(t;y),t)=y\), it follows that
\begin{align*}
\int_{\mathbb R}\rho_{\pm,0}(X_\pm(0;x,t))\psi(x)\,\rd x
=\int_{\mathbb R}\rho_{\pm,0}(y)\psi(Y_\pm(t;y))
\partial_yY_\pm(t;y)\,\rd y\longrightarrow
\int_{\mathbb R} \rho_{\pm,0}(y)\psi(y)\,\rd y
\end{align*}
for every \(\psi\in L^1(\mathbb R)\).

Finally, the definitions of \(A_\pm\) and \(G_\pm\) in
Subsection~\ref{subsec:characteristics-mild-formulation}, along with \eqref{eq:wp-initiation}, \eqref{eq:wp-velocity-regularity}, and
Definition~\ref{def:mild_solution}, imply, for every
\(t\in(0,T]\), that
\begin{align*}
\|A_\pm(\cdot;\cdot,t)\|_{
L^\infty((0,t)\times\mathbb R)}
&\leq
M_{\partial v} + 2M_v \left\|\frac1s\right\|_{L^\infty(\mathbb R\times(0,T))}
\max\{\|\rho_+\|_{L^\infty(\mathbb R\times(0,T))}, \|\rho_-\|_{L^\infty(\mathbb R\times(0,T))}\}\,,
\\
\|G_\pm(\cdot;\cdot,t)\|_{L^\infty((0,t)\times\mathbb R)}
&\leq M_I,
\end{align*}
and therefore
\begin{align*}
&\left\|\exp\left(-\int_0^tA_\pm(\eta;\cdot,t)\,\rd\eta
\right)-1\right\|_{L^\infty(\mathbb R)}
+ \left\|\int_0^t G_\pm(\sigma;\cdot,t)\exp\left(
-\int_\sigma^tA_\pm(\eta;\cdot,t)\,\rd\eta
\right)\,\rd\sigma \right\|_{L^\infty(\mathbb R)}
\\
&\qquad\leq
(1+M_It)\exp\left(t\|A_\pm(\cdot;\cdot,t)\|_{L^\infty((0,t)\times\mathbb R)}\right)-1
\longrightarrow0
\end{align*}
as $t\downarrow0$, and, since the bi-Lipschitz bijection
\(x\mapsto X_\pm(0;x,t)\) satisfies
\[
\left\|
\rho_{\pm,0}\bigl(X_\pm(0;\cdot,t)\bigr)
\right\|_{L^\infty(\mathbb R)}
=
\|\rho_{\pm,0}\|_{L^\infty(\mathbb R)},
\]
the preceding convergence of the transported initial term, the last
estimate, and~\eqref{eq:duhamel_plus}--\eqref{eq:duhamel_minus} prove
\[
\rho_\pm(\cdot,t)
\stackrel{*}{\rightharpoonup}\rho_{\pm,0}
\quad\text{in }L^\infty(\mathbb R)
\quad\text{as }t\downarrow0.
\]
\end{proof}

\subsection{The fixed-point map}
We now construct a local-in-time mild solution by a Banach fixed-point
argument. To this end, for \(R\geq R_0\) and \(\theta\in(0,T]\), we define the set
\begin{equation}
K_{R,\theta}\coloneqq 
\left\{
(\rho_+,\rho_-,f)\in
\bigl(L^\infty(\mathbb R\times(0,\theta))\bigr)^3:
\begin{array}{l}
0\leq\rho_\pm(x,t)\leq R,\\[1mm]
0\leq f(x,t)\leq1,\\[1mm]
1-f(x,t)\geq s_*/2
\end{array}
\text{ for a.e. }(x,t)\in\mathbb R\times(0,\theta)
\right\}.
\label{eq:KRT}
\end{equation}
For $u\coloneqq (\rho_+,\rho_-,f)$, $\widetilde u\coloneqq (\widetilde\rho_+,\widetilde\rho_-,\widetilde f) \in K_{R,\theta}$, we endow \(K_{R,\theta}\) with the metric \(d_\theta\) given by
\begin{equation}
d_\theta(u,\widetilde u)\coloneqq
\|\rho_+-\widetilde\rho_+\|_{L^\infty(\mathbb R\times(0,\theta))}
+
\|\rho_--\widetilde\rho_-\|_{L^\infty(\mathbb R\times(0,\theta))}
+
\|f-\widetilde f\|_{L^\infty(\mathbb R\times(0,\theta))}\,.
\label{eq:metric}
\end{equation}
By~\eqref{eq:wp-initial-forks}--\eqref{eq:wp-initial-gap},
\[
0\leq\rho_{\pm,0}(x)\leq R_0\leq R\,,
\qquad
0\leq f_0(x)\leq1-s_*<1\,,
\qquad
1-f_0(x)=s_0(x)\geq s_*>\frac{s_*}{2}
\]
for a.e.\ \(x\in\mathbb R\). Consequently, the time-independent
initial-data triple \(\bigl(\rho_{+,0},\rho_{-,0},f_0\bigr)\) belongs to \(K_{R,\theta}\), and we observe that, since each inequality in~\eqref{eq:KRT} is preserved under convergence in
\(\bigl(L^\infty(\mathbb R\times(0,\theta))\bigr)^3\), the set
\(K_{R,\theta}\) is closed in this Banach space and hence, when
endowed with the metric \(d_\theta\), is a complete metric space.
Now, let $u\coloneqq(\rho_+,\rho_-,f)\in K_{R,\theta}$, set
$s(x,t)\coloneqq 1-f(x,t)$ for a.e. 
$(x,t)\in\mathbb R\times(0,\theta)$, and define
\begin{equation}
\label{eq:Phi_theta}
\Phi_\theta(u)
\coloneqq
\widehat u
\coloneqq
(\widehat\rho_+,\widehat\rho_-,\widehat f)\,,
\end{equation}
where the components of \(\widehat u\) are given as follows. First, we define
\begin{equation}
\label{eq:Ahat_pm}
\widehat A_\pm(\tau;x,t)\coloneqq
\pm(\partial_xv)\bigl(X_\pm(\tau;x,t),\tau\bigr)+
\frac{2v\bigl(X_\pm(\tau;x,t),\tau\bigr)}{s\bigl(X_\pm(\tau;x,t),\tau\bigr)}
\rho_\mp\bigl(X_\pm(\tau;x,t),\tau\bigr)
\end{equation}
and
\begin{equation}
\label{eq:Ghat_pm}
\widehat G_\pm(\tau;x,t)
\coloneqq I\bigl(X_\pm(\tau;x,t),\tau\bigr) s\bigl(X_\pm(\tau;x,t),\tau\bigr)
\end{equation}
for a.e. \((\tau,x,t)\in (0,\theta)\times\mathbb R\times(0,\theta)\) with \(\tau<t\).
For a.e. \((x,t)\in\mathbb R\times(0,\theta)\) and every
\(\sigma\in[0,t]\), we set
\begin{equation}
\label{eq:Ehat}
\widehat E_\pm(\sigma,t;x)
\coloneqq 
\exp\left(-\int_\sigma^t \widehat A_\pm(\eta;x,t)\,\rd\eta \right)
\end{equation}
and then define
\begin{equation}
\label{eq:Phi_plus}
\widehat\rho_+(x,t)\coloneqq
\rho_{+,0}\bigl(X_+(0;x,t)\bigr)
\widehat E_+(0,t;x)+ \int_0^t \widehat G_+(\sigma;x,t) \widehat E_+(\sigma,t;x)\,\rd\sigma\,,
\end{equation}
\begin{equation}
\label{eq:Phi_minus}
\widehat\rho_-(x,t)
\coloneqq \rho_{-,0}\bigl(X_-(0;x,t)\bigr)
\widehat E_-(0,t;x)+ \int_0^t \widehat G_-(\sigma;x,t) \widehat E_-(\sigma,t;x)\,\rd\sigma\,,
\end{equation}
and
\begin{equation}
\label{eq:Phi_f}
\widehat f(x,t)
\coloneqq f_0(x) + \int_0^t v(x,\tau) \bigl( \widehat\rho_+(x,\tau)
+ \widehat\rho_-(x,\tau) \bigr)\,\rd\tau
\end{equation}
for a.e. \((x,t)\in\mathbb R\times(0,\theta)\).
By~\eqref{eq:wp-initiation}, \eqref{eq:wp-velocity-regularity},
and~\eqref{eq:KRT}, equations~\eqref{eq:Ahat_pm}
and~\eqref{eq:Ghat_pm} imply
\[
|\widehat A_\pm(\tau;x,t)|\leq M_{\partial v}
+\frac{4M_vR}{s_*} \quad \text{and} \quad 
0\leq\widehat G_\pm(\tau;x,t) \leq M_I
\]
for a.e.\((\tau,x,t)\in(0,\theta)\times\mathbb R\times(0,\theta)\)
with \(\tau<t\), and~\eqref{eq:Ehat}, together with the first
estimate, further gives
\[
0<\widehat E_\pm(\sigma,t;x)  
\leq
\exp\left[\left(M_{\partial v}+\frac{4M_vR}{s_*}\right)(t-\sigma)\right]
\]
for a.e.\ \((x,t)\in\mathbb R\times(0,\theta)\) and every
\(\sigma\in[0,t]\). Hence,
\[
\|\widehat\rho_\pm\|_{L^\infty(\mathbb R\times(0,\theta))}\leq
(R_0+M_I\theta) \exp\left[\left(M_{\partial v}+\frac{4M_vR}{s_*}\right)\theta\right]
\]
and 
\[ 
\|\widehat f\|_{L^\infty(\mathbb R\times(0,\theta))}\leq
\|f_0\|_{L^\infty(\mathbb R)}+2M_v\theta(R_0+M_I\theta)
\exp\left[\left(M_{\partial v}+\frac{4M_vR}{s_*}
\right)\theta\right]\,,
\]
and consequently, \eqref{eq:Phi_theta}--\eqref{eq:Phi_f} yield the
well-defined map
\(\Phi_\theta:K_{R,\theta}\longrightarrow\bigl(L^\infty(\mathbb R\times(0,\theta))\bigr)^3\).

\subsection{Self-mapping property}

We next show that, for a suitable \(R\) and sufficiently small
\(\theta\), \(\Phi_\theta\) maps \(K_{R,\theta}\) into itself, and we begin by proving the non-negativity of its components.

\begin{lem}
\label{lem:selfmap-positivity}
Assume~\eqref{eq:wp-initiation}--\eqref{eq:wp-initial-fraction}.
Then, for every \(u\in K_{R,\theta}\),
\[
\widehat\rho_\pm(x,t)\geq0,
\qquad
\widehat f(x,t)\geq0
\]
for a.e. \((x,t)\in\mathbb R\times(0,\theta)\).
\end{lem}

\begin{proof}
Using~\eqref{eq:wp-initiation}, \eqref{eq:wp-initial-forks},
and~\eqref{eq:KRT}, together with the positivity of the exponential
in~\eqref{eq:Ehat}, it follows from~\eqref{eq:Ghat_pm},
\eqref{eq:Phi_plus}, and~\eqref{eq:Phi_minus} that
\[
\widehat\rho_\pm(x,t)\geq0 \qquad \text{ for a.e. } (x,t)\in\mathbb R\times(0,\theta)
\]
and, since \(v(x,t)\geq0\) and \(f_0(x)\geq0\) by
\eqref{eq:wp-velocity-regularity}
and~\eqref{eq:wp-initial-fraction}, respectively,
\eqref{eq:Phi_f} further gives
\[
\widehat f(x,t)\geq f_0(x)\geq0 \qquad \text{ for a.e. } (x,t)\in\mathbb R\times(0,\theta)\,.
\]
\end{proof}

We prove next that the fork densities are uniformly bounded.

\begin{lem}
\label{lem:selfmap-rho-bound}
Assume~\eqref{eq:wp-initiation}--\eqref{eq:wp-initial-forks}, and let
\(B:[0,T]\to[0,\infty)\) be given by
\begin{equation}
B(t):=
\begin{cases}
R_0e^{M_{\partial v}t}+\dfrac{M_I}{M_{\partial v}}\big(e^{M_{\partial v}t}-1\big),& M_{\partial v}>0\,,\\[2mm]
R_0+M_It,& M_{\partial v}=0\,.
\end{cases}
\label{eq:BofT}
\end{equation}
Then, for every \(u\in K_{R,\theta}\),
\[
0\leq\widehat\rho_\pm(x,t)\leq B(t)
\]
for a.e. \((x,t)\in\mathbb R\times(0,\theta)\).
\end{lem}

\begin{proof}
It follows from~\eqref{eq:KRT} that
\[
\rho_\mp\bigl(X_\pm(\tau;x,t),\tau\bigr)\ge0\,,
\qquad
s\bigl(X_\pm(\tau;x,t),\tau\bigr)\ge\frac{s_*}{2}
\]
for a.e. \((\tau,x,t)\in(0,\theta)\times\mathbb R\times(0,\theta)\)
with \(\tau<t\), and hence
\eqref{eq:wp-velocity-regularity} and~\eqref{eq:Ahat_pm} give
\[
\widehat A_\pm(\tau;x,t)
\geq -\left| (\partial_xv)\bigl(X_\pm(\tau;x,t),\tau\bigr)\right|
\geq -M_{\partial v}
\]
for the same arguments, which, together with~\eqref{eq:Ehat}, yields
\begin{equation}
0<\widehat E_\pm(\sigma,t;x)
= \exp\left( -\int_\sigma^t \widehat A_\pm(\eta;x,t)\,\rd\eta \right)
\leq
\exp\bigl(M_{\partial v}(t-\sigma)\bigr)
\label{eq:Ebound}
\end{equation}
for a.e.\ \((x,t)\in\mathbb R\times(0,\theta)\) and every
\(\sigma\in[0,t]\). Moreover, \eqref{eq:wp-initiation},
\eqref{eq:KRT}, and~\eqref{eq:Ghat_pm} give
\[
0\leq\widehat G_\pm(\sigma;x,t)\leq
M_I s\bigl(X_\pm(\sigma;x,t),\sigma\bigr)\leq M_I
\]
for a.e.\((\sigma,x,t)\in(0,\theta)\times\mathbb R\times(0,\theta)\)
with \(\sigma<t\), and therefore, combining this estimate
and~\eqref{eq:Ebound} with~\eqref{eq:wp-initial-forks}
and~\eqref{eq:Phi_plus}--\eqref{eq:Phi_minus}, we obtain
\[
0\leq\widehat\rho_\pm(x,t)
\leq
R_0e^{M_{\partial v}t}
+M_I\int_0^t e^{M_{\partial v}(t-\sigma)}\,\rd\sigma
=B(t)
\]
for a.e. \((x,t)\in\mathbb R\times(0,\theta)\).
\end{proof}

We next derive a lower bound for the unreplicated fraction.

\begin{lem}
\label{lower-bound}
Assume~\eqref{eq:wp-initiation}--\eqref{eq:wp-initial-forks}
and~\eqref{eq:wp-initial-gap}, and suppose that
\begin{equation}
\label{eq:smalltime_s}
2M_v\theta B(\theta)\leq\frac{s_*}{2}.
\end{equation}
Then, for every \(u\in K_{R,\theta}\), the component
\(\widehat f\) of \(\Phi_\theta(u)\) satisfies
\[
1-\widehat f(x,t)\geq\frac{s_*}{2}
\]
for a.e.\ \((x,t)\in\mathbb R\times(0,\theta)\).
\end{lem}

\begin{proof}
It follows from~\eqref{eq:wp-velocity-regularity},
\eqref{eq:wp-initial-gap}, and~\eqref{eq:Phi_f}, together with
Lemma~\ref{lem:selfmap-rho-bound}, the monotonicity of \(B\) following
from~\eqref{eq:BofT}, and~\eqref{eq:smalltime_s}, that
\begin{align*}
1-\widehat f(x,t)
&= s_0(x) -\int_0^t v(x,\tau) \bigl(
\widehat\rho_+(x,\tau)+\widehat\rho_-(x,\tau)
\bigr)\,\rd\tau
\geq
s_*-2M_v\int_0^tB(\tau)\,\rd\tau\\
&\geq
s_*-2M_vtB(t)
\geq
s_*-2M_v\theta B(\theta)
\geq
\frac{s_*}{2}
\end{align*}
for a.e. \((x,t)\in\mathbb R\times(0,\theta)\), which proves the
assertion.
\end{proof}

We can now prove the self-mapping property by combining
Lemmas~\ref{lem:selfmap-positivity}--\ref{lower-bound}.

\begin{prop}
\label{prop:self-mapping}
Fix \(T\in(0,\infty)\) and assume~\eqref{eq:wp-assumptions}.
Then there exist \(R_1>R_0\) and \(T_1\in(0,T]\) such that, for every \(0<\theta\leq T_1\),
\[
\Phi_\theta\bigl(K_{R_1,\theta}\bigr)
\subset K_{R_1,\theta}\,.
\]
\end{prop}

\begin{proof}
Since \(B(0)=R_0\) and \(B\) is non-decreasing on \([0,T]\)
by~\eqref{eq:BofT}, we set
\[
R_1:=B(T)+1>R_0
\]
and
\[
T_1:=
\begin{cases}
T\,,
& M_vB(T)=0\,,\\
\displaystyle
\min\left\{
T,\frac{s_*}{4M_vB(T)}
\right\}\,,
& M_vB(T)>0\,.
\end{cases}
\]
Then \(T_1\in(0,T]\), and, for every \(0<\theta\leq T_1\) and
\(0<t<\theta\),
\[
B(t)\leq B(\theta)\leq B(T)<R_1,
\]
while~\eqref{eq:smalltime_s} holds, and therefore
Lemmas~\ref{lem:selfmap-positivity}--\ref{lower-bound}
and~\eqref{eq:KRT} give
\[
\Phi_\theta(u)\in K_{R_1,\theta}
\qquad
\text{for every }u\in K_{R_1,\theta},
\]
which proves the assertion.
\end{proof}

\subsection{Contraction estimate}
We next prove that the fixed-point map is a strict contraction when the time interval is chosen sufficiently small.
\begin{prop}
\label{prop:contraction}
Assume~\eqref{eq:wp-assumptions}, and let \(R_1\) and \(T_1\) be
given by Proposition~\ref{prop:self-mapping}. Then there exists
\(T_2\in(0,T_1]\) such that, for every \(0<\theta\leq T_2\), the map
\[
\Phi_\theta:
K_{R_1,\theta}
\longrightarrow
K_{R_1,\theta}
\]
is a strict contraction with respect to the metric \(d_\theta\).
\end{prop}

\begin{proof}
By Proposition~\ref{prop:self-mapping}, \(\Phi_\theta\) is a self-map of \(K_{R_1,\theta}\) for every \(0<\theta\leq T_1\). Fix such a
\(\theta\), take
\[
u=(\rho_+,\rho_-,f)\,,
\quad
\widetilde u=
(\widetilde\rho_+,\widetilde\rho_-,\widetilde f)
\]
in \(K_{R_1,\theta}\), and set
\[
s(x,t):=1-f(x,t)\,,
\quad
\widetilde s(x,t):=1-\widetilde f(x,t)\,.
\]
To indicate the input in the quantities defined by
\eqref{eq:Ahat_pm}--\eqref{eq:Phi_f}, we write, for example,
\(\widehat A_\pm[u]\), \(\widehat\rho_\pm[u]\), and
\(\widehat f[u]\), and use the corresponding notation with
\(\widetilde u\). The characteristics are the same for both inputs
since they depend only on \(v\). It follows from~\eqref{eq:KRT} that
\[
s(x,t),\widetilde s(x,t)\geq\frac{s_*}{2}\,,
\quad
0\leq\rho_\pm(x,t),\widetilde\rho_\pm(x,t)\leq R_1
\]
for a.e. \((x,t)\in\mathbb R\times(0,\theta)\), and hence
\begin{align*}
\left|
\frac{\rho_\pm(x,t)}{s(x,t)}
- \frac{\widetilde\rho_\pm(x,t)}{\widetilde s(x,t)}
\right|
& = \left| \frac{\rho_\pm(x,t)-\widetilde\rho_\pm(x,t)}{s(x,t)}
+ \widetilde\rho_\pm(x,t) \frac{f(x,t)-\widetilde f(x,t)}
{s(x,t)\widetilde s(x,t)} \right| \\
&\leq
\frac{2}{s_*}
|\rho_\pm(x,t)-\widetilde\rho_\pm(x,t)|
+ \frac{4R_1}{s_*^2}
|f(x,t)-\widetilde f(x,t)|\,.
\end{align*}
Therefore, setting
\[
C_A:=\frac{4M_v}{s_*}+ \frac{8M_vR_1}{s_*^2},
\]
the cancellation of the terms containing \(\partial_xv\), together
with~\eqref{eq:wp-velocity-regularity}
and~\eqref{eq:Ahat_pm}, gives
\begin{align}
\label{eq:A_difference}
\left|
\widehat A_\pm[u](\tau;x,t)
- \widehat A_\pm[\widetilde u](\tau;x,t) \right|
&=
2|v\bigl(X_\pm(\tau;x,t),\tau\bigr)|
\left|\frac{\rho_\mp\bigl(X_\pm(\tau;x,t),\tau\bigr)}{s\bigl(X_\pm(\tau;x,t),\tau\bigr)}
-\frac{\widetilde\rho_\mp\bigl(X_\pm(\tau;x,t),\tau\bigr)}{\widetilde s\bigl(X_\pm(\tau;x,t),\tau\bigr)}\right| \nonumber
\\
&\leq\frac{4M_v}{s_*}\|\rho_\mp-\widetilde\rho_\mp\|_{L^\infty(\mathbb R\times(0,\theta))}
+\frac{8M_vR_1}{s_*^2}\|f-\widetilde f\|_{L^\infty(\mathbb R\times(0,\theta))} \nonumber
\\
&\leq
C_A d_\theta(u,\widetilde u)
\end{align}
for a.e. \((\tau,x,t)\in(0,\theta)\times\mathbb R\times(0,\theta)\)
with \(\tau<t\). Moreover, \eqref{eq:wp-initiation}, \eqref{eq:KRT},
and~\eqref{eq:Ghat_pm} give
\begin{align}
\label{eq:G_difference}
\left|
\widehat G_\pm[u](\tau;x,t)
- \widehat G_\pm[\widetilde u](\tau;x,t) \right|
&=\left|I\bigl(X_\pm(\tau;x,t),\tau\bigr)\bigl(f-\widetilde f\bigr)
\bigl(X_\pm(\tau;x,t),\tau\bigr)\right|
\nonumber \\
&\leq
M_I\|f-\widetilde f\|_{L^\infty(\mathbb R\times(0,\theta))}\leq
M_I d_\theta(u,\widetilde u)
\end{align}
for a.e.\((\tau,x,t)\in(0,\theta)\times\mathbb R\times(0,\theta)\)
with \(\tau<t\), and the same equations further give
\[
\max\left\{
\left|\widehat G_\pm[u](\tau;x,t)\right|,
\left|\widehat G_\pm[\widetilde u](\tau;x,t)\right|
\right\}
\leq M_I
\]
for the same \((\tau,x,t)\). Next, applying the mean value theorem to~\eqref{eq:Ehat}, and using \eqref{eq:Ebound} for both inputs together with~\eqref{eq:A_difference}, we obtain
\begin{align}
\label{eq:E_difference}
\left|\widehat E_\pm[u](\sigma,t;x)
-\widehat E_\pm[\widetilde u](\sigma,t;x)\right|
&\leq e^{M_{\partial v}\theta} \int_\sigma^t
\left|\widehat A_\pm[u](\eta;x,t)
-\widehat A_\pm[\widetilde u](\eta;x,t)
\right|\rd\eta \nonumber \\
&\leq\theta e^{M_{\partial v}\theta} C_A d_\theta(u,\widetilde u)
\end{align}
for a.e. \((x,t)\in\mathbb R\times(0,\theta)\) and every
\(\sigma\in[0,t]\). Now, let
\begin{equation}
\label{eq:Lambda}
\Lambda(\theta) \coloneqq e^{M_{\partial v}\theta}\theta
\Bigl((R_0+\theta M_I)C_A+M_I \Bigr)\,.
\end{equation}
It follows from~\eqref{eq:wp-initial-forks},
\eqref{eq:Phi_plus}--\eqref{eq:Phi_minus},
\eqref{eq:Ebound}, \eqref{eq:G_difference},
and~\eqref{eq:E_difference}, together with
\(|\widehat G_\pm[\widetilde u]|\leq M_I\), that
\begin{align}
\label{eq:rho_difference}
&\left\|\widehat\rho_\pm[u]
-\widehat\rho_\pm[\widetilde u]\right\|_{L^\infty(\mathbb R\times(0,\theta))} \nonumber 
\\
&\quad\leq \left\|\rho_{\pm,0}\bigl(X_\pm(0;x,t)\bigr)
\left( \widehat E_\pm[u](0,t;x)
- \widehat E_\pm[\widetilde u](0,t;x) \right)   \right\|_{L^\infty(\mathbb R\times(0,\theta))}
\nonumber \\
&\qquad+ \left\|
\int_0^t \left( \widehat G_\pm[u](\sigma;x,t)
- \widehat G_\pm[\widetilde u](\sigma;x,t) \right)
\widehat E_\pm[u](\sigma,t;x)\,\rd\sigma
\right\|_{L^\infty(\mathbb R\times(0,\theta))}
\nonumber \\
&\qquad+ \left\| \int_0^t \widehat G_\pm[\widetilde u](\sigma;x,t)
\left( \widehat E_\pm[u](\sigma,t;x)
- \widehat E_\pm[\widetilde u](\sigma,t;x) \right) \rd\sigma
\right\|_{L^\infty(\mathbb R\times(0,\theta))}
\nonumber \\
&\quad\leq \Bigl(
R_0\theta e^{M_{\partial v}\theta}C_A
+
\theta M_Ie^{M_{\partial v}\theta}
+
\theta^2M_Ie^{M_{\partial v}\theta}C_A
\Bigr)
d_\theta(u,\widetilde u)
=
\Lambda(\theta)d_\theta(u,\widetilde u)\,.
\end{align}
Furthermore, \eqref{eq:wp-velocity-regularity},
\eqref{eq:Phi_f}, and~\eqref{eq:rho_difference} yield
\begin{align}
\label{eq:ef_bound}
\left\|
\widehat f[u]
- \widehat f[\widetilde u] \right\|_{L^\infty(\mathbb R\times(0,\theta))}
&\leq
M_v\theta\left(\left\|\widehat\rho_+[u]
- \widehat\rho_+[\widetilde u] \right\|_{L^\infty(\mathbb R\times(0,\theta))}
+ \left\|
\widehat\rho_-[u]
-\widehat\rho_-[\widetilde u]\right\|_{L^\infty(\mathbb R\times(0,\theta))}\right) \nonumber 
\\
&\leq
2M_v\theta\Lambda(\theta)d_\theta(u,\widetilde u)\,.
\end{align}
Consequently, defining
\begin{equation}
\label{eq:qT}
q(\theta)\coloneqq 2(1+M_v\theta)\Lambda(\theta),
\end{equation}
equations~\eqref{eq:metric}, \eqref{eq:rho_difference},
and~\eqref{eq:ef_bound} give
\begin{equation}
\label{eq:contraction}
d_\theta\bigl(\Phi_\theta(u), \Phi_\theta(\widetilde u)\bigr)
\leq q(\theta)d_\theta(u,\widetilde u)\,.
\end{equation}
Since \(C_A\) is independent of \(\theta\),
\eqref{eq:Lambda} and~\eqref{eq:qT} give
\[
\lim_{\theta \downarrow0}\Lambda(\theta)=0,
\quad
\lim_{\theta \downarrow0} q(\theta)=0,
\]
and therefore there exists \(T_2\in(0,T_1]\) such that
\[
q(\theta)<1
\qquad
\text{for every }0<\theta\leq T_2,
\]
which proves the assertion.
\end{proof}

\subsection{Local-in-time well-posedness}
We can now establish the local-in-time existence and uniqueness of mild solutions to~\eqref{eq:collapsed_system}--\eqref{eq:initial_data}.

\begin{thm}
\label{thm:local-well-posedness}
Fix \(T\in(0,\infty)\), assume~\eqref{eq:wp-assumptions}, and let \(R_1\) be given by Proposition~\ref{prop:self-mapping}. Then  there exists \(T_*\in(0,T]\) such that the Cauchy problem
\eqref{eq:collapsed_system}--\eqref{eq:initial_data} admits a unique
mild solution
\[
u=(\rho_+,\rho_-,f)\in K_{R_1,T_*}
\]
on \(\mathbb R\times[0,T_*]\), and hence
\[
0\leq\rho_\pm(x,t)\leq R_1\,,
\qquad
0\leq f(x,t)\leq1\,,
\qquad
1-f(x,t)\geq\frac{s_*}{2}
\]
for a.e. \((x,t)\in\mathbb R\times(0,T_*)\).
\end{thm}

\begin{proof}
Let \(T_*\coloneqq T_2\). By Proposition~\ref{prop:contraction},
\(\Phi_{T_*}\) is a strict contraction from \(K_{R_1,T_*}\) into
itself, and therefore the completeness of
\((K_{R_1,T_*},d_{T_*})\) and the Banach fixed-point theorem give a
unique fixed point
\[
u=(\rho_+,\rho_-,f)\in K_{R_1,T_*}.
\]
By~\eqref{eq:Phi_theta}--\eqref{eq:Phi_f}, this fixed-point identity
is equivalent to~\eqref{eq:duhamel_plus}--\eqref{eq:duhamel_f}, and
hence \(u\) is a mild solution in the sense of
Definition~\ref{def:mild_solution}.
It remains to prove uniqueness among all mild solutions, and therefore let
\[
\widetilde u
=(\widetilde\rho_+,\widetilde\rho_-,\widetilde f)
\]
be any mild solution on \(\mathbb R\times[0,T_*]\). By
Definition~\ref{def:mild_solution}, we have
\(\widetilde\rho_\pm\geq0\), \(0\leq\widetilde f\leq1\), and
\(1-\widetilde f>0\) a.e. on
\(\mathbb R\times(0,T_*)\), and hence the argument used in
Lemma~\ref{lem:selfmap-rho-bound} gives
\begin{align*}
0\leq\widetilde\rho_\pm(x,t) \leq
R_0e^{M_{\partial v}t} +M_I\int_0^t
e^{M_{\partial v}(t-\sigma)}\,\rd\sigma =B(t)\leq B(T)<R_1
\end{align*}
for a.e. \((x,t)\in\mathbb R\times(0,T_*)\), while
\eqref{eq:duhamel_f}, the monotonicity of \(B\), and
\(T_*=T_2\leq T_1\) yield
\[
1-\widetilde f(x,t)\geq
s_*-2M_v\int_0^tB(\tau)\,\rd\tau\geq
s_*-2M_vT_*B(T_*)\geq\frac{s_*}{2}\,.
\]
Thus \(\widetilde u\in K_{R_1,T_*}\), and comparison of
\eqref{eq:duhamel_plus}--\eqref{eq:duhamel_f} with
\eqref{eq:Phi_plus}--\eqref{eq:Phi_f} gives $\Phi_{T_*}(\widetilde u)=\widetilde u$, so that the uniqueness of the fixed point implies
\(\widetilde u=u\).
\end{proof}

The construction in Propositions~\ref{prop:self-mapping}
and~\ref{prop:contraction} depends on the initial data only through
a common \(L^\infty(\mathbb R)\)-bound for \(\rho_{\pm,0}\) and a
common positive lower bound for \(1-f_0\), and therefore yields the
following continuous-dependence estimate.

\begin{prop}
\label{prop:continuous_dependence}
Fix \(T\in(0,\infty)\), assume~\eqref{eq:wp-initiation}
and~\eqref{eq:wp-velocity-regularity},
and let
\(R_{\mathrm{in}}>0\) and \(s_*>0\). Suppose that
\[
u_0=(\rho_{+,0},\rho_{-,0},f_0)\,,
\quad \widetilde u_0
=(\widetilde\rho_{+,0},\widetilde\rho_{-,0},\widetilde f_0)
\]
belong to \(\bigl(L^\infty(\mathbb R)\bigr)^3\) and satisfy
\[
0\leq\rho_{\pm,0}(x)\leq R_{\mathrm{in}}\,,\quad 
0\leq\widetilde\rho_{\pm,0}(x)\leq R_{\mathrm{in}}\,,\quad
0\leq f_0(x)\leq1-s_*\,, \quad 
0\leq\widetilde f_0(x)\leq1-s_*
\]
for a.e. \(x\in\mathbb R\). Then there exist
\(R>R_{\mathrm{in}}\), \(T_c\in(0,T]\), and \(C>0\), depending only on $T$, $ R_{\mathrm{in}}$, $s_*$, $M_v$, $M_{\partial v}$, and $M_I$, such that the corresponding Cauchy problems admit unique mild
solutions
\[
u=(\rho_+,\rho_-,f)\,,
\quad
\widetilde u
= (\widetilde\rho_+,\widetilde\rho_-,\widetilde f)
\]
in \(K_{R,T_c}\), and
\begin{align}
\label{eq:cd}
d_{T_c}(u,\widetilde u)
\leq C\Bigl(\|\rho_{+,0}-\widetilde\rho_{+,0}\|_{L^\infty(\mathbb R)}+\|\rho_{-,0}-\widetilde\rho_{-,0}\|_{L^\infty(\mathbb R)}+
\|f_0-\widetilde f_0\|_{L^\infty(\mathbb R)}
\Bigr)\,.
\end{align}
\end{prop}
\begin{proof}
Setting
\[
R:=R_{\mathrm{in}}e^{M_{\partial v}T}+M_I\int_0^Te^{M_{\partial v}(T-\sigma)}\,\rd\sigma+1\,,
\]
the estimates in the proofs of
Propositions~\ref{prop:self-mapping}
and~\ref{prop:contraction}, with \(R_0\) replaced by
\(R_{\mathrm{in}}\) and \(R_1\) by \(R\), give
\(T_c\in(0,T]\) and \(q\in[0,1)\), depending only on
\(T\), \(R_{\mathrm{in}}\), \(s_*\), \(M_v\),
\(M_{\partial v}\), and \(M_I\), such that the two fixed-point maps are self-maps of \(K_{R,T_c}\) with contraction factor \(q\).
Letting \(\Phi_{u_0}\) and \(\Phi_{\widetilde u_0}\) denote these maps on \(K_{R,T_c}\), we obtain from the Banach fixed-point theorem and the uniqueness argument in Theorem~\ref{thm:local-well-posedness}
the corresponding unique mild solutions satisfying
\begin{equation}
 \label{eq:fpi}   
u=\Phi_{u_0}(u)\,,
\quad
\widetilde u=\Phi_{\widetilde u_0}(\widetilde u)\,.
\end{equation}
For every \(w\in K_{R,T_c}\),
\eqref{eq:Ahat_pm}--\eqref{eq:Ehat} show that the coefficients,
source terms, and exponential factors in the two fixed-point maps
agree when they are evaluated at \(w\), and using
\eqref{eq:wp-velocity-regularity}, \eqref{eq:metric},
\eqref{eq:Phi_plus}--\eqref{eq:Phi_f}, and~\eqref{eq:Ebound},
we obtain
\begin{align}
\label{eq:fpdd}
&d_{T_c}\bigl(\Phi_{u_0}(w),\Phi_{\widetilde u_0}(w)\bigr)\nonumber \\
&\leq
\Big( e^{M_{\partial v}T_c} + M_v\int_0^{T_c} e^{M_{\partial v}\tau}\,\rd\tau \Big)
\Bigl( \|\rho_{+,0}-\widetilde\rho_{+,0}\|_{L^\infty(\mathbb R)}
+
\|\rho_{-,0}-\widetilde\rho_{-,0}\|_{L^\infty(\mathbb R)} \Bigr)
\nonumber \\
&\qquad + \|f_0-\widetilde f_0\|_{L^\infty(\mathbb R)}
\nonumber \\
&\leq
(1+M_vT_c)e^{M_{\partial v}T_c}\Bigl(\|\rho_{+,0}-\widetilde\rho_{+,0}\|_{L^\infty(\mathbb R)}
+\|\rho_{-,0}-\widetilde\rho_{-,0}\|_{L^\infty(\mathbb R)}\Bigr)
+ \|f_0-\widetilde f_0\|_{L^\infty(\mathbb R)}
\nonumber
\\
&\leq
(1+M_vT_c)e^{M_{\partial v}T_c}
\Bigl(
\|\rho_{+,0}-\widetilde\rho_{+,0}\|_{L^\infty(\mathbb R)}
+
\|\rho_{-,0}-\widetilde\rho_{-,0}\|_{L^\infty(\mathbb R)}
+
\|f_0-\widetilde f_0\|_{L^\infty(\mathbb R)}
\Bigr)\,,
\end{align}
where we used
\[
\int_0^{T_c}e^{M_{\partial v}\tau}\,\rd\tau
\leq T_ce^{M_{\partial v}T_c}
\quad\text{and}\quad
1\leq(1+M_vT_c)e^{M_{\partial v}T_c}
\]
for the second and last inequalities, respectively.
Using~\eqref{eq:fpi}, the triangle inequality, and
the contraction estimate~\eqref{eq:contraction} with contraction
factor \(q\), we obtain
\begin{align*}
d_{T_c}(u,\widetilde u)
&=
d_{T_c}\bigl(
\Phi_{u_0}(u),
\Phi_{\widetilde u_0}(\widetilde u)
\bigr)
\leq
d_{T_c}\bigl(
\Phi_{u_0}(u),
\Phi_{u_0}(\widetilde u)
\bigr)
+
d_{T_c}\bigl(
\Phi_{u_0}(\widetilde u),
\Phi_{\widetilde u_0}(\widetilde u)
\bigr)
\\
&\leq
q d_{T_c}(u,\widetilde u)
+
d_{T_c}\bigl(
\Phi_{u_0}(\widetilde u),
\Phi_{\widetilde u_0}(\widetilde u)
\bigr),
\end{align*}
and hence
\[
(1-q)d_{T_c}(u,\widetilde u)
\leq
d_{T_c}\bigl(
\Phi_{u_0}(\widetilde u),
\Phi_{\widetilde u_0}(\widetilde u)
\bigr),
\]
and consequently~\eqref{eq:fpdd} proves
\eqref{eq:cd} with
\[
C:=
\frac{(1+M_vT_c)e^{M_{\partial v}T_c}}{1-q}.
\]
\end{proof}

\subsection{A priori estimates and maximal continuation}

We now derive the a priori estimates that allow us to extend the
local-in-time mild solution obtained in
Theorem~\ref{thm:local-well-posedness} uniquely to its maximal interval
of existence \([0,T_{\max})\) and show that \(T_{\max}<\infty\) can
occur only if the unreplicated fraction \(s=1-f\) is no longer
uniformly bounded away from zero. To this end, we first establish the continuity
properties of mild solutions with respect to time and show that
\eqref{eq:duhamel_plus}--\eqref{eq:duhamel_f} remain valid after
restarting the solution at every time in the interval of existence.

\begin{lem}
\label{lem:temporal-traces-restart}
Fix \(T\in(0,\infty)\), assume
\eqref{eq:wp-initiation}--\eqref{eq:wp-initial-fraction}, let
\(T'\in(0,T]\), and let \((\rho_+,\rho_-,f)\) be a mild solution on
\(\mathbb R\times[0,T']\). Then \(\rho_\pm\) and \(f\) admit
representatives, denoted by the same symbols, such that
\begin{equation}
\label{eq:continuity_in_time}
\rho_\pm \in C_{w^\ast}([0,T'];L^\infty(\mathbb R)) \quad \text{and}
\quad f \in C([0,T'];L^\infty(\mathbb R))\,,
\end{equation}
where \(C_{w^\ast}\) denotes weak-star continuity, and
\[
\rho_\pm(\cdot,0)=\rho_{\pm,0} \quad \text{and} \quad f(\cdot,0)=f_0
\qquad \text{in }L^\infty(\mathbb R)\,.
\]
For every \(\theta\in[0,T')\), every \(t\in[\theta,T']\), and
for a.e. \(x\in\mathbb R\), these functions satisfy
\begin{align}
\label{eq:restarted-rho-duhamel}
\rho_\pm(x,t) &= \rho_\pm (X_\pm(\theta;x,t),\theta)
\exp\big( -\int_\theta^tA_\pm(\eta;x,t)\,\rd\eta \big) \nonumber \\
&\qquad + \int_\theta^t G_\pm(\sigma;x,t) \exp \big( -\int_\sigma^tA_\pm(\eta;x,t)\,\rd\eta
\big)\,\rd\sigma
\end{align}
and
\begin{equation}
\label{eq:restarted-f-duhamel}
f(x,t)
= f(x,\theta) +\int_\theta^t v(x,\sigma)
\big(\rho_+(x,\sigma)+\rho_-(x,\sigma)\big)\,\rd\sigma \,.
\end{equation}
Moreover, the inequalities in Definition~\ref{def:mild_solution} hold for every
\(t\in[0,T']\) and for a.e. \(x\in\mathbb R\), and any two mild solutions with the same initial data agree on their common interval of existence.
\end{lem}

\begin{proof}
Let \(Y_\pm(\cdot;y)\) be the forward characteristics introduced in
the proof of Proposition~\ref{prop:mild-implies-distributional}. The
flow identity~\eqref{eq:flow_identity} and the inverse relation
established there give
\[
X_\pm(\sigma;Y_\pm(t;y),t)=Y_\pm(\sigma;y),
\qquad
Y_\pm(t;\cdot)^{-1}(x)=X_\pm(0;x,t)
\]
for every \(0\leq\sigma\leq t\leq T'\). Consequently, substituting
\(x=Y_\pm(t;y)\) in
\eqref{eq:duhamel_plus}--\eqref{eq:duhamel_minus} gives
\begin{align*}
&\rho_\pm(Y_\pm(t;y),t)
\\
&\quad=
\rho_{\pm,0}(y)
\exp\!\left(
-\int_0^t
\left[
\pm(\partial_xv)(Y_\pm(\eta;y),\eta)
+\frac{2v(Y_\pm(\eta;y),\eta)}
{s(Y_\pm(\eta;y),\eta)}
\rho_\mp(Y_\pm(\eta;y),\eta)
\right]\,\rd\eta
\right)
\\
&\qquad+
\int_0^t
I(Y_\pm(\sigma;y),\sigma)
s(Y_\pm(\sigma;y),\sigma)
\\
&\hspace{1.5cm}\times
\exp\!\left(
-\int_\sigma^t
\left[
\pm(\partial_xv)(Y_\pm(\eta;y),\eta)
+\frac{2v(Y_\pm(\eta;y),\eta)}
{s(Y_\pm(\eta;y),\eta)}
\rho_\mp(Y_\pm(\eta;y),\eta)
\right]\,\rd\eta
\right)
\,\rd\sigma
\end{align*}
for a.e. \((y,t)\in\mathbb R\times(0,T')\). Indeed, the
bi-Lipschitz property of \(Y_\pm(t;\cdot)\), the preservation of null
sets discussed after~\eqref{eq:char_lipschitz}, and Fubini's theorem
show that the full flow pullbacks preserve the a.e. identities and
inequalities in Definition~\ref{def:mild_solution}.

For a.e. \(y\in\mathbb R\), the functions occurring in the time
integrals in the preceding formula belong to \(L^\infty(0,T')\).
Another application of Fubini's theorem therefore gives one
full-measure set of labels such that, for every label in this set,
the preceding identity holds for a.e. \(t\in(0,T')\) and its
right-hand side defines an absolutely continuous function on
\([0,T']\). We take this right-hand side as the representative of
\(t\mapsto\rho_\pm(Y_\pm(t;y),t)\) for every \(t\in[0,T']\).
It satisfies
\begin{align*}
&\frac{\rd}{\rd t}\rho_\pm(Y_\pm(t;y),t)
\\
&\quad+
\left[
\pm(\partial_xv)(Y_\pm(t;y),t)
+\frac{2v(Y_\pm(t;y),t)}{s(Y_\pm(t;y),t)}
\rho_\mp(Y_\pm(t;y),t)
\right]
\rho_\pm(Y_\pm(t;y),t)
\\
&\qquad=
I(Y_\pm(t;y),t)s(Y_\pm(t;y),t)
\end{align*}
for a.e. \(y\in\mathbb R\) and for a.e. \(t\in(0,T')\), and
\[
\rho_\pm(Y_\pm(0;y),0)=\rho_{\pm,0}(y)
\qquad\text{for a.e. }y\in\mathbb R\,.
\]
The common full-measure set of labels can be chosen so that all the
preceding conclusions hold simultaneously for the two signs and for
all the coefficient pullbacks appearing above.

Moreover, Definition~\ref{def:mild_solution} and
\eqref{eq:wp-initiation}--\eqref{eq:wp-velocity-regularity} imply
\[
\pm(\partial_xv)(Y_\pm(\eta;y),\eta)
+\frac{2v(Y_\pm(\eta;y),\eta)}
{s(Y_\pm(\eta;y),\eta)}
\rho_\mp(Y_\pm(\eta;y),\eta)
\geq-M_{\partial v}
\]
for a.e. \((y,\eta)\in\mathbb R\times(0,T')\), while
\[
0\leq
I(Y_\pm(\sigma;y),\sigma)
s(Y_\pm(\sigma;y),\sigma)
\leq M_I
\]
for a.e. \((y,\sigma)\in\mathbb R\times(0,T')\). It follows from
the characteristic representation and~\eqref{eq:BofT} that
\[
0\leq\rho_\pm(Y_\pm(t;y),t)
\leq
R_0e^{M_{\partial v}t}
+M_I\int_0^t e^{M_{\partial v}(t-\sigma)}\,\rd\sigma
=B(t)\leq B(T')
\]
for every \(t\in[0,T']\) and for a.e. \(y\in\mathbb R\). For each
\(t\in[0,T']\), we define \(\rho_\pm(x,t)\) by evaluating the
preceding characteristic representative at
\(y=X_\pm(0;x,t)\). The joint measurability of the flow and of the
right-hand side above shows that this gives a jointly measurable
representative, and the bi-Lipschitz inverse relation shows that
\(\rho_\pm(\cdot,t)\in L^\infty(\mathbb R)\). The resulting
space-time representatives agree with the original mild solution
for a.e. \((x,t)\in\mathbb R\times(0,T')\), and
\[
\|\rho_\pm(\cdot,t)\|_{L^\infty(\mathbb R)}
\leq B(T')
\qquad\text{for every }t\in[0,T']\,.
\]

We next prove weak-star continuity. Let
\(\zeta\in C_c^\infty(\mathbb R)\), and choose \(K>0\) such that
\(\operatorname{supp}\zeta\subset[-K,K]\). The change of variables
\[
x=Y_\pm(t;y),
\qquad
\rd x=\partial_yY_\pm(t;y)\,\rd y
\]
gives
\begin{align*}
\int_{\mathbb R}\rho_\pm(x,t)\zeta(x)\,\rd x
=\int_{\mathbb R}
\rho_\pm(Y_\pm(t;y),t)
\zeta(Y_\pm(t;y))
\partial_yY_\pm(t;y)\,\rd y\,.
\end{align*}
For a.e. \(y\in\mathbb R\), the first factor in the integrand is
absolutely continuous in \(t\), the second is continuous in \(t\),
and~\eqref{eq:forward_jacobian} provides the representative
\[
\partial_yY_\pm(t;y)
=
\exp\!\left(
\pm\int_0^t
(\partial_xv)(Y_\pm(r;y),r)\,\rd r
\right),
\]
which is also absolutely continuous in \(t\). Here
\(\partial_yY_\pm(t;\cdot)\) is understood through this
time-continuous representative; for every fixed \(t\), it agrees
with the spatial derivative for a.e. \(y\), which is sufficient for
the change of variables above. Moreover,
\eqref{eq:forward_displacement} shows that
\(\zeta(Y_\pm(t;y))\neq0\) implies
\(|y|\leq K+M_vT'\), and~\eqref{eq:forward_jacobian} gives
\[
\begin{aligned}
&\left|
\rho_\pm(Y_\pm(t;y),t)
\zeta(Y_\pm(t;y))
\partial_yY_\pm(t;y)
\right|
\\
&\qquad\leq
B(T')\|\zeta\|_{L^\infty(\mathbb R)}e^{M_{\partial v}T'}
\mathbf 1_{[-K-M_vT',K+M_vT']}(y)
\end{aligned}
\]
for every \(t\in[0,T']\) and for a.e. \(y\in\mathbb R\). The
right-hand side belongs to \(L^1(\mathbb R)\), and hence the
dominated convergence theorem proves that
\[
t\longmapsto
\int_{\mathbb R}\rho_\pm(x,t)\zeta(x)\,\rd x
\]
is continuous on \([0,T']\).

Now let \(\zeta\in L^1(\mathbb R)\). By the density of
\(C_c^\infty(\mathbb R)\) in \(L^1(\mathbb R)\), choose
\(\zeta_n\in C_c^\infty(\mathbb R)\) such that
\[
\|\zeta_n-\zeta\|_{L^1(\mathbb R)}\longrightarrow0\,.
\]
For \(r,t\in[0,T']\), the uniform bound obtained above gives
\begin{align*}
&\left|
\int_{\mathbb R}
\bigl(\rho_\pm(x,t)-\rho_\pm(x,r)\bigr)
\zeta(x)\,\rd x
\right|
\\
&\quad\leq
\left|
\int_{\mathbb R}
\bigl(\rho_\pm(x,t)-\rho_\pm(x,r)\bigr)
\zeta_n(x)\,\rd x
\right|
+2B(T')\|\zeta-\zeta_n\|_{L^1(\mathbb R)}\,.
\end{align*}
For fixed \(n\), the first term converges to zero as \(t\to r\).
Choosing \(n\) first and then letting \(t\to r\) proves
\[
\rho_\pm\in
C_{w^\ast}\bigl([0,T'];L^\infty(\mathbb R)\bigr)\,.
\]
Since \(Y_\pm(0;y)=y\), the characteristic representatives at
\(t=0\) satisfy
\[
\rho_\pm(\cdot,0)=\rho_{\pm,0}
\qquad\text{in }L^\infty(\mathbb R)\,.
\]

For \(f\), we take, for every \(t\in[0,T']\), the representative
defined by
\[
f(x,t)=f_0(x)+\int_0^t
v(x,\sigma)
\bigl(\rho_+(x,\sigma)+\rho_-(x,\sigma)\bigr)\,\rd\sigma
\]
for a.e. \(x\in\mathbb R\). Since the temporal representatives of
\(\rho_\pm\) agree with the original mild solution for a.e.
\((x,t)\), this representative agrees with the original \(f\) for
a.e. \((x,t)\in\mathbb R\times(0,T')\). For every
\(r,t\in[0,T']\),
\[
\|f(\cdot,t)-f(\cdot,r)\|_{L^\infty(\mathbb R)}
\leq 2M_vB(T')|t-r|\,,
\]
and therefore
\[
f\in C\bigl([0,T'];L^\infty(\mathbb R)\bigr),
\qquad
f(\cdot,0)=f_0
\quad\text{in }L^\infty(\mathbb R)\,.
\]
Subtracting the defining identities at \(t\) and at
\(\theta\in[0,t]\) gives~\eqref{eq:restarted-f-duhamel}.

Since the representatives just constructed agree with the original
mild solution for a.e. \((x,t)\), the preservation of space-time
null sets under the full flows and Fubini's theorem show that, for
a.e. \(y\), their pullbacks along either \(Y_+\) or \(Y_-\) agree
with the corresponding original pullbacks for a.e. time. Hence
replacing the original space-time representatives by the temporal
representatives in the characteristic coefficients and source terms
does not change any of the preceding time integrals or
characteristic equations.

It remains to prove the restarted formulae for the fork densities.
Fix \(0\leq\theta\leq t\leq T'\), set
\(x=Y_\pm(t;y)\), and use~\eqref{eq:flow_identity} to obtain
\[
Y_\pm(\sigma;y)=X_\pm(\sigma;x,t)
\qquad\text{for every }\sigma\in[0,t]\,.
\]
For every label in the full-measure set chosen above, the
characteristic equation obtained previously therefore becomes
\begin{align*}
&\frac{\rd}{\rd\sigma}
\rho_\pm(X_\pm(\sigma;x,t),\sigma)
+A_\pm(\sigma;x,t)
\rho_\pm(X_\pm(\sigma;x,t),\sigma)
=G_\pm(\sigma;x,t)
\end{align*}
for a.e. \(\sigma\in(\theta,t)\). Integrating this identity from
\(\theta\) to \(t\) after multiplying by the integrating factor gives
\[
\begin{aligned}
&\rho_\pm(x,t)
\exp\!\left(\int_\theta^t A_\pm(\eta;x,t)\,\rd\eta\right)
-\rho_\pm(X_\pm(\theta;x,t),\theta)
\\
&\qquad=
\int_\theta^t G_\pm(\sigma;x,t)
\exp\!\left(\int_\theta^\sigma
A_\pm(\eta;x,t)\,\rd\eta\right)\,\rd\sigma\,.
\end{aligned}
\]
Multiplying by
\(\exp(-\int_\theta^tA_\pm(\eta;x,t)\,\rd\eta)\) yields
\begin{align*}
\rho_\pm(x,t)
={}&\rho_\pm(X_\pm(\theta;x,t),\theta)
\exp\!\left(-\int_\theta^t
A_\pm(\eta;x,t)\,\rd\eta\right)
\\
&+\int_\theta^t G_\pm(\sigma;x,t)
\exp\!\left(-\int_\sigma^t
A_\pm(\eta;x,t)\,\rd\eta\right)\,\rd\sigma\,.
\end{align*}
Since the full-measure set of labels is independent of
\(\theta\) and \(t\), and \(Y_\pm(t;\cdot)\) preserves null sets,
this proves~\eqref{eq:restarted-rho-duhamel} for every
\(\theta\in[0,T')\), every \(t\in[\theta,T']\), and for a.e.
\(x\in\mathbb R\).

We next extend the inequalities in Definition~\ref{def:mild_solution}
to every temporal slice. There is a full-measure set of times in
\((0,T')\) at which these inequalities hold for a.e.
\(x\in\mathbb R\) and the temporal representatives agree with the
original mild solution. Fix \(t\in[0,T']\), and choose a sequence
of such times \(t_n\to t\). For every non-negative
\(\zeta\in L^1(\mathbb R)\), weak-star continuity gives
\[
\int_{\mathbb R}\rho_\pm(x,t)\zeta(x)\,\rd x
=\lim_{n\to\infty}
\int_{\mathbb R}\rho_\pm(x,t_n)\zeta(x)\,\rd x
\geq0\,.
\]
Hence \(\rho_\pm(x,t)\geq0\) for a.e.
\(x\in\mathbb R\). Moreover, strong \(L^\infty\)-continuity of
\(f\) gives
\[
\|f(\cdot,t_n)-f(\cdot,t)\|_{L^\infty(\mathbb R)}
\longrightarrow0\,.
\]
Passing to the limit in the inequalities satisfied at \(t_n\)
therefore gives
\[
0\leq f(x,t)\leq1
\]
and
\[
1-f(x,t)
\geq
\operatorname*{ess\,inf}_{(x,r)\in\mathbb R\times(0,T')}
s(x,r)>0
\]
for a.e. \(x\in\mathbb R\). Thus all the inequalities in
Definition~\ref{def:mild_solution} hold for every
\(t\in[0,T']\).

Finally, let
\[
(\rho_+,\rho_-,f)
\qquad\text{and}\qquad
(\widetilde\rho_+,\widetilde\rho_-,\widetilde f)
\]
be two mild solutions with the same initial data on a common
interval, which we denote by \([0,T']\), and set
\(\widetilde s:=1-\widetilde f\). Equip both solutions with the
representatives constructed above. The every-time lower bounds at
\(t=0\) give
\[
s_0(x)\geq
\min\left\{
\operatorname*{ess\,inf}_{(x,t)\in\mathbb R\times(0,T')}s(x,t),
\operatorname*{ess\,inf}_{(x,t)\in\mathbb R\times(0,T')}
\widetilde s(x,t)
\right\}>0
\]
for a.e. \(x\in\mathbb R\). Thus the common initial data satisfy all
the assumptions of Theorem~\ref{thm:local-well-posedness}, and that
theorem shows that the two solutions agree on a nontrivial interval
beginning at \(0\).

Let
\[
\theta:=
\sup\left\{
r\in[0,T']:
\begin{array}{l}
\rho_\pm(\cdot,t)=\widetilde\rho_\pm(\cdot,t)
\text{ in }L^\infty(\mathbb R),\\
f(\cdot,t)=\widetilde f(\cdot,t)
\text{ in }L^\infty(\mathbb R)
\text{ for every }t\in[0,r]
\end{array}
\right\}.
\]
If \(\theta<T'\), choose \(t_n\uparrow\theta\) such that the two
representatives agree at each \(t_n\). Weak-star continuity of the
fork densities and strong \(L^\infty\)-continuity of the replicated
fractions then give
\[
\rho_\pm(\cdot,\theta)
=\widetilde\rho_\pm(\cdot,\theta),
\qquad
f(\cdot,\theta)=\widetilde f(\cdot,\theta)
\quad\text{in }L^\infty(\mathbb R)\,.
\]
The inequalities already proved show that this common slice
consists of non-negative fork densities and a non-negative
replicated fraction whose unreplicated fraction is bounded away
from zero. Moreover, the shifted coefficients
\(I(x,\theta+\cdot)\) and \(v(x,\theta+\cdot)\) satisfy the same
assumptions as \(I\) and \(v\), while the restarted formulae
\eqref{eq:restarted-rho-duhamel}--\eqref{eq:restarted-f-duhamel}
and the every-time inequalities show that the restrictions of the
two solutions after \(\theta\) are mild solutions of the same
time-shifted problem with this common initial value. The uniqueness
assertion of Theorem~\ref{thm:local-well-posedness}, applied to this
time-shifted problem on a sufficiently short interval, therefore
extends their equality beyond \(\theta\). This contradicts the
definition of \(\theta\), and hence \(\theta=T'\), which proves
uniqueness on the common interval and completes the proof.
\end{proof}

We can now use Lemma~\ref{lem:temporal-traces-restart} to derive the
a priori estimates needed for the maximal continuation argument.

\begin{prop}
\label{prop:apriori-bounds}
Fix \(T\in(0,\infty)\), assume~\eqref{eq:wp-assumptions}, let
\(T'\in(0,T]\), and let \((\rho_+,\rho_-,f)\) be a mild solution of
\eqref{eq:collapsed_system}--\eqref{eq:initial_data} on \(\mathbb R\times[0,T']\). Then
\begin{align*}
\|\rho_\pm(\cdot,t)\|_{L^\infty(\mathbb R)} \leq B(t) \quad \text{for every }t\in[0,T'] 
\end{align*}
and 
\begin{align*}
s(x,t)\geq s_*-2M_v\int_0^tB(\tau)\,\rd\tau \quad \text{for a.e. }x\in\mathbb R \text{ and every }t\in[0,T']\,,
\end{align*}
where \(B\) is given by~\eqref{eq:BofT}.
\end{prop}

\begin{proof}
By Lemma~\ref{lem:temporal-traces-restart}, the inequalities in
Definition~\ref{def:mild_solution} hold for every \(t\in[0,T']\) and
for a.e. \(x\in\mathbb R\). Hence
\eqref{eq:wp-velocity-regularity} and the definitions of
\(A_\pm\) in Subsection~\ref{subsec:characteristics-mild-formulation}
give \( A_\pm(\eta;x,t)\geq-M_{\partial v} \) for every \(t\in(0,T']\) and for a.e. \((\eta,x)\in(0,t)\times\mathbb R\),  while
\eqref{eq:wp-initiation}, the definition of \(G_\pm\) in
Subsection~\ref{subsec:characteristics-mild-formulation}, and
\(0<s(x,t)\leq1\) give \( 0\leq G_\pm(\sigma;x,t)\leq M_I \)
for every \(t\in(0,T']\) and for a.e. \((\sigma,x)\in(0,t)\times\mathbb R\). Therefore, \eqref{eq:wp-initial-forks} and \eqref{eq:restarted-rho-duhamel} with \(\theta=0\) yield
\begin{align*}
0\leq\rho_\pm(x,t)\leq R_0e^{M_{\partial v}t} +M_I\int_0^t e^{M_{\partial v}(t-\sigma)}\,\rd\sigma=B(t)
\end{align*}
for every \(t\in[0,T']\) and for a.e. \(x\in\mathbb R\), which
proves the asserted bound for \(\rho_\pm\), and, in combination with
\eqref{eq:wp-velocity-regularity}, \eqref{eq:wp-initial-gap},
and~\eqref{eq:restarted-f-duhamel} with \(\theta=0\), we further obtain
\begin{align*}
s(x,t)=s_0(x)-\int_0^t v(x,\tau)(\rho_+(x,\tau)+\rho_-(x,\tau))\,\rd\tau \geq
s_*-2M_v\int_0^tB(\tau)\,\rd\tau\,,
\end{align*}
which proves the asserted lower bound for \(s\).
\end{proof}

Lemma~\ref{lem:temporal-traces-restart} and
Proposition~\ref{prop:apriori-bounds} now yield the following
maximal continuation criterion.

\begin{thm}
\label{thm:maximal-continuation}
Assume~\eqref{eq:wp-initial-forks}--\eqref{eq:wp-initial-gap}, and
suppose that \eqref{eq:wp-initiation}
and~\eqref{eq:wp-velocity-regularity} hold for every finite
\(T>0\). Then there exist \(T_{\max}\in(0,\infty]\) and a unique
maximal mild solution of~\eqref{eq:collapsed_system}--
\eqref{eq:initial_data} on \(\mathbb R\times[0,T_{\max})\), in the
sense that its restriction to \(\mathbb R\times[0,T']\) is a mild
solution according to Definition~\ref{def:mild_solution} for every
\(T'\in(0,T_{\max})\). Moreover, if \(T_{\max}<\infty\), then
\begin{equation}
\label{eq:blowup_criterion}
\liminf_{t\uparrow T_{\max}}
\operatorname*{ess\,inf}_{x\in\mathbb R}s(x,t)=0.
\end{equation}
\end{thm}

\begin{proof}
Theorem~\ref{thm:local-well-posedness} gives a mild solution on a nontrivial time interval. Let \(T_{\max}\) be the supremum of all times for which a mild solution exists. By Lemma~\ref{lem:temporal-traces-restart}, any two such solutions agree
on their common interval of existence, and therefore determine a
unique solution on
\([0,T_{\max})\).

Suppose that \(T_{\max}<\infty\) and that
\eqref{eq:blowup_criterion} does not hold. Then there exist
\(\delta>0\) and \(t_0\in[0,T_{\max})\) such that
\[
s(x,t)\geq\delta \quad \text{for a.e. }x\in\mathbb R \text{ and every }t\in[t_0,T_{\max})\,.
\]
Fix \(T>T_{\max}\), take the constants \(M_I\), \(M_v\), and \(M_{\partial v}\) introduced at the beginning of Section~\ref{Sec3} with respect to
\(\mathbb R\times(0,T)\), and let \(B\) be given by \eqref{eq:BofT} of Lemma~\ref{lem:selfmap-rho-bound}
with these constants. For every \(\tau\in[t_0,T_{\max})\), set
\[
I_\tau(x,r):=I(x,\tau+r)\,, \quad v_\tau(x,r):=v(x,\tau+r)
\]
for \((x,r)\in\mathbb R\times(0,T-T_{\max})\).
Since \(0<\tau+r<T\) on this interval, \eqref{eq:wp-initiation} and~\eqref{eq:wp-velocity-regularity} give
\[
\|I_\tau\|_{L^\infty(\mathbb R\times(0,T-T_{\max}))}\leq M_I\,,\quad 
\|v_\tau\|_{L^\infty(\mathbb R\times(0,T-T_{\max}))}\leq M_v\,,\quad \|\partial_xv_\tau\|_{L^\infty(\mathbb R\times(0,T-T_{\max}))}\leq M_{\partial v}\,,
\]
and \(I_\tau(x,r)\geq0\) and \(v_\tau(x,r)\geq0\) for a.e.  \((x,r)\in\mathbb R\times(0,T-T_{\max})\), while
Lemma~\ref{lem:temporal-traces-restart} and 
Proposition~\ref{prop:apriori-bounds}, together with the monotonicity of \(B\) and the lower bound for \(s\), yield 
\[
0\leq\rho_\pm(x,\tau) \leq B(\tau) \leq B(T_{\max})\,,
\quad 0\leq f(x,\tau)\leq1-\delta
\]
for every \(\tau\in[t_0,T_{\max})\) and for a.e. \(x\in\mathbb R\). Consequently, the choices in
Propositions~\ref{prop:self-mapping} and~\ref{prop:contraction},
applied to the shifted coefficients \(I_\tau\) and \(v_\tau\) and the initial data
\((\rho_+(\cdot,\tau),\rho_-(\cdot,\tau),f(\cdot,\tau))\),
give a local existence time \(h_0>0\) independent of \(\tau\).  Set
\[
h:=\min\left\{h_0,\frac{T-T_{\max}}2\right\}>0 \quad \text{ and choose } \quad
\tau\in(\max\{t_0,T_{\max}-h\},T_{\max})\,.
\]
Then $\tau+h>T_{\max}$, $\tau+h<T$, and Theorem~\ref{thm:local-well-posedness}, applied to the time-shifted problem with \(r=t-\tau\), gives a mild solution for \(r\in[0,h]\), or equivalently for \(t=\tau+r\in[\tau,\tau+h]\), with initial data \((\rho_+(\cdot,\tau),
\rho_-(\cdot,\tau),f(\cdot,\tau)\bigr)\).
Lemma~\ref{lem:temporal-traces-restart} shows that this solution
agrees with the original solution on every compact subinterval of
\([\tau,T_{\max})\), and hence the two solutions define an extension
beyond \(T_{\max}\), contradicting its maximality. This
proves~\eqref{eq:blowup_criterion}.
\end{proof}

\begin{rem}
\label{rem:criterion_for_well-posedness}
The nonlinear coalescence term in the first two equations
of~\eqref{eq:collapsed_system} contains the singular factor \(1/s\), and the continuation criterion~\eqref{eq:blowup_criterion} in Theorem~\ref{thm:maximal-continuation} shows that
\(T_{\max}<\infty\) can occur only if the unreplicated fraction \(s\) is no longer uniformly bounded away from zero as
\(t\uparrow T_{\max}\).
\end{rem}

\subsection{Normalised fork densities and global-in-time  existence}

We now connect the local-in-time well-posedness result in
Theorem~\ref{thm:local-well-posedness} and the maximal-continuation
criterion in Theorem~\ref{thm:maximal-continuation} with the
replication-time analysis in Section~\ref{Sec4} by showing that, for
initial data satisfying~\eqref{eq:compatible_data}, normalising the
fork densities by the unreplicated fraction \(s\) cancels the
nonlinear coalescence term in the first two equations
of~\eqref{eq:collapsed_system} and gives the linear transport
system~\eqref{eq:normalised-system} together with the
representations~\eqref{eq:normalised-characteristic-representation}
and~\eqref{eq:s-normalised-representation} below, from which we derive the
direct decay estimate for \(s\) in
Corollary~\ref{cor:direct-s-bound} and the global-in-time existence
result in Theorem~\ref{thm:global-compatible}, both stated below.

\begin{prop}
\label{prop:normalised-equations}
Fix \(T\in(0,\infty)\), assume~\eqref{eq:compatible_data}
and~\eqref{eq:wp-assumptions}, and let
\((\rho_+,\rho_-,f)\) be a mild solution of
\eqref{eq:collapsed_system}--\eqref{eq:initial_data} on
\(\mathbb R\times[0,T]\). Then the functions \(z_\pm\) and
\(z_{\pm,0}\), defined by
\[
z_\pm(x,t):=\frac{\rho_\pm(x,t)}{s(x,t)} \quad\text{and}\quad z_{\pm,0}(x):=\frac{\rho_{\pm,0}(x)}{s_0(x)} 
\]
for a.e. \((x,t)\in\mathbb R\times(0,T)\) and for a.e. 
\(x\in\mathbb R\), respectively, satisfy
\[
s\in W_{\mathrm{loc}}^{1,\infty}(\mathbb R\times(0,T))\,,
\qquad z_\pm\in L^\infty(\mathbb R\times(0,T))\,,
\]
and \(z_\pm\) admit representatives, denoted by the same symbols,
such that
\[
z_\pm \in C_{w^\ast}([0,T];L^\infty(\mathbb R))\,,
\quad z_\pm(\cdot,0)=z_{\pm,0}
\quad\text{in }L^\infty(\mathbb R)\,.
\]
Moreover,
\begin{equation}
\label{eq:normalised-system}
\left\{
\begin{aligned}
\partial_tz_+(x,t) +\partial_x\big(v(x,t)z_+(x,t)\big) &=I(x,t)\,,\\
\partial_tz_-(x,t) -\partial_x\big(v(x,t)z_-(x,t)\big) &=I(x,t)
\end{aligned}
\right.
\end{equation}
in \(\mathcal D'(\mathbb R\times(0,T))\), and, for every
\(t\in[0,T]\) and for a.e. \(x\in\mathbb R\),
\begin{align}
\label{eq:normalised-characteristic-representation}
z_\pm(x,t) &= z_{\pm,0}(X_\pm(0;x,t))
\exp\big( \mp\int_0^t (\partial_xv)(X_\pm(\eta;x,t),\eta)\,\rd\eta
\big) \nonumber \\
&\qquad + \int_0^t I(X_\pm(\sigma;x,t),\sigma)
\exp\big( \mp\int_\sigma^t (\partial_xv)(X_\pm(\eta;x,t),\eta)\,\rd\eta
\big)\,\rd\sigma \,.
\end{align}
Finally,
\begin{equation}
\label{eq:s-normalised-representation}
s(x,t) =
s_0(x) \exp\big(
-\int_0^t v(x,\tau)(z_+(x,\tau)+z_-(x,\tau))\,\rd\tau
\big)
\end{equation}
for a.e. \(x\in\mathbb R\) and every \(t\in[0,T]\).
\end{prop}

\begin{proof}
By Proposition~\ref{prop:mild-implies-distributional}, the mild solution satisfies the size-collapsed system~\eqref{eq:collapsed_system} in
\(\mathcal D'(\mathbb R\times(0,T))\), and the trace convergences of
\(\rho_\pm\) and \(f\) established there yield
\[
\rho_\pm(\cdot,t)\longrightarrow\rho_{\pm,0}\,,
\qquad
f(\cdot,t)\longrightarrow f_0
\quad\text{in }\mathcal D'(\mathbb R)
\qquad
\text{as }t\downarrow0\,,
\]
and hence Lemma~\ref{lem:compatibility-propagation}, together with the
third equation in~\eqref{eq:collapsed_system}, gives
\[
\partial_x s(x,t) =\rho_+(x,t)-\rho_-(x,t)\,,
\quad
\partial_ts(x,t)
=-v(x,t)(\rho_+(x,t)+\rho_-(x,t)),
\]
$\text{in }\mathcal D'(\mathbb R\times(0,T))$ and, since the right-hand sides belong to
\(L^\infty(\mathbb R\times(0,T))\), it follows that
\[
s\in
W_{\mathrm{loc}}^{1,\infty} (\mathbb R\times(0,T))\,.
\]
Moreover, Definition~\ref{def:mild_solution}
and~\eqref{eq:wp-initial-gap} give
\[
\delta:=\min\{s_*,\operatorname*{ess\,inf}_{(x,t)\in\mathbb R\times(0,T)}s(x,t)\}>0\,,
\]
and therefore
\[
\|z_\pm\|_{L^\infty(\mathbb R\times(0,T))}
\leq \frac{1}{\delta} \|\rho_\pm\|_{L^\infty(\mathbb R\times(0,T))}\,,
\quad \frac1s\in
W_{\mathrm{loc}}^{1,\infty} (\mathbb R\times(0,T))\,
\]
where the latter follows from the Sobolev chain rule
(see, e.g.,~\cite{Brezis2011}). Let \(\varphi\in C_c^\infty(\mathbb R\times(0,T))\). Since \(\tfrac{1}{s}\in W_{\mathrm{loc}}^{1,\infty}(\mathbb R\times(0,T))\) and \(\varphi\) has compact support
in \(\mathbb R\times(0,T)\), the Sobolev product rule implies that
\(\tfrac{\varphi}{s}\in W^{1,1}(\mathbb R\times(0,T))\)
and \(\tfrac{\varphi}{s}=0\) a.e.\ outside \(\operatorname{supp}\varphi\),
while, for \(j\in\{x,t\}\),
\[
\partial_j\biggl(\frac{\varphi}{s}\biggr) =
\frac{\partial_j\varphi}{s}
- \frac{\varphi\partial_js}{s^2}
\quad
\text{a.e. in }\mathbb R\times(0,T)\,.
\]
Consequently, a density argument provides a sequence
\((\varphi_n)_{n\in\mathbb N}\) satisfying
\[
\varphi_n\in
C_c^\infty\bigl(\mathbb R\times(0,T)\bigr),
\qquad
\varphi_n\rightarrow\frac{\varphi}{s}
\quad\text{in }
W^{1,1}\bigl(\mathbb R\times(0,T)\bigr),
\]
and each \(\varphi_n\) is an admissible test function in the
distributional formulations of the first two equations
of~\eqref{eq:collapsed_system}, which hold by
Proposition~\ref{prop:mild-implies-distributional}. Since
\[
\rho_\pm, \: v\rho_\pm, \:
Is-\frac{2v\rho_+\rho_-}{s}
\in
L^\infty(\mathbb R\times(0,T))\,,
\]
using the preceding \(W^{1,1}\)-convergence and the
\(L^\infty\)-bounds of these functions, H\"older's inequality allows
us to pass to the limit in these identities as \(n\to\infty\), and we
obtain
\begin{align*}
-&\int_0^T\int_{\mathbb R}
\rho_\pm(x,t) \Big[
\partial_t \Big(\frac{\varphi}{s}\Big)(x,t)
\pm v(x,t)
\partial_x \Big(\frac{\varphi}{s}\Big)(x,t)
\Big]\,\rd x\,\rd t\\
&\quad=
\int_0^T\int_{\mathbb R}
\biggl( I(x,t)s(x,t) - \frac{2v(x,t)\rho_+(x,t)\rho_-(x,t)}{s(x,t)}
\biggr)\frac{\varphi(x,t)}{s(x,t)}\,\rd x\,\rd t\,.
\end{align*}
Moreover, the identities for \(\partial_xs\) and \(\partial_ts\)
obtained above imply that
\[
\partial_ts
\pm v\partial_xs
=
-2v\rho_\mp
=
-2vs z_\mp
\quad
\text{a.e. in }\mathbb R\times(0,T)\,,
\]
and thus the left- and right-hand sides of the preceding weak identity can be written, respectively, as
\begin{align*}
&-\int_0^T\int_{\mathbb R} \rho_\pm(x,t)
\Big[
\partial_t\biggl(\frac{\varphi}{s}\biggr)(x,t)
\pm v(x,t) \partial_x\biggl(\frac{\varphi}{s}\biggr)(x,t)
\Big]\,\rd x\,\rd t = \\
& -\int_0^T\int_{\mathbb R} z_\pm(x,t)
\big(\partial_t\varphi(x,t)\pm v(x,t)\partial_x\varphi(x,t)
\big)\,\rd x\,\rd t -
2\int_0^T\int_{\mathbb R}
v(x,t)z_+(x,t)z_-(x,t)\varphi(x,t)\,\rd x\,\rd t
\end{align*}
and
\begin{align*}
&\int_0^T\int_{\mathbb R}
\biggl( I(x,t)s(x,t)
- \frac{2v(x,t)\rho_+(x,t)\rho_-(x,t)}{s(x,t)} \biggr)
\frac{\varphi(x,t)}{s(x,t)}\,\rd x\,\rd t\\
&\quad =
\int_0^T\int_{\mathbb R}
I(x,t)\varphi(x,t)\,\rd x\,\rd t
- 2\int_0^T\int_{\mathbb R}
v(x,t)z_+(x,t)z_-(x,t)\varphi(x,t)\,\rd x\,\rd t\,.
\end{align*}
Hence the last integrals on the two sides cancel, and we conclude that
\[
-\int_0^T\int_{\mathbb R}
z_\pm(x,t)(\partial_t\varphi(x,t) \pm v(x,t)\partial_x\varphi(x,t))\,\rd x\,\rd t
= \int_0^T\int_{\mathbb R}
I(x,t)\varphi(x,t)\,\rd x\,\rd t\,,
\]
which is the distributional formulation
of~\eqref{eq:normalised-system}.
Furthermore, Lemma~\ref{lem:temporal-traces-restart}, Proposition~\ref{prop:apriori-bounds}, and the definition of \(\delta\) give
\[
s(\cdot,t)\geq\delta
\quad\text{a.e. in }\mathbb R\,,
\qquad \|\rho_\pm(\cdot,t)\|_{L^\infty(\mathbb R)}
\leq B(t)\leq B(T)
\qquad \text{for every }t\in[0,T]\,,
\]
and hence $z_\pm(\cdot,t)\in L^\infty(\mathbb R)$ for every $t\in[0,T]$.
For every \(\zeta\in L^1(\mathbb R)\) and \(r,t\in[0,T]\), one has
\(\tfrac{\zeta}{s}(\cdot,r)\in L^1(\mathbb R)\), and therefore
\begin{align*}
&\left| \int_{\mathbb R}
\big(z_\pm(x,t)-z_\pm(x,r)\big)\zeta(x)\,\rd x \right|
\\
&\quad= \left| \int_{\mathbb R}
\big(\rho_\pm(x,t)-\rho_\pm(x,r)\big) \frac{\zeta(x)}{s(x,r)}\,\rd x + \int_{\mathbb R} \rho_\pm(x,t)\zeta(x)
\Big( \frac1{s(x,t)}- \frac1{s(x,r)}
\Big)\, \rd x \right|
\\
&\quad\leq
\left| \int_{\mathbb R}
\big(\rho_\pm(x,t)-\rho_\pm(x,r)\big)\frac{\zeta(x)}{s(x,r)}\,\rd x\right|+
\|\rho_\pm(\cdot,t)\|_{L^\infty(\mathbb R)}
\|\zeta\|_{L^1(\mathbb R)} \left\| \frac1{s(\cdot,t)}
-\frac1{s(\cdot,r)}\right\|_{L^\infty(\mathbb R)}
\\
&\quad\leq
\left|\int_{\mathbb R}\big(\rho_\pm(x,t)-\rho_\pm(x,r)\big)\frac{\zeta(x)}{s(x,r)}\,\rd x\right|+
\frac{B(T)}{\delta^2}
\|\zeta\|_{L^1(\mathbb R)}
\|s(\cdot,t)-s(\cdot,r)\|_{L^\infty(\mathbb R)}
\quad \longrightarrow 0
\end{align*}
as $t\to r$. Here the convergence follows from~\eqref{eq:continuity_in_time} and
\(s=1-f\). Hence
\[
z_\pm
\in
C_{w^\ast}\bigl([0,T];L^\infty(\mathbb R)\bigr)\,,
\qquad z_\pm(\cdot,0)
=\frac{\rho_\pm(\cdot,0)}{s(\cdot,0)}
=\frac{\rho_{\pm,0}}{s_0}
=z_{\pm,0}\quad\text{in }L^\infty(\mathbb R)\,,
\]
where the identity at \(t=0\) follows from the initial values in Lemma~\ref{lem:temporal-traces-restart}.

Next, \eqref{eq:forward_characteristic_equation} and~\eqref{eq:forward_pullback_continuity} which were established
in the proof of Proposition~\ref{prop:mild-implies-distributional}, together with the identities for \(\partial_xs\) and \(\partial_ts\) obtained above imply
that, for a.e. \(y\in\mathbb R\),
\(\rho_\pm(Y_\pm(\cdot;y),\cdot),
s(Y_\pm(\cdot;y),\cdot)\in W^{1,\infty}(0,T)\), and 
\begin{align*}
\frac{\rd}{\rd t}s(Y_\pm(t;y),t)=
\big(\partial_ts\pm v\partial_xs\big)
(Y_\pm(t;y),t)=
-2(v\rho_\mp)(Y_\pm(t;y),t)
\qquad
\text{for a.e. }t\in(0,T)\,.
\end{align*}
Moreover, the definition of \(\delta\) and the continuity of
\(t\mapsto s(Y_\pm(t;y),t)\) give
\(s(Y_\pm(t;y),t)\geq\delta\) for every
\(t\in[0,T]\), and thus
\(z_\pm(Y_\pm(\cdot;y),\cdot)\in W^{1,\infty}(0,T)\) and
\begin{align*}
\frac{\rd}{\rd t}z_\pm(Y_\pm(t;y),t)
&=\Big(I-\frac{2v\rho_+\rho_-}{s^2}
\mp(\partial_xv)z_\pm +\frac{2v\rho_\pm\rho_\mp}{s^2}
\Big)(Y_\pm(t;y),t)\\
&=I(Y_\pm(t;y),t)\mp (\partial_xv)(Y_\pm(t;y),t)
z_\pm(Y_\pm(t;y),t)
\end{align*}
for a.e. \(t\in(0,T)\), while the initial values in
Lemma~\ref{lem:temporal-traces-restart} and  \(Y_\pm(0;y)=y\) give
\[
z_\pm\bigl(Y_\pm(0;y),0\bigr)
= \frac{\rho_{\pm,0}(y)}{s_0(y)}
= z_{\pm,0}(y)\,.
\]
The integrating-factor formula gives the corresponding representation along the forward characteristics \(Y_\pm\), and, setting \(x=Y_\pm(t;y)\), the inverse relation \(y=Y_\pm(t;\cdot)^{-1}(x)=X_\pm(0;x,t)\) and the flow relation \(Y_\pm(\cdot;y)=X_\pm(\cdot;x,t)\) on \([0,t]\), both established in
the proof of Proposition~\ref{prop:mild-implies-distributional}, yield, upon
substitution into this representation,
\eqref{eq:normalised-characteristic-representation} for every \(t\in[0,T]\) and for a.e. \(x\in\mathbb R\).

Finally,~\eqref{eq:restarted-f-duhamel} with \(\theta=0\), \(\rho_\pm=sz_\pm\), and the lower bound for \(s\) obtained above give, for a.e. \(x\in\mathbb R\),
\[
s(x,\cdot)\,, \, 
\log s(x,\cdot) \in W^{1,\infty}(0,T)\quad \text{and} \quad 
s(x,0)=s_0(x),
\]
and
\[
\frac{\rd}{\rd t}\log s(x,t)
=
-v(x,t)\big(z_+(x,t)+z_-(x,t)\big)
\qquad
\text{for a.e. }t\in(0,T)\,.
\]
Consequently,
\[
\log s(x,t)-\log s_0(x)
=-\int_0^tv(x,\tau)\bigl(z_+(x,\tau)+z_-(x,\tau)\bigr)\,\rd\tau
\]
for a.e. \(x\in\mathbb R\) and every \(t\in[0,T]\), which proves \eqref{eq:s-normalised-representation} after taking the exponential.
\end{proof}

Proposition~\ref{prop:normalised-equations} now gives the following
direct decay estimate for the unreplicated fraction.

\begin{cor}
\label{cor:direct-s-bound}
Under the assumptions of Proposition~\ref{prop:normalised-equations}, assume, in addition, that there exist \(I_{\min}>0\) and \(v_{\min}>0\) such that
\[
I(x,t)\geq I_{\min}\,,
\quad
v(x,t)\geq v_{\min}
\qquad
\text{for a.e. }(x,t)\in\mathbb R\times(0,T)\,.
\]
Then
\begin{equation}
\label{eq:direct-general-s-bound}
\|s(\cdot,t)\|_{L^\infty(\mathbb R)}
\leq \|s_0\|_{L^\infty(\mathbb R)}
\exp\left[ -2v_{\min}I_{\min} \int_0^t\int_0^\tau
e^{-M_{\partial v}(\tau-\sigma)} \,\rd\sigma\,\rd\tau
\right]
\end{equation}
for every \(t\in[0,T]\). In particular, if
\(M_{\partial v}=0\), which holds, for example, when the fork speed
is independent of \(x\), then
\begin{equation}
\label{eq:direct-constant-speed-bound}
\|s(\cdot,t)\|_{L^\infty(\mathbb R)}
\leq \|s_0\|_{L^\infty(\mathbb R)} e^{-v_{\min}I_{\min}t^2}
\end{equation}
for every \(t\in[0,T]\).
\end{cor}

\begin{proof}
By~\eqref{eq:wp-initial-forks} and~\eqref{eq:wp-initial-gap}, \(z_{\pm,0}(x)=\tfrac{\rho_{\pm,0}(x)}{s_0(x)}\geq0\) for a.e. \(x\in\mathbb R\), while~\eqref{eq:wp-velocity-regularity} gives
\begin{align*}
\exp\left(\mp\int_\sigma^t
(\partial_xv)(X_\pm(\eta;x,t),\eta)\,\rd\eta
\right)\geq \exp\left(-\int_\sigma^t |
(\partial_xv)(X_\pm(\eta;x,t),\eta)|\,\rd\eta
\right) \geq
e^{-M_{\partial v}(t-\sigma)}
\end{align*}
for every \(0\leq\sigma\leq t\leq T\) and for a.e.\
\(x\in\mathbb R\). Hence
\eqref{eq:normalised-characteristic-representation} of
Proposition~\ref{prop:normalised-equations} and the lower bound for
\(I\) yield
\begin{align*}
z_\pm(x,t)\geq
\int_0^t I(X_\pm(\sigma;x,t),\sigma)
\exp\left(\mp\int_\sigma^t(\partial_xv)\bigl(X_\pm(\eta;x,t),\eta)\,\rd\eta\right)\,\rd\sigma
\geq I_{\min} \int_0^t
e^{-M_{\partial v}(t-\sigma)}\,\rd\sigma
\end{align*}
for every \(t\in[0,T]\) and for a.e.\
\(x\in\mathbb R\). Consequently,
\begin{align*}
\int_0^t v(x,\tau)(z_+(x,\tau)+z_-(x,\tau))\,\rd\tau
\geq 2v_{\min}I_{\min} \int_0^t\int_0^\tau e^{-M_{\partial v}(\tau-\sigma)} \,\rd\sigma\,\rd\tau
\end{align*}
for every \(t\in[0,T]\) and for a.e.\
\(x\in\mathbb R\), and therefore
\eqref{eq:s-normalised-representation} of
Proposition~\ref{prop:normalised-equations} gives
\[
s(x,t)
\leq s_0(x) \exp\left[ -2v_{\min}I_{\min} \int_0^t\int_0^\tau e^{-M_{\partial v}(\tau-\sigma)}
\,\rd\sigma\,\rd\tau
\right]\,.
\]
Taking the essential supremum over \(x\in\mathbb R\) proves
\eqref{eq:direct-general-s-bound}, while, if
\(M_{\partial v}=0\), then
\(
2v_{\min}I_{\min}
\int_0^t\int_0^\tau
1\,\rd\sigma\,\rd\tau
=
v_{\min}I_{\min}t^2,
\)
and~\eqref{eq:direct-constant-speed-bound} follows from
\eqref{eq:direct-general-s-bound}.
\end{proof}

The normalised representation formulas
\eqref{eq:normalised-characteristic-representation}
and~\eqref{eq:s-normalised-representation}, established in Proposition~\ref{prop:normalised-equations} under the compatibility condition~\eqref{eq:compatible_data}, together with the continuation criterion~\eqref{eq:blowup_criterion} in
Theorem~\ref{thm:maximal-continuation}, give the following
global-in-time existence result.

\begin{thm}
\label{thm:global-compatible}
Assume~\eqref{eq:compatible_data} and
\eqref{eq:wp-initial-forks}--\eqref{eq:wp-initial-gap}, and suppose that \eqref{eq:wp-initiation}
and~\eqref{eq:wp-velocity-regularity} hold for every finite \(T>0\). Then
\eqref{eq:collapsed_system}--\eqref{eq:initial_data} admits a unique global-in-time mild solution on
\(\mathbb R\times[0,\infty)\), in the sense that its restriction to \(\mathbb R\times[0,T']\) is a mild solution according to
Definition~\ref{def:mild_solution} for every \(T'>0\).
\end{thm}

\begin{proof}
By Theorem~\ref{thm:maximal-continuation}, there exist \(T_{\max}\in(0,\infty]\) and a unique maximal mild solution on \(\mathbb R\times[0,T_{\max})\).  Assume, for contradiction, that \(T_{\max}<\infty\).

Fix \(T>T_{\max}\) and, for this value of \(T\), use the constants \(M_I\), \(M_v\), and \(M_{\partial v}\) defined at the beginning of
Section~\ref{Sec3} (after~\eqref{eq:wp-assumptions}). For every \(t\in[0,T_{\max})\), choose
\(T'\in(t,T_{\max})\). The restriction of the maximal mild solution to \(\mathbb R\times[0,T']\) is a mild solution, and hence Proposition~\ref{prop:normalised-equations} applies on this interval,
while~\eqref{eq:wp-initial-forks}
and~\eqref{eq:wp-initial-gap} give
\[
\|z_{\pm,0}\|_{L^\infty(\mathbb R)}
\leq \frac{1}{s_*} \|\rho_{\pm,0}\|_{L^\infty(\mathbb R)}
<\infty\,.
\]
Therefore, it follows from
\eqref{eq:normalised-characteristic-representation} in
Proposition~\ref{prop:normalised-equations} that
\begin{align*}
\|z_\pm(\cdot,t)\|_{L^\infty(\mathbb R)}
\leq
\|z_{\pm,0}\|_{L^\infty(\mathbb R)}
e^{M_{\partial v}T_{\max}}
+ M_I\int_0^{T_{\max}}
e^{M_{\partial v}(T_{\max}-\sigma)}\,\rd\sigma
<\infty
\end{align*}
for every \(t\in[0,T_{\max})\), and thus
\[
\sup_{0\leq t<T_{\max}} \big(\|z_+(\cdot,t)\|_{L^\infty(\mathbb R)}
+\|z_-(\cdot,t)\|_{L^\infty(\mathbb R)} \big) <\infty\,.
\]
Combining this estimate with
\eqref{eq:s-normalised-representation} of
Proposition~\ref{prop:normalised-equations},
\eqref{eq:wp-velocity-regularity},
and~\eqref{eq:wp-initial-gap}, we obtain
\begin{align*}
s(x,t)
&\geq s_*
\exp\Big(-M_v\int_0^t
\big(\|z_+(\cdot,\tau)\|_{L^\infty(\mathbb R)}
+\|z_-(\cdot,\tau)\|_{L^\infty(\mathbb R)}
\big)\,\rd\tau\Big)\\
&\geq
s_* \exp\Big( -M_vT_{\max}\sup_{0\leq t<T_{\max}}
\big(\|z_+(\cdot,t)\|_{L^\infty(\mathbb R)}
+\|z_-(\cdot,t)\|_{L^\infty(\mathbb R)}\big)
\Big)>0
\end{align*}
for a.e. \(x\in\mathbb R\) and every \(t\in[0,T_{\max})\). Hence
\[
\liminf_{t\uparrow T_{\max}}
\operatorname*{ess\,inf}_{x\in\mathbb R}s(x,t)>0\,,
\]
which contradicts the continuation
criterion~\eqref{eq:blowup_criterion} in
Theorem~\ref{thm:maximal-continuation}. Therefore
\(T_{\max}=\infty\), and the unique maximal mild solution is
global-in-time, which proves the theorem.
\end{proof}

We end this section with the following monotonicity consequence of \eqref{eq:duhamel_f} and the non-negativity of the fork densities and fork speed.

\begin{cor}
\label{cor:monotonicity}
Fix \(T\in(0,\infty)\), assume
\eqref{eq:wp-velocity-regularity}, and let
\((\rho_+,\rho_-,f)\) be a mild solution on
\(\mathbb R\times[0,T]\). Then, for a.e.\
\(x\in\mathbb R\), the functions
\[
t\mapsto f(x,t)\,,\quad t\mapsto s(x,t)
\]
are absolutely continuous on \([0,T]\) and satisfy
\[
\partial_tf(x,t) =v(x,t)(\rho_+(x,t)+\rho_-(x,t))\geq0 \,,
\quad \partial_ts(x,t) = -v(x,t)(\rho_+(x,t)+\rho_-(x,t))\leq0
\]
for a.e. \(t\in(0,T)\), and hence these are,
respectively, non-decreasing and non-increasing on \([0,T]\).
\end{cor}

\begin{proof}
It follows from~\eqref{eq:duhamel_f} that the functions in the
statement are absolutely continuous and satisfy the derivative identities in the statement, and, since 
\eqref{eq:wp-velocity-regularity} and
Definition~\ref{def:mild_solution} give
\(v\geq0\) and \(\rho_\pm\geq0\), respectively, we obtain, 
for every \(0\leq r\leq t\leq T\),
\begin{align*}
f(x,t)-f(x,r)= \int_r^t v(x,\tau)(\rho_+(x,\tau)+\rho_-(x,\tau))\,\rd\tau
\geq0 
\end{align*}
and 
\begin{align*}
s(x,t)-s(x,r)=-\int_r^t v(x,\tau)(\rho_+(x,\tau)+\rho_-(x,\tau))\,\rd\tau
\leq0
\end{align*}
for a.e. \(x\in\mathbb R\), which proves the claimed monotonicity.
\end{proof}

\section{Replication Timing}
\label{Sec4}
In this section, we derive replication-time bounds on the real line
\(\mathbb R\) and on the one-dimensional torus of length \(L>0\), denoted by
\(\mathbb T_L\coloneqq\mathbb R/(L\mathbb Z)\). 
Moreover, under the additional assumptions stated later in Subsection~\ref{Sec4.2}, we obtain
replication-time bounds on the bounded interval \(D\coloneqq[0,L]\) and on the
half-line \(\mathbb R_+\coloneqq[0,\infty)\). Finally, we compare the resulting estimates with stochastic simulations of DNA replication fitted to experimental replication-timing data. \\

For the results on \(\mathbb R\) and \(\mathbb T_L\), we assume that
\eqref{eq:wp-initiation} and~\eqref{eq:wp-velocity-regularity} hold on \(\mathbb R\times(0,T)\) for every \(T\in(0,\infty)\), where, for \(\mathbb T_L\), \(I\) and \(v\) are regarded as \(L\)-periodic functions on \(\mathbb R\). We consider the initially unreplicated genome
\begin{subequations}
\label{eq:ass_light-cone-general}
\begin{equation}
\label{eq:ass_light-cone-initial}
s_0(x)=1\,, \quad
\rho_{+,0}(x)=\rho_{-,0}(x)=0
\qquad\text{for a.e. }x\in\mathbb R\,,
\end{equation}
and assume that there exists \(v_{\min}>0\) such that 
\begin{equation}
\label{eq:ass_light-cone-speed}
v(x,t)\geq v_{\min}
\quad\text{for a.e. }(x,t)\in\mathbb R\times(0,\infty)\,,
\end{equation}
\end{subequations}
Note that the initial data
in~\eqref{eq:ass_light-cone-initial} imply
\eqref{eq:compatible_data} and
\eqref{eq:wp-initial-forks}--\eqref{eq:wp-initial-gap} with \(s_*=1\). By Theorem~\ref{thm:global-compatible}, \eqref{eq:collapsed_system}--\eqref{eq:initial_data} admits a unique
global-in-time mild solution on \(\mathbb R\times[0,\infty)\). In what
follows, let \((\rho_+,\rho_-,f)\) denote this solution, set \(s:=1-f\), and let \(z_\pm\) denote the normalised fork densities defined in Proposition~\ref{prop:normalised-equations} on every finite time interval. 

\subsection{Replication timing on \texorpdfstring{$\R$}{R}}
\label{Sec4.1}
In this subsection, we prove replication-time bounds on \(\mathbb R\),
which are sharp when the initiation rate and the positive fork speed are
constant.

Using the normalised system~\eqref{eq:normalised-system}, we first derive the following light-cone representation formula for the unreplicated
fraction \(s(x,t)\)

\begin{lem}
\label{lem:light-cone-representation}
Fix \(T\in(0,\infty)\), assume \eqref{eq:wp-initiation}, \eqref{eq:wp-velocity-regularity}, and~\eqref{eq:ass_light-cone-initial}, and let
\((\rho_+,\rho_-,f)\) be a mild solution on \(\mathbb R\times[0,T]\). Suppose, in addition, that the fork speed is constant,
\[
v(x,t)=v>0
\quad\text{for a.e. }(x,t)\in\mathbb R\times(0,T)\,.
\]
Then the normalised fork densities defined in
Proposition~\ref{prop:normalised-equations} satisfy
\[
z_\pm(x,t)
= \int_0^t I(x\mp v(t-\sigma),\sigma)\,\rd\sigma
\]
for every \(t\in[0,T]\) and for a.e. \(x\in\mathbb R\). Consequently, the unreplicated fraction admits the light-cone representation
\begin{equation}
\label{eq:light-cone-derived-constant-v}
s(x,t) = \exp\left( -\int_0^t \int_{x-v(t-\sigma)}^{x+v(t-\sigma)} I(\xi,\sigma)\,\rd\xi\,\rd\sigma \right)
\end{equation}
for every \(t\in[0,T]\) and for a.e. \(x\in\mathbb R\).
\end{lem}

\begin{proof}
Since the fork speed is constant, \eqref{eq:normalised-system} of
Proposition~\ref{prop:normalised-equations} reduces to
\[
\partial_tz_\pm(x,t)\pm v\partial_xz_\pm(x,t) =I(x,t) \quad\text{in } \mathcal D'(\mathbb R\times(0,T))\,,
\]
while~\eqref{eq:ass_light-cone-initial} and the definition of
\(z_{\pm,0}\) give \(z_{\pm,0}(x)=\frac{\rho_{\pm,0}(x)}{s_0(x)}=0 \) for a.e. $x\in\mathbb R$.
Moreover, the corresponding backward characteristics \(X_\pm\) satisfy
\[
X_\pm(\sigma;x,t)
=x\mp v(t-\sigma)\qquad\text{for }0\leq\sigma\leq t\leq T\,,
\]
and, since \(\partial_xv=0\), it follows from
\eqref{eq:normalised-characteristic-representation} of
Proposition~\ref{prop:normalised-equations} that
\[
z_\pm(x,t)=\int_0^t I (x\mp v(t-\sigma),\sigma)\,\rd\sigma
\]
for every \(t\in[0,T]\) and for a.e. \(x\in\mathbb R\).

Furthermore,~\eqref{eq:s-normalised-representation}, \eqref{eq:ass_light-cone-initial}, and the constant value of the fork speed give \(s(x,t)>0\) and
\begin{align*}
\log s(x,t)= -v\int_0^t (z_+(x,\tau)+z_-(x,\tau))\,\rd\tau =
-v\int_0^t\int_0^\tau [I(x-v(\tau-\sigma),\sigma) + I(x+v(\tau-\sigma),\sigma)]\,\rd\sigma\,\rd\tau
\end{align*}
for every \(t\in[0,T]\) and for a.e. \(x\in\mathbb R\). Since
\(I\in L^\infty\bigl(\mathbb R\times(0,T)\bigr)\), Fubini's theorem
applies to the preceding integral, and therefore
\[
\log s(x,t)
=-v\int_0^t\int_\sigma^t\left[I(x-v(\tau-\sigma),\sigma)
+I(x+v(\tau-\sigma),\sigma)\right]\rd\tau\,\rd\sigma\,.
\]
For a.e. \(\sigma\in(0,t)\), the changes of variables
\(\xi=x\mp v(\tau-\sigma)\) and \(\rd\xi=\mp v\,\rd\tau\), applied to
the two terms of the inner integral, respectively, give
\begin{align*}
v\int_\sigma^t I\bigl(x-v(\tau-\sigma),\sigma\bigr)\,\rd\tau
= \int_{x-v(t-\sigma)}^x I(\xi,\sigma)\,\rd\xi 
\end{align*}
and
\begin{align*}
v\int_\sigma^t I\bigl(x+v(\tau-\sigma),\sigma\bigr)\,\rd\tau
= \int_x^{x+v(t-\sigma)} I(\xi,\sigma)\,\rd\xi\,,
\end{align*}
and adding these two identities and substituting the resulting identity
into the preceding formula yields
\[
\log s(x,t) = -\int_0^t \int_{x-v(t-\sigma)}^{x+v(t-\sigma)} I(\xi,\sigma)\,\rd\xi\,\rd\sigma,
\]
which proves~\eqref{eq:light-cone-derived-constant-v} after taking the
exponential.
\end{proof}

For a spatially and temporally dependent fork speed \(v(x,t)\), the backward characteristics \(X_\pm(\cdot;x,t)\) defined by~\eqref{eq:char} in Subsection~\ref{subsec:characteristics-mild-formulation} determine, for every \(x\in\mathbb R\), \(t\in[0,\infty)\), and  \(\tau\in[0,t]\), the interval
\begin{equation}
\label{eq:general-cone-cross-section}
C_{x,t}^{\mathbb R}(\tau)
\coloneqq \left[ X_+(\tau;x,t), X_-(\tau;x,t) \right] \subset\mathbb R\,,
\end{equation}
and the following lemma gives the corresponding light-cone
representation formula for the unreplicated fraction \(s(x,t)\).

\begin{lem}
\label{lem:general-light-cone-representation}
Fix \(T\in(0,\infty)\), assume \eqref{eq:wp-initiation}, \eqref{eq:wp-velocity-regularity}, and~\eqref{eq:ass_light-cone-general}, and let
\((\rho_+,\rho_-,f)\) be a mild solution on \(\mathbb R\times[0,T]\). Then
\begin{equation}
\label{eq:general_lightcone_representation}
s(x,t)
= \exp\left( -\int_0^t \int_{C_{x,t}^{\mathbb R}(\tau)} I(\xi,\tau)\,\rd\xi\,\rd\tau \right)
\end{equation}
for every \(t\in[0,T]\) and for a.e. \(x\in\mathbb R\).
\end{lem}

\begin{proof}
The assertion for \(t=0\) follows from \eqref{eq:ass_light-cone-initial}. Observe that, for every \(x\in\mathbb R\) and \(0\leq\tau\leq r\leq T\), the integral
formulation of~\eqref{eq:char} reads
\begin{equation}
\label{eq:bw-char-integral}
X_\pm(\tau;x,r) = x\mp\int_\tau^r v(X_\pm(q;x,r),q)\,\rd q\,,
\end{equation}
and differentiation with respect to the terminal position \(x\) gives
\begin{equation}
\label{eq:bw-char-jacobian}
\partial_xX_\pm(\tau;x,r)
= \exp\left(\mp\int_\tau^r(\partial_xv)(X_\pm(q;x,r),q)\,\rd q\right)>0
\end{equation}
for a.e. \((x,\tau,r)\in \mathbb R\times(0,T)\times(0,T)\) satisfying \(\tau<r\).
Moreover, the composition property of the characteristics \(X_\pm\) reads
\begin{equation}
\label{eq:bw-char-composition}
X_\pm(\tau;x,r)
=X_\pm(\tau;X_\pm(\sigma;x,r),\sigma)\,,
\qquad 0\leq\tau\leq\sigma\leq r\leq T\,.
\end{equation}
Equations~\eqref{eq:bw-char-integral}, \eqref{eq:bw-char-composition}, and~\eqref{eq:char_lipschitz} show that, for every
\(x\in\mathbb R\) and \(\tau\in[0,T)\), the map
\(r\mapsto X_\pm(\tau;x,r)\) belongs to
\(W^{1,\infty}(\tau,T)\), and
\eqref{eq:bw-char-integral}, \eqref{eq:bw-char-jacobian}, and~\eqref{eq:bw-char-composition} further imply
\begin{equation}
\label{eq:bw-char-terminal-derivative}
\frac{\rd}{\rd r}X_\pm(\tau;x,r)= \mp v(x,r)\partial_xX_\pm(\tau;x,r)
\end{equation}
for a.e. \((x,\tau,r)\in\mathbb R\times(0,T)\times(0,T)\) satisfying \(\tau<r\).
Consequently, \eqref{eq:wp-velocity-regularity}, \eqref{eq:ass_light-cone-speed},
\eqref{eq:bw-char-jacobian}, and~\eqref{eq:bw-char-terminal-derivative} imply,
for every \(t\in(0,T]\) and for a.e.\((x,\tau)\in\mathbb R\times(0,t)\),
\begin{equation}
\label{eq:bw-char-r-bounds}
\frac{\rd}{\rd r}X_+(\tau;x,r)<0<\frac{\rd}{\rd r}X_-(\tau;x,r)\,,
\quad v_{\min}e^{-M_{\partial v}T}
\leq \Big|\frac{\rd}{\rd r}X_\pm(\tau;x,r)\Big|\leq M_ve^{M_{\partial v}T}
\end{equation}
for a.e. \(r\in(\tau,t)\), and hence, since \(X_\pm(\tau;x,\tau)=x\), the maps
\begin{align*}
X_+(\tau;x,\cdot): [\tau,t]\rightarrow [X_+(\tau;x,t),x]\,, \quad 
X_-(\tau;x,\cdot):[\tau,t]\rightarrow [x,X_-(\tau;x,t)]
\end{align*}
are bi-Lipschitz bijections, where the first map is strictly
decreasing and the second is strictly increasing.

Since~\eqref{eq:ass_light-cone-initial} implies
\(z_{\pm,0}(x)=0\) for a.e.\ \(x\in\mathbb R\), the characteristic representation~\eqref{eq:normalised-characteristic-representation} of Proposition~\ref{prop:normalised-equations}, together with~\eqref{eq:bw-char-jacobian} and~\eqref{eq:bw-char-terminal-derivative}, gives
\begin{align}
\label{eq:norm-char-zero-data}
v(x,r)z_\pm(x,r)&=
v(x,r)\int_0^r I(X_\pm(\tau;x,r),\tau)\partial_xX_\pm(\tau;x,r)\,\rd\tau\nonumber \\
&=\mp\int_0^rI(X_\pm(\tau;x,r),\tau) \frac{\rd}{\rd r}X_\pm(\tau;x,r)\,\rd\tau
\end{align}
for a.e. \((x,r)\in\mathbb R\times(0,T)\). Fix \(t\in(0,T]\). The measurability of the characteristics and the fact that \(x\mapsto X_\pm(\tau;x,r)\) is a bi-Lipschitz bijection of \(\mathbb R\) for every \(0\leq\tau\leq r\leq T\), both established in Subsection~\ref{subsec:characteristics-mild-formulation}, show that
the functions \(I\bigl(X_\pm(\tau;x,r),\tau\bigr)\) are measurable
and satisfy the a.e. bounds in~\eqref{eq:wp-initiation}, and
\eqref{eq:bw-char-r-bounds} therefore gives
\begin{align*}
0 \leq I(X_\pm(\tau;x,r),\tau) \Big|\frac{\rd}{\rd r}X_\pm(\tau;x,r)\Big|
\leq M_I \Big|\frac{\rd}{\rd r}X_\pm(\tau;x,r)\Big| \leq M_IM_ve^{M_{\partial v}T}
\end{align*}
for a.e. \((x,\tau,r)\in\mathbb R\times(0,t)\times(0,t)\) satisfying \(\tau<r\).
Moreover, \eqref{eq:s-normalised-representation} and~\eqref{eq:ass_light-cone-initial} imply that \(s(x,t)>0\) and
\[
-\log s(x,t) = \int_0^t v(x,r)(z_+(x,r)+z_-(x,r))\,\rd r
\]
for a.e. \(x\in\mathbb R\). Since the preceding bound is integrable over
\(\{(\tau,r)\in(0,t)^2:\tau<r\}\), substituting
\eqref{eq:norm-char-zero-data} into this formula, applying Fubini's
theorem to change the order of integration, using the signs
in~\eqref{eq:bw-char-r-bounds}, and applying, for a.e. \((x,\tau)\in\mathbb R\times(0,t)\), the change-of-variables formula
to the two bi-Lipschitz maps
\begin{align*}
r&\mapsto \xi=X_\pm(\tau;x,r)\,, \quad r\in[\tau,t]\,, \qquad 
\Big| \frac{\rd\xi}{\rd r}\Big|= \Big|\frac{\rd}{\rd r}X_\pm(\tau;x,r)\Big|
\quad\text{for a.e. }r\in(\tau,t)\,,
\end{align*}
we obtain, together with~\eqref{eq:general-cone-cross-section},
\begin{align*}
-\log s(x,t)&= \int_0^t\int_\tau^t \Big[ I(X_+(\tau;x,r),\tau)
\Big|\frac{\rd}{\rd r}X_+(\tau;x,r)\Big|
+I(X_-(\tau;x,r),\tau)\Big|\frac{\rd}{\rd r}X_-(\tau;x,r)\Big|\Big]
\,\rd r\,\rd\tau
\\
&=\int_0^t \Big[ \int_{X_+(\tau;x,t)}^x I(\xi,\tau)\,\rd\xi + \int_x^{X_-(\tau;x,t)}
I(\xi,\tau)\,\rd\xi\Big]\, \rd\tau
\\
&=\int_0^t\int_{X_+(\tau;x,t)}^{X_-(\tau;x,t)}I(\xi,\tau)\,\rd\xi\,\rd\tau
=\int_0^t\int_{C_{x,t}^{\mathbb R}(\tau)}I(\xi,\tau)\,\rd\xi\,\rd\tau
\end{align*}
for a.e. \(x\in\mathbb R\). Since \(t\in(0,T]\) was arbitrary and the assertion for \(t=0\) was proved above, taking the exponential proves
\eqref{eq:general_lightcone_representation}.
\end{proof}

\begin{rem}
Since the global solution fixed above is a mild solution on
\(\mathbb R\times[0,T]\) for every \(T\in(0,\infty)\),
Lemma~\ref{lem:general-light-cone-representation}
shows that \eqref{eq:general_lightcone_representation} holds for every \(t\geq0\) and for a.e. \(x\in\mathbb R\). If, in addition, \(v(x,t)=v>0\) for a.e. \((x,t)\in\mathbb R\times(0,\infty)\), then
Lemma~\ref{lem:light-cone-representation} similarly shows that \eqref{eq:light-cone-derived-constant-v} holds for every \(t\geq0\) and
for a.e. \(x\in\mathbb R\).
\end{rem}

Observe that the definition~\eqref{eq:general-cone-cross-section},
the integral formulation~\eqref{eq:bw-char-integral} of the
characteristics \(X_\pm\), and~\eqref{eq:ass_light-cone-speed} imply 
\begin{align}
\label{eq:cone_lbR}
|C_{x,t}^{\mathbb R}(\tau)|
&=X_-(\tau;x,t)-X_+(\tau;x,t) \nonumber\\
&=\int_\tau^t\big(v(X_-(q;x,t),q)
+ v(X_+(q;x,t),q)\big)\,\rd q 
\geq
2v_{\min}(t-\tau)
\end{align}
for every \(x\in\mathbb R\), \(t\geq0\), and
\(0\leq\tau\leq t\).

In order to estimate the integral in the light-cone representation
formula~\eqref{eq:general_lightcone_representation} of
Lemma~\ref{lem:general-light-cone-representation} uniformly with respect
to \(x\in\mathbb R\), we use the following definition.

\begin{deff}
\label{def:local_initiation_mass} 
For a.e. \(\tau>0\) and every \(r\geq0\), we define the \emph{local initiation mass} \(m_I(r,\tau)\) by
\begin{equation}
\label{eq:def_mI_all_domains}
m_I(r,\tau)\coloneqq \operatorname*{ess\,inf}_{y\in\mathbb R}
\int_y^{y+r}I(\xi,\tau)\,\rd\xi.
\end{equation}
At the remaining times, set
\(m_I(r,\tau):=0\) for every \(r\geq0\).
\end{deff}

In order to use \(m_I\) in the time integrals below and to compare its
values at different interval lengths, we first establish its
measurability on \([0,\infty)\times(0,\infty)\) and then its
monotonicity and Lipschitz continuity with respect to \(r\). Indeed,
for a.e. \(\tau>0\), every \(r\geq0\), and all
\(y,h\in\mathbb R\),
\begin{align*}
\left|\int_{y+h}^{y+h+r}I(\xi,\tau)\,\rd\xi
-\int_y^{y+r}I(\xi,\tau)\,\rd\xi
\right|\leq
2\|I(\cdot,\tau)\|_{L^\infty(\mathbb R)}|h|\,,
\end{align*}
and hence the map
\(y\mapsto\int_y^{y+r}I(\xi,\tau)\,\rd\xi
\)
is continuous on \(\mathbb R\), which, together with
\eqref{eq:def_mI_all_domains} and the density of \(\mathbb Q\) in
\(\mathbb R\), gives
\begin{equation}
\label{eq:mI-rational}
m_I(r,\tau)
=\inf_{y\in\mathbb R}
\int_y^{y+r}I(\xi,\tau)\,\rd\xi=
\inf_{q\in\mathbb Q}
\int_q^{q+r}I(\xi,\tau)\,\rd\xi\,.
\end{equation}
Since, by Tonelli's theorem, for every \(q\in\mathbb Q\), the map
\((r,\tau)\mapsto\int_q^{q+r}I(\xi,\tau)\,\rd\xi\) is measurable on
\([0,\infty)\times(0,\infty)\), it follows from
\eqref{eq:mI-rational} and
Definition~\ref{def:local_initiation_mass} that
\(m_I:[0,\infty)\times(0,\infty)\rightarrow[0,\infty)\) is
measurable.  Moreover, for \(0\leq r_1\leq r_2\) and for a.e.\(\tau>0\),
\begin{align*}
0\leq m_I(r_1,\tau) \leq m_I(r_2,\tau) \leq
m_I(r_1,\tau)
+\|I(\cdot,\tau)\|_{L^\infty(\mathbb R)} (r_2-r_1),
\end{align*}
so that \(r\mapsto m_I(r,\tau)\) is non-decreasing and Lipschitz continuous.

With these preparations, the light-cone representation formula
\eqref{eq:general_lightcone_representation} of
Lemma~\ref{lem:general-light-cone-representation},
the cone-width estimate~\eqref{eq:cone_lbR}, and
Definition~\ref{def:local_initiation_mass} yield the following bounds
for the unreplicated fraction \(s\) and for the corresponding
locuswise near-completion time, which is defined below by requiring
that \(s\) be bounded above by a prescribed threshold at almost every
genomic position.

\begin{prop}
\label{prop:master_completion_bound_R}
Assume~\eqref{eq:ass_light-cone-general}, and suppose that
\eqref{eq:wp-initiation} and~\eqref{eq:wp-velocity-regularity} hold for
every finite \(T>0\). Let \((\rho_+,\rho_-,f)\) be the unique
global-in-time mild solution given by
Theorem~\ref{thm:global-compatible}, and set \(s:=1-f\). Then, for every
\(t\geq0\),
\begin{equation}
\label{eq:master_bound_R}
\|s(\cdot,t)\|_{L^\infty(\mathbb R)}
\leq
\exp\left\{
-\int_0^t
m_I\bigl(2v_{\min}(t-\tau),\tau\bigr)\,\rd\tau
\right\}.
\end{equation}
Consequently, for every \(\varepsilon\in(0,1)\),
\begin{align}
\label{eq:master_Teps_R}
T_\varepsilon^{\mathbb R} &\coloneqq \inf\left\{ t\geq0: \|s(\cdot,t)\|_{L^\infty(\mathbb R)} \leq\varepsilon \right\}\nonumber \\
&\leq
\inf\left\{
t\geq0:
\int_0^t m_I\bigl(2v_{\min}(t-\tau),\tau\bigr)\,\rd\tau\geq
\ln\left(\frac{1}{\varepsilon}\right)
\right\}\,,
\end{align}
with the convention \(\inf\varnothing:=+\infty\).
\end{prop}

\begin{proof}
Fix \(t\geq0\), choose \(T>t\), and apply
Lemma~\ref{lem:general-light-cone-representation} to the restriction
of the global solution to \(\mathbb R\times[0,T]\).
Since \(C_{x,t}^{\mathbb R}(\tau)\) is an interval,
Definition~\ref{def:local_initiation_mass},
\eqref{eq:mI-rational}, the monotonicity of \(m_I(\cdot,\tau)\),
and~\eqref{eq:cone_lbR} imply
\begin{align*}
\int_{C_{x,t}^{\mathbb R}(\tau)}
I(\xi,\tau)\,\rd\xi \geq
m_I\big( |C_{x,t}^{\mathbb R}(\tau)|,\tau \big)\geq
m_I\big(2v_{\min}(t-\tau),\tau\big)
\end{align*}
for every \(x\in\mathbb R\) and for a.e. \(\tau\in(0,t)\).
Inserting this estimate into
\eqref{eq:general_lightcone_representation} of
Lemma~\ref{lem:general-light-cone-representation} and taking the
essential supremum with respect to \(x\in\mathbb R\) proves
\eqref{eq:master_bound_R}, which also shows that the set on the
right-hand side of~\eqref{eq:master_Teps_R} is contained in the
set defining \(T_\varepsilon^{\mathbb R}\), and hence
\eqref{eq:master_Teps_R} follows by taking the infimum. 
\end{proof}

\begin{rem}
\setlength{\emergencystretch}{2em}
For the population-level model~\eqref{eq:collapsed_system}, the
quantity \(T_\varepsilon^{\mathbb R}\), defined by~\eqref{eq:master_Teps_R} of
Proposition~\ref{prop:master_completion_bound_R}, describes
near-completion at individual genomic positions, since
\mbox{\(\|s(\cdot,t)\|_{L^\infty(\mathbb R)}\leq\varepsilon\)}
means that, for a.e.\ genomic position, the fraction of cells
in which that position remains unreplicated at time \(t\)
is at most \(\varepsilon\).
This estimate does not imply that the entire genome has been
replicated in at least a fraction \(1-\varepsilon\) of cells,
since the remaining unreplicated positions may differ
between cells.\par
\end{rem} 

The following corollary gives the corresponding bound for the
expected replication time.

\begin{cor}
\label{cor:master_E(T)_bound}
Under the assumptions of Proposition~\ref{prop:master_completion_bound_R}, define, for a.e. \(x\in\mathbb R\), the expected replication time by
\begin{equation}
\label{eq:expectation_time_def}
\mathbb E[T(x)] \coloneqq \int_0^\infty s(x,t)\,\rd t \in[0,\infty]\,.
\end{equation}
Then
\begin{equation}
\label{eq:bound_on_E(x)_R}
\mathbb E[T(x)] \leq \int_0^\infty \exp\left\{
-\int_0^t m_I\bigl(2v_{\min}(t-\tau),\tau\bigr)\,\rd\tau \right\}\rd t
\end{equation}
for a.e. \(x\in\mathbb R\).
\end{cor}

\begin{proof} 
Integrating the bound~\eqref{eq:master_bound_R} of
Proposition~\ref{prop:master_completion_bound_R} with respect
to \(t\) and using~\eqref{eq:expectation_time_def}, we obtain
\eqref{eq:bound_on_E(x)_R} for a.e. \(x\in\mathbb R\).
\end{proof}

If the initiation rate is bounded below by a positive constant,
the bounds~\eqref{eq:master_bound_R} and~\eqref{eq:master_Teps_R}
of Proposition~\ref{prop:master_completion_bound_R} become explicit.

\begin{cor}
\label{cor:explicit-positive-initiation}
Under the assumptions of
Proposition~\ref{prop:master_completion_bound_R}, assume, in addition,
that there exists \(I_{\min}>0\) such that
\[
I(x,t)\geq I_{\min}
\quad\text{for a.e. }(x,t)\in\mathbb R\times(0,\infty)\,.
\]
Then, for every \(t\geq0\),
\begin{equation}
\label{eq:explicit_s_bound}
\|s(\cdot,t)\|_{L^\infty(\mathbb R)}
\leq e^{-I_{\min}v_{\min}t^2}\,.
\end{equation}
Consequently, for every \(\varepsilon\in(0,1)\),
\begin{equation}
\label{eq:explicit_Teps_bound}
T_\varepsilon^{\mathbb R}
\leq \sqrt{ \frac{1}{I_{\min}v_{\min}} \ln\left(\frac{1}{\varepsilon}\right)
}\,.
\end{equation}
\end{cor}

\begin{proof}
For every \(t\geq0\) and for a.e. \(\tau\in(0,t)\),
\eqref{eq:def_mI_all_domains} of
Definition~\ref{def:local_initiation_mass} and the lower bound
for \(I\) give
\begin{align*}
m_I\bigl(2v_{\min}(t-\tau),\tau\bigr) \geq
\operatorname*{ess\,inf}_{y\in\mathbb R} \int_y^{y+2v_{\min}(t-\tau)}
I_{\min}\,\rd\xi=
2I_{\min}v_{\min}(t-\tau)\,,
\end{align*}
and therefore
\begin{align*}
\int_0^t
m_I\bigl(2v_{\min}(t-\tau),\tau\bigr)\,\rd\tau \geq
2I_{\min}v_{\min}
\int_0^t(t-\tau)\,\rd\tau
=
I_{\min}v_{\min}t^2\,.
\end{align*}
Consequently, \eqref{eq:master_bound_R} and
\eqref{eq:master_Teps_R} of
Proposition~\ref{prop:master_completion_bound_R} imply
\eqref{eq:explicit_s_bound} and
\begin{align*}
T_\varepsilon^{\mathbb R}
\leq \inf\left\{ t\geq0: I_{\min}v_{\min}t^2
\geq \ln\left(\frac1\varepsilon\right)\right\}=
\sqrt{ \frac{1}{I_{\min}v_{\min}} \ln\left(\frac1\varepsilon\right)}
\end{align*}
for every \(\varepsilon\in(0,1)\), which proves
\eqref{eq:explicit_Teps_bound}.
\end{proof}

\begin{rem}
\label{rem:relation_between_T_max_and_T_epsilon}
Under the assumptions of
Proposition~\ref{prop:master_completion_bound_R},
Theorem~\ref{thm:global-compatible} gives \(T_{\max}=\infty\), whereas
\(T_\varepsilon^{\mathbb R}\), defined by~\eqref{eq:master_Teps_R}, is the infimum of the times at which the unreplicated fraction is bounded by \(\varepsilon\) in \(L^\infty(\mathbb R)\).
\end{rem}

The following corollary gives a two-sided estimate for the
near-completion time \(T_\varepsilon^{\mathbb R}\) under upper and lower bounds for the initiation rate \(I(x,t)\) and the fork speed \(v(x,t)\).

\begin{cor}
\label{cor:two-sided-completion}
Under the assumptions of
Proposition~\ref{prop:master_completion_bound_R}, assume, in addition,
that there exist \(I_{\min}>0\), \(I_{\max}>0\), and \(v_{\max}>0\) such that
\[
I_{\min}\le I(x,t)\le I_{\max}\,, \quad
v(x,t)\le v_{\max} \qquad\text{for a.e. }(x,t)\in\mathbb R\times(0,\infty)\,.
\]
Then
\begin{equation}
e^{-I_{\max}v_{\max}t^2} \le s(x,t)
\le e^{-I_{\min}v_{\min}t^2}
\qquad
\text{for every }t\geq0\text{ and for a.e. }x\in\mathbb R\,.
\label{eq:two_sided_s_bound}
\end{equation}
In particular, for every \(\varepsilon\in(0,1)\),
\begin{equation}
\sqrt{ \frac{1}{I_{\max}v_{\max}} \ln\left(\frac{1}{\varepsilon}\right)
} \le T_\varepsilon^\R
\le \sqrt{ \frac{1}{I_{\min}v_{\min}} \ln\left(\frac{1}{\varepsilon}\right)}\,.
\label{eq:two_sided_Teps}
\end{equation}
\end{cor}

\begin{proof}
The upper bound in~\eqref{eq:two_sided_s_bound} follows from
\eqref{eq:explicit_s_bound} of
Corollary~\ref{cor:explicit-positive-initiation}, while the
equalities in~\eqref{eq:cone_lbR} and the upper bound for
\(v(x,t)\) give
\[
|C_{x,t}^{\mathbb R}(\tau)| =\int_\tau^t \big(v\bigl(X_-(q;x,t),q\bigr)
+ v\bigl(X_+(q;x,t),q\bigr) \big)\,\rd q \leq \int_\tau^t 2v_{\max}\,\rd q = 2v_{\max}(t-\tau)
\]
for every \(x\in\mathbb R\), \(t\geq0\), and \(\tau\in[0,t]\). Combining this estimate with the upper bound for \(I(x,t)\) and
\eqref{eq:general_lightcone_representation} of
Lemma~\ref{lem:general-light-cone-representation}, we obtain
\begin{align*}
-\log s(x,t)&= \int_0^t \int_{C_{x,t}^{\mathbb R}(\tau)}I(\xi,\tau)\,\rd\xi\,\rd\tau \\
&\leq
I_{\max}\int_0^t |C_{x,t}^{\mathbb R}(\tau)|\,\rd\tau
\leq 2I_{\max}v_{\max} \int_0^t(t-\tau)\,\rd\tau =
I_{\max}v_{\max}t^2
\end{align*}
for every \(t\geq0\) and for a.e. \(x\in\mathbb R\), which
proves the lower bound in~\eqref{eq:two_sided_s_bound} after
taking the exponential. Moreover, for every \(\varepsilon\in(0,1)\), the lower bound in~\eqref{eq:two_sided_s_bound} and the definition
in~\eqref{eq:master_Teps_R} of Proposition~\ref{prop:master_completion_bound_R} give
\begin{align*}
T_\varepsilon^{\mathbb R} &\geq
\inf\left\{ t\geq0: e^{-I_{\max}v_{\max}t^2}\leq\varepsilon
\right\}\\
&=
\inf\left\{ t\geq0: I_{\max}v_{\max}t^2
\geq\ln\left(\frac1\varepsilon\right) \right\}=
\sqrt{
\frac{1}{I_{\max}v_{\max}}
\ln\left(\frac1\varepsilon\right)
}\,,
\end{align*}
which proves the lower bound in~\eqref{eq:two_sided_Teps},
while the upper bound follows from~\eqref{eq:explicit_Teps_bound}
of Corollary~\ref{cor:explicit-positive-initiation}.
\end{proof}

The following proposition gives two-sided bounds for the expected
replication time \(\mathbb E[T(x)]\) and establishes its Lipschitz
continuity with respect to \(x\).

\begin{prop}
\label{prop:expected-replication-time}
Assume~\eqref{eq:ass_light-cone-general}, and suppose that
\eqref{eq:wp-initiation} and~\eqref{eq:wp-velocity-regularity}
hold on \(\mathbb R\times(0,T)\) for every \(T\in(0,\infty)\).
Let \((\rho_+,\rho_-,f)\) be the unique global-in-time mild
solution given by Theorem~\ref{thm:global-compatible}, and
set \(s:=1-f\). Assume, in addition, that there exist
\(I_{\min}>0\), \(I_{\max}>0\), and \(v_{\max}>0\) such that
\[
I_{\min}\leq I(x,t)\leq I_{\max}\,,
\quad
v(x,t)\leq v_{\max}
\qquad\text{for a.e. }(x,t)\in\mathbb R\times(0,\infty)\,,
\]
and let \(\mathbb E[T(x)]\) be defined by~\eqref{eq:expectation_time_def} of Corollary~\ref{cor:master_E(T)_bound}. Then
\begin{equation}
\label{eq:expected_time_bounds}
\frac{\sqrt{\pi}}{2\sqrt{I_{\max}v_{\max}}}
\leq \mathbb E[T(x)] \leq
\frac{\sqrt{\pi}}{2\sqrt{I_{\min}v_{\min}}}
\end{equation}
for a.e. \(x\in\mathbb R\). Moreover, \(\mathbb E[T(\cdot)]\in W^{1,\infty}(\mathbb R)\), with
\begin{equation}
\label{eq:expected_time_lipschitz}
|\partial_x\mathbb E[T(x)]|
\leq \frac{1}{v_{\min}}
\end{equation}
for a.e. \(x\in\mathbb R\).
\end{prop}

\begin{proof}
Integrating~\eqref{eq:two_sided_s_bound} of
Corollary~\ref{cor:two-sided-completion} with respect to \(t\)
and using~\eqref{eq:expectation_time_def} of
Corollary~\ref{cor:master_E(T)_bound}, together with
\[
\int_0^\infty e^{-at^2}\,\rd t = \frac{\sqrt{\pi}}{2\sqrt a}\,,
\quad a>0\,,
\]
proves~\eqref{eq:expected_time_bounds}. For the derivative with respect to \(x\), observe that the non-negativity of \(\rho_\pm\), \eqref{eq:ass_light-cone-speed}, and the integral
formula~\eqref{eq:duhamel_f} in
Subsection~\ref{subsec:characteristics-mild-formulation} give
\begin{align*}
0 \leq \int_0^R (\rho_+(x,t)+\rho_-(x,t))\,\rd t
&\leq
\frac{1}{v_{\min}} \int_0^R v(x,t)(\rho_+(x,t)+\rho_-(x,t))\,\rd t\\
&=
\frac{f(x,R)-f_0(x)}{v_{\min}}
\leq \frac{1}{v_{\min}}
\end{align*}
for every \(R>0\) and for a.e.\ \(x\in\mathbb R\), and
monotone convergence as \(R\to\infty\) therefore yields
\[
\int_0^\infty
\bigl(\rho_+(x,t)+\rho_-(x,t)\bigr)\,\rd t
\leq
\frac{1}{v_{\min}}
\quad\text{for a.e. }x\in\mathbb R\,.
\]
Together with~\eqref{eq:expected_time_bounds}, this estimate
allows us to apply Fubini's theorem and
\eqref{eq:spatial_derivative_of_s} of
Lemma~\ref{lem:compatibility-propagation}, which yield, for
every \(\zeta\in C_c^\infty(\mathbb R)\),
\begin{align*}
-\int_{\mathbb R}
\mathbb E[T(x)]\partial_x\zeta(x)\,\rd x
= -\int_0^\infty\int_{\mathbb R} s(x,t)\partial_x\zeta(x)\,\rd x\,\rd t
&=\int_{\mathbb R}
\left(\int_0^\infty(\rho_+(x,t)-\rho_-(x,t))\,\rd t
\right)
\zeta(x)\,\rd x\,.
\end{align*}
Consequently, the weak derivative satisfies
\begin{align*}
|\partial_x\mathbb E[T(x)]|
=\left| \int_0^\infty (\rho_+(x,t)-\rho_-(x,t))\,\rd t
\right|  \leq
\int_0^\infty
\bigl(\rho_+(x,t)+\rho_-(x,t)\bigr)\,\rd t
\leq
\frac{1}{v_{\min}}
\end{align*}
for a.e. \(x\in\mathbb R\), which, together
with~\eqref{eq:expected_time_bounds}, proves
\(\mathbb E[T(\cdot)]\in W^{1,\infty}(\mathbb R)\)
and~\eqref{eq:expected_time_lipschitz}.
\end{proof}

\begin{rem}
\label{rem:sharpness}
The bounds~\eqref{eq:two_sided_s_bound}
and~\eqref{eq:two_sided_Teps} of
Corollary~\ref{cor:two-sided-completion} are sharp, since
equality holds when \(I(x,t)=I_{\min}=I_{\max}>0\) and
\(v(x,t)=v_{\min}=v_{\max}>0\) for a.e. \((x,t)\in\mathbb R\times(0,\infty)\).
\end{rem}

Finally, we show that the unreplicated fraction \(s(x,t)\)
converges uniformly to zero and the replicated fraction \(f(x,t)\)
converges uniformly to one as \(t\to\infty\).

\begin{cor}
Assume~\eqref{eq:ass_light-cone-general}, and suppose that
\eqref{eq:wp-initiation} and~\eqref{eq:wp-velocity-regularity}
hold on \(\mathbb R\times(0,T)\) for every \(T\in(0,\infty)\).
Let \((\rho_+,\rho_-,f)\) be the unique global-in-time mild
solution given by Theorem~\ref{thm:global-compatible}, and
set \(s:=1-f\). Assume, in addition, that there exists
\(I_{\min}>0\) such that
\[
I(x,t)\geq I_{\min}
\quad\text{for a.e. }(x,t)\in\mathbb R\times(0,\infty)\,.
\]
Then
\[
\|f(\cdot,t)-1\|_{L^\infty(\mathbb R)} = \|s(\cdot,t)\|_{L^\infty(\mathbb R)}
\longrightarrow0
\quad\text{as }t\to\infty\,.
\]
\end{cor}

\begin{proof}
The assertion follows from~\eqref{eq:explicit_s_bound} of
Corollary~\ref{cor:explicit-positive-initiation},
\(I_{\min}v_{\min}>0\), and the identity \(s=1-f\).
\end{proof}

\subsection{Replication timing on other geometries}
\label{Sec4.2}

In this subsection, we derive corresponding replication-time bounds
on the torus \(\mathbb T_L\) of length \(L\), the bounded non-periodic
interval \(D=[0,L]\), and the half-line \(\mathbb R_+\), where the
results on \(D\) and \(\mathbb R_+\) require additional assumptions
which are stated below.\\

For each \(\Omega\in\{\mathbb T_L,D,\mathbb R_+\}\), we assume
\eqref{eq:ass_light-cone-general}, and suppose that the initiation
rate \(I\) and the fork speed \(v\) satisfy
\eqref{eq:wp-initiation} and~\eqref{eq:wp-velocity-regularity},
respectively, for every \(T\in(0,\infty)\), with \(\mathbb R\)
replaced by \(\Omega\) in these assumptions.
On \(\mathbb T_L\), \(I\) and \(v\) are regarded as \(L\)-periodic
functions on \(\mathbb R\), and, for each domain,
\(v_{\min}>0\) denotes a uniform lower bound for \(v(x,t)\)
as in~\eqref{eq:ass_light-cone-speed}.

For later use, we define the geometry-dependent uniform width
function by
\begin{equation}
\label{eq:def_width_function}
\omega_\Omega(\sigma)
\coloneqq
\begin{cases}
\min\{2v_{\min}\sigma,L\}\,,
& \Omega=\mathbb T_L\,,\\
\min\{v_{\min}\sigma,L\}\,,
& \Omega=D\,,\\
v_{\min}\sigma\,,
& \Omega=\mathbb R_+\,,
\end{cases}
\qquad \sigma\geq0\,.
\end{equation}
The factor \(2\) in~\eqref{eq:def_width_function} reflects fork propagation in both directions, whereas at a boundary position of \(D\) or \(\mathbb R_+\) only one side of the interval remains within the domain, and the minimum with \(L\) accounts for the maximal length of an interval in \(D\) or of the projection of an interval
onto \(\mathbb T_L\).

\begin{rem}
\label{rem:domain-consistency}
For \(\Omega=\mathbb T_L\), uniqueness in
Theorem~\ref{thm:global-compatible} shows that the global mild
solution \((\rho_+,\rho_-,f)\) corresponding to the \(L\)-periodic functions \(I\) and \(v\) and the initial data
in~\eqref{eq:ass_light-cone-initial} is \(L\)-periodic, so that
we may regard \(s\coloneqq 1-f\) as a function on
\(\mathbb T_L\times[0,\infty)\). For \(\Omega\in\{D,\mathbb R_+\}\), we do not treat the
corresponding boundary-value problems here, but assume that
\(s:\Omega\times[0,\infty)\longrightarrow[0,1]\) is measurable
and satisfies
\begin{subequations}
\label{eq:additional-domain-light-cone-assumptions}
\begin{equation}
\label{eq:additional-domain-light-cone-representation}
s(x,t) = \exp\left\{ -\int_0^t
\int_{C_{x,t}^{\Omega}(\tau)} I(\xi,\tau)\,\rd\xi\,\rd\tau
\right\}
\end{equation}
for every \(t\geq0\) and for a.e. \(x\in\Omega\), where
\(C_{x,t}^{\Omega}(\tau)\subset\Omega\) are bounded intervals
whose endpoints are measurable in \(\tau\), with
\begin{equation}
\label{eq:additional-domain-cone-width}
|C_{x,t}^{\Omega}(\tau)| \geq \omega_\Omega(t-\tau)
\end{equation}
for a.e. \(\tau\in(0,t)\).
\end{subequations}
Here we use the convention \(\exp(-\infty):=0\).
\end{rem}

To estimate the initiation integrals, we define the local
initiation mass on \(\Omega\) by
\begin{equation}
\label{eq:def_domain_dependent_mI}
m_{I,\Omega}(r,\tau)
\coloneqq \inf_{A\in\mathcal A_\Omega(r)} \int_A I(\xi,\tau)\,\rd\xi
\end{equation}
for a.e. \(\tau>0\) and every \(r>0\), with \(r\leq L\)
when \(\Omega\in\{\mathbb T_L,D\}\), where
\(\mathcal A_\Omega(r)\) consists of the intervals of length
\(r\) contained in \(\Omega\) when
\(\Omega\in\{D,\mathbb R_+\}\), and of the connected arcs
of length \(r\) when \(\Omega=\mathbb T_L\).
We set \(m_{I,\Omega}(0,\tau):=0\) for every \(\tau>0\),
and at the remaining times we set
\(m_{I,\Omega}(r,\tau)\coloneqq 0\) for every admissible \(r\).

Arguing as after Definition~\ref{def:local_initiation_mass},
the map \((r,\tau)\mapsto m_{I,\Omega}(r,\tau)\) is measurable
on its domain and satisfies
\begin{align}
\label{eq:additional-domain-initiation-mass-properties}
0 \leq m_{I,\Omega}(r_1,\tau)
\leq m_{I,\Omega}(r_2,\tau) \leq m_{I,\Omega}(r_1,\tau)
+\|I(\cdot,\tau)\|_{L^\infty(\Omega)}(r_2-r_1)
\end{align}
for a.e. \(\tau>0\) and all admissible \(0\leq r_1\leq r_2\), so that \(m_{I,\Omega}(\cdot,\tau)\) is non-decreasing and Lipschitz
continuous. With these properties, we obtain the following
bounds for the unreplicated fraction \(s(x,t)\) and the
near-completion time \(T_\varepsilon^\Omega\), defined below.

\begin{prop}
\label{prop:geometry_master_completion_bound}
Let \(\Omega\in\{\mathbb T_L,D,\mathbb R_+\}\), and assume
\eqref{eq:wp-initiation}, \eqref{eq:wp-velocity-regularity},
and~\eqref{eq:ass_light-cone-general} on the corresponding
domain as specified above. For \(\Omega=\mathbb T_L\), let \(s\) be the unreplicated fraction
of the periodic global-in-time mild solution, and, for
\(\Omega\in\{D,\mathbb R_+\}\), assume
\eqref{eq:additional-domain-light-cone-assumptions} as stated in Remark~\ref{rem:domain-consistency}.
Then, for every \(t\geq0\),
\begin{equation}
\label{eq:geometry_master_completion_bound}
\|s(\cdot,t)\|_{L^\infty(\Omega)}
\leq \exp\left\{ -\int_0^t m_{I,\Omega}(\omega_\Omega(t-\tau),\tau)\,\rd\tau \right\}\,.
\end{equation}
In particular, for every \(\varepsilon\in(0,1)\),
\begin{align}
T_\varepsilon^\Omega
&\coloneqq
\inf\left\{ t\geq0: \|s(\cdot,t)\|_{L^\infty(\Omega)} \leq\varepsilon \right\}
\nonumber\\
&\leq
\inf\left\{ t\geq0: \int_0^t
m_{I,\Omega}(\omega_\Omega(t-\tau),\tau)\,\rd\tau \geq \ln\left(\frac1\varepsilon\right) \right\}\,,
\label{eq:geometry_master_Teps_bound}
\end{align}
with \(\inf\varnothing:=\infty\).
\end{prop}

\begin{proof}
The identity \(s_0=1\) in~\eqref{eq:ass_light-cone-initial}
proves~\eqref{eq:geometry_master_completion_bound} at \(t=0\)
when \(\Omega=\mathbb T_L\), while
\eqref{eq:additional-domain-light-cone-representation}
of Remark~\ref{rem:domain-consistency} at \(t=0\) gives the same
conclusion when \(\Omega\in\{D,\mathbb R_+\}\), and we therefore
fix \(t>0\). 

For \(\Omega=\mathbb T_L\), we regard \(I\), \(v\), and \(s\)
as \(L\)-periodic functions on \(\mathbb R\), so that
\eqref{eq:cone_lbR} and~\eqref{eq:def_width_function} imply
\[
0<\omega_{\mathbb T_L}(t-\tau) \leq
\min\left\{ X_-(\tau;x,t)-X_+(\tau;x,t),L \right\}
\]
for every \(x\in\mathbb R\) and \(\tau\in(0,t)\).
The interval \(\big[ X_+(\tau;x,t), X_+(\tau;x,t)+\omega_{\mathbb T_L}(t-\tau)\big] \) is therefore contained in \(C_{x,t}^{\mathbb R}(\tau)\), and its projection onto \(\mathbb T_L\) is a connected arc of length \(\omega_{\mathbb T_L}(t-\tau)\) when this length is smaller than \(L\) and is the whole torus when this length equals \(L\), so that the non-negativity of \(I\) and \eqref{eq:def_domain_dependent_mI} imply
\begin{align*}
\int_{C_{x,t}^{\mathbb R}(\tau)} I(\xi,\tau)\,\rd\xi \geq
\int_{X_+(\tau;x,t)}^{X_+(\tau;x,t)+\omega_{\mathbb T_L}(t-\tau)} I(\xi,\tau)\,\rd\xi
\geq m_{I,\mathbb T_L}(\omega_{\mathbb T_L}(t-\tau),\tau)
\end{align*}
for every \(x\in\mathbb R\) and for a.e. \(\tau\in(0,t)\). For \(\Omega\in\{D,\mathbb R_+\}\), since
\(C_{x,t}^{\Omega}(\tau)\) is a bounded interval by Remark~\ref{rem:domain-consistency},
\eqref{eq:additional-domain-cone-width} of this remark, together
with~\eqref{eq:def_domain_dependent_mI} and
\eqref{eq:additional-domain-initiation-mass-properties},
similarly implies
\begin{align*}
\int_{C_{x,t}^{\Omega}(\tau)} I(\xi,\tau)\,\rd\xi \geq
m_{I,\Omega}(|C_{x,t}^{\Omega}(\tau)|,\tau)
\geq m_{I,\Omega}(\omega_\Omega(t-\tau),\tau)
\end{align*}
for a.e. \(x\in\Omega\) and for a.e. \(\tau\in(0,t)\). Inserting these estimates into
\eqref{eq:general_lightcone_representation} of Lemma~\ref{lem:general-light-cone-representation} in the periodic case, and into \eqref{eq:additional-domain-light-cone-representation}
of Remark~\ref{rem:domain-consistency} in the other two cases,
and taking the essential supremum with respect to \(x\in\Omega\) proves~\eqref{eq:geometry_master_completion_bound}, which also
shows that the set on the right-hand side of
\eqref{eq:geometry_master_Teps_bound} is contained in the set
defining \(T_\varepsilon^\Omega\), and hence
\eqref{eq:geometry_master_Teps_bound} follows by taking the infimum.
\end{proof}

As in Corollary~\ref{cor:master_E(T)_bound}, integration with respect to time gives the following bound.

\begin{cor}
\label{cor:geometry_master_expectation_bound}
Under the assumptions of
Proposition~\ref{prop:geometry_master_completion_bound}, define,
for a.e. \(x\in\Omega\), the geometry-dependent replication-time
integral by
\[
\mathbb E[T_\Omega(x)] \coloneqq \int_0^\infty s(x,t)\,\rd t \in[0,\infty]\,.
\]
Then
\begin{equation}
\label{eq:geometry_master_expectation_bound}
\mathbb E[T_\Omega(x)] \leq \int_0^\infty
\exp\left\{ -\int_0^t m_{I,\Omega}(\omega_\Omega(t-\tau),\tau)\,\rd\tau \right\}\,\rd t
\end{equation}
for a.e. \(x\in\Omega\).
\end{cor}

\begin{proof}
Integrating~\eqref{eq:geometry_master_completion_bound} of
Proposition~\ref{prop:geometry_master_completion_bound}
with respect to \(t\) and using the definition above proves
\eqref{eq:geometry_master_expectation_bound}.
\end{proof}

Note that, if \(s(x,t)\) is the probability that the genomic position \(x\) remains unreplicated at time \(t\), then \(\mathbb E[T_\Omega(x)]\), defined in
Corollary~\ref{cor:geometry_master_expectation_bound}, is the expected replication time at \(x\).
We next show that the bounds for the unreplicated fraction \(s(x,t)\) and the near-completion time \(T_\varepsilon^\Omega\) in~\eqref{eq:geometry_master_completion_bound}
and~\eqref{eq:geometry_master_Teps_bound} of Proposition~\ref{prop:geometry_master_completion_bound}
become explicit if the initiation rate \(I(x,t)\) is bounded below by a positive constant.

\begin{cor}
\label{cor:explicit-geometry-bounds}
Under the assumptions of
Proposition~\ref{prop:geometry_master_completion_bound}, assume, in addition, that there exists $I_{\min} > 0$ such that
\[
I(x,t)\geq I_{\min} \quad\text{for a.e. }(x,t)\in\Omega\times(0,\infty)\,
\]
and define the corresponding accumulated width \(A_\Omega\) by
\begin{equation}
\label{eq:def_accumulated_width}
A_\Omega(t) \coloneqq  \int_0^t \omega_\Omega(\sigma)\,\rd \sigma, \quad t\geq 0\,.
\end{equation}
Then, for every \(t\geq0\),
\begin{equation}
\label{eq:explicit_geometry_s_bound}
\|s(\cdot,t)\|_{L^\infty(\Omega)} \le \exp\{-I_{\min}A_\Omega(t)\}\,
\end{equation}
where 
\begin{equation}
\label{eq:A_torus}
A_{\mathbb T_L}(t) = \begin{cases}
v_{\min}t^2\,, &0\leq t\leq\frac{L}{2v_{\min}}\,,\\
Lt-\frac{L^2}{4v_{\min}}\,, &t\geq\dfrac{L}{2v_{\min}}\,,
\end{cases}
\end{equation}
and 
\begin{equation}
\label{eq:additional-domain-accumulated-width-interval}
A_D(t)
= \begin{cases} \frac{v_{\min}}2t^2\,, &0\leq t\leq\frac{L}{v_{\min}}\,,\\
Lt-\dfrac{L^2}{2v_{\min}}\,, &t\geq\frac{L}{v_{\min}}\,,
\end{cases}
\end{equation}
and
\begin{equation}
\label{eq:additional-domain-accumulated-width-half-line}
A_{\mathbb R_+}(t)
= \frac{v_{\min}}2t^2\,.
\end{equation}
In particular, for every \(\varepsilon\in(0,1)\),
\begin{equation}
\label{eq:explicit_geometry_Teps_implicit}
T_\varepsilon^\Omega \leq \inf\left\{ t\geq0:
I_{\min}A_\Omega(t) \geq\ln\left(\frac1\varepsilon\right)
\right\},
\end{equation}
and thus
\begin{equation}
\label{eq:explicit_Teps_torus}
T_\varepsilon^{\mathbb T_L} \leq
\begin{cases}
\displaystyle \sqrt{ \frac{1}{I_{\min}v_{\min}} \ln\left(\frac1\varepsilon\right)}\,,
&\displaystyle \ln\left(\frac1\varepsilon\right) \leq\frac{I_{\min}L^2}{4v_{\min}}\,,
\\
\displaystyle \frac{1}{I_{\min}L}\ln\left(\frac1\varepsilon\right) +\frac{L}{4v_{\min}}\,,
&\displaystyle \ln\left(\frac1\varepsilon\right)>\frac{I_{\min}L^2}{4v_{\min}}\,,
\end{cases}
\end{equation}
and 
\begin{equation}
\label{eq:additional-domain-completion-bound-interval}
T_\varepsilon^D \leq
\begin{cases}
\displaystyle \sqrt{ \frac{2}{I_{\min}v_{\min}} \ln\left(\frac1\varepsilon\right) }\,,
&\displaystyle \ln\left(\frac1\varepsilon\right) \leq\frac{I_{\min}L^2}{2v_{\min}}\,,
\\
\displaystyle\frac{1}{I_{\min}L}\ln\left(\frac1\varepsilon\right) +\frac{L}{2v_{\min}}\,,
&\displaystyle \ln\left(\frac1\varepsilon\right) >\frac{I_{\min}L^2}{2v_{\min}}\,,
\end{cases}
\end{equation}
and
\begin{equation}
\label{eq:additional-domain-completion-bound-half-line}
T_\varepsilon^{\mathbb R_+}
\leq \sqrt{ \frac{2}{I_{\min}v_{\min}} \ln\left(\frac1\varepsilon\right)}\,.
\end{equation}
\end{cor}

\begin{proof}
For every \(t\geq0\) and for a.e. \(\tau\in(0,t)\), \eqref{eq:def_domain_dependent_mI} and the lower bound for \(I(x,t)\) imply
\[
m_{I,\Omega}(\omega_\Omega(t-\tau),\tau) \geq I_{\min}\omega_\Omega(t-\tau),
\]
and therefore, after the change of variables
\(\sigma=t-\tau\) and using~\eqref{eq:def_accumulated_width},
we obtain
\begin{align*}
\int_0^t m_{I,\Omega}(\omega_\Omega(t-\tau),\tau)\,\rd\tau
\geq I_{\min}\int_0^t\omega_\Omega(t-\tau)\,\rd\tau
= I_{\min}\int_0^t\omega_\Omega(\sigma)\,\rd\sigma
= I_{\min}A_\Omega(t)\,.
\end{align*}
Together with~\eqref{eq:geometry_master_completion_bound} and~\eqref{eq:geometry_master_Teps_bound} of
Proposition~\ref{prop:geometry_master_completion_bound}, this proves~\eqref{eq:explicit_geometry_s_bound} and~\eqref{eq:explicit_geometry_Teps_implicit}.
The formulas~\eqref{eq:A_torus}--\eqref{eq:additional-domain-accumulated-width-half-line} follow by inserting~\eqref{eq:def_width_function} into \eqref{eq:def_accumulated_width}, and solving the inequality in
\eqref{eq:explicit_geometry_Teps_implicit} gives
\eqref{eq:explicit_Teps_torus}--\eqref{eq:additional-domain-completion-bound-half-line}.
\end{proof}

\subsection{Numerical simulations}
\label{Sec4.3}

We next compare the near-completion time estimates derived above with stochastic simulations of DNA replication fitted to experimental replication-timing data. In particular, we compare the empirical near-completion times obtained from Monte Carlo simulations with the bounds for $T_\varepsilon$ in Propositions~\ref{prop:master_completion_bound_R} and~\ref{prop:geometry_master_completion_bound}, and compare empirical mean replication times with the corresponding expectation bounds in Corollaries~\ref{cor:master_E(T)_bound} and~\ref{cor:geometry_master_expectation_bound}.

\begin{figure}[t]
    \centering
    \includegraphics[width=.82\textwidth]{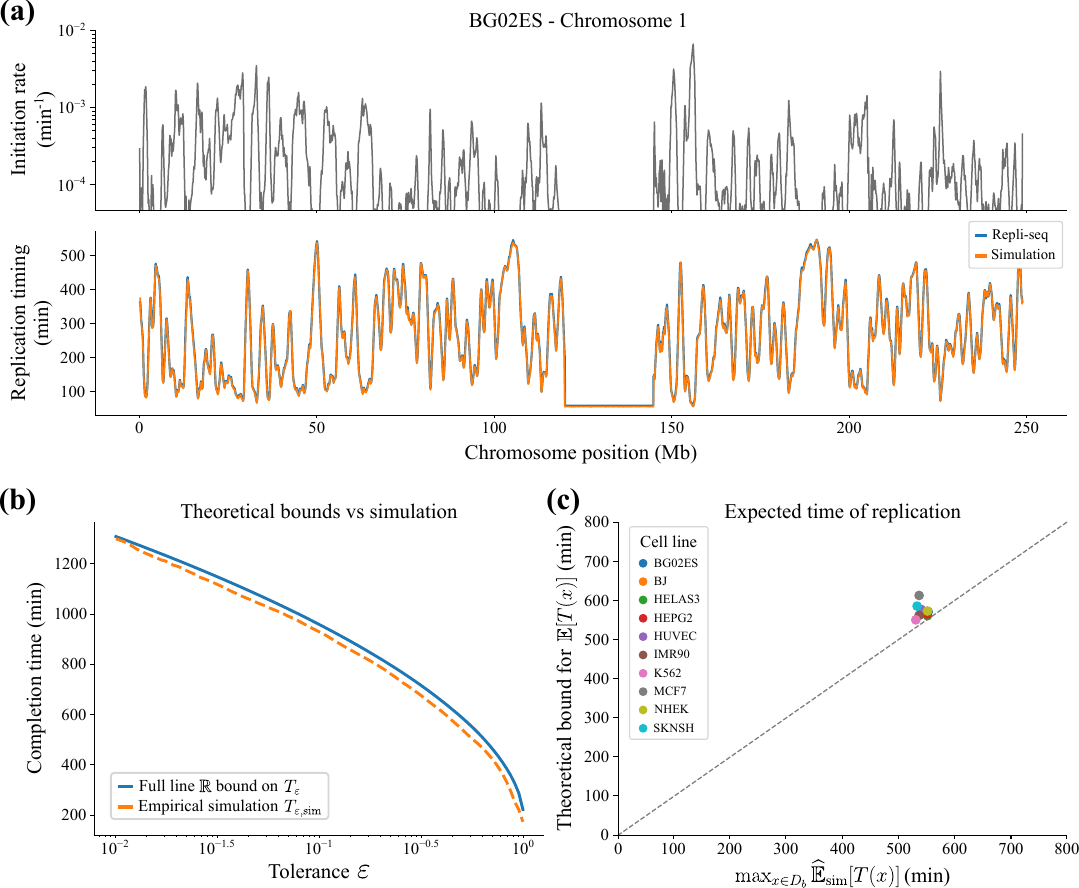}
    \caption{
    Numerical comparison of analytical completion-time estimates with finite-grid Monte Carlo simulations driven by Repli-seq-derived initiation landscapes.
    \textbf{(a)} Top: inferred position-dependent per-bin firing-rate profile ($I_j$) used to drive the simulations. Bottom: experimental replication-timing profile and mean simulated timing obtained from the fitted initiation landscape.
    \textbf{(b)} Comparison between the empirical near-completion time $T_{\varepsilon,\mathrm{sim}}$ obtained from Monte Carlo simulations and the theoretical upper bound derived from the local initiation-mass function.
    \textbf{(c)} Comparison between the maximum empirical mean replication time over the analysis domain and the corresponding uniform theoretical bound across datasets; the dashed diagonal indicates equality. Empirical means were estimated from 10,000 independent simulations on a 10~kb analysis grid.
    }
    \label{fig:completion_bounds_simulations}
\end{figure}

We follow the modelling approach from Berkemeier et al.~\cite{berkemeier2025dna}, where Repli-seq timing profiles are mapped to genomic position and converted into fitted initiation landscapes~\cite{hansen2010sequencing,davis2018encyclopedia,zhao2020high}. Although the analytical results allow both initiation and fork speed to vary over genomic position and S phase time, the numerical comparison is deliberately restricted to a baseline, unperturbed setting with time-independent initiation and constant fork speed set to $v=1.4\,\mathrm{kb\,min^{-1}}$~\cite{conti2007replication}. In particular, origin firing times follow exponential distributions with spatially varying rates. This controlled special case avoids fitting time-dependent initiation or fork-speed fields that are not identifiable from Repli-seq alone, while still testing the central mechanism in the estimates, namely how the spatial organisation of initiation, together with a lower bound on fork speed, controls replication completion.

Although each Repli-seq chromosome profile is defined on a finite interval \(D=[0,L]\), we compare the simulations with the full-line \(\mathbb{R}\) bound rather than with the substantially looser finite-interval estimate. The full-line bound assumes that every locus has a two-sided backward causal cone, whereas in a finite chromosome these cones are truncated near the boundaries. A direct comparison over the whole interval \(D\) would therefore introduce artificial boundary effects that arise from the mismatch between the full-line estimate and the finite simulation, rather than from the inferred replication programme itself. To avoid this, we extend the fitted initiation profile periodically beyond \(D\), solely to define the local initiation mass on \(\mathbb{R}\) using the full-line cone width \(2v\sigma\), while keeping the stochastic simulations finite and non-periodic. Empirical quantities are then evaluated on an interior domain \(D_b \coloneqq [b,L-b]\), with \(b=3000\,\mathrm{kb}\). This margin is chosen from the maximum analysis time \(t_{\max}=2000\,\mathrm{min}\) so that a fork moving at \(v=1.4\,\mathrm{kb\,min^{-1}}\) can travel at most \(vt_{\max}=1.4\times2000=2800\,\mathrm{kb}\), which we round conservatively to \(3000\,\mathrm{kb}\). Hence, for every locus in \(D_b \) and throughout the analysis window, the entire backward light cone remains inside the simulated chromosome. Consequently, the chromosome boundaries do not affect
replication on \(D_b\) throughout this analysis window. Furthermore, because centromeric Repli-seq profiles are poorly resolved owing to low short-read mappability~\cite{wheeler2025comparison}, we preserve their span but assign them the earliest replication timing before fitting, preventing them from contributing artificial late-completion bottlenecks.

In addition, we take a discretisation of $D_b$ and the initiation rate $I$. Let $B_j$ be a genomic bin of width $\Delta x$ and let $I_j\geq0$ denote its fitted per-bin firing rate, with units $\mathrm{min}^{-1}$. The corresponding continuum initiation density is
\[
I(x)=\frac{I_j}{\Delta x},
\qquad x\in B_j,
\]
with units $\mathrm{kb}^{-1}\mathrm{min}^{-1}$, so that, for some $m>0$,
\[
\int_{\bigcup_{j=k}^{m}B_j} I(x)\,\mathrm{d}x
=
\sum_{j=k}^{m}I_j.
\]
This conversion preserves the initiation mass of each bin. The finite-grid simulations approximate the corresponding
continuum initiation model, and the continuum bounds are not
claimed to hold exactly for the discrete model at fixed
grid spacing. Each genomic bin is therefore treated as a potential origin with firing rate $I_j$. For each fitted initiation landscape, we generate independent Monte Carlo realisations of replication and record two classes of observables. First, we compute the empirical replicated fraction \(f_{\mathrm{sim}}(x,t)\) and the corresponding unreplicated fraction \(s_{\mathrm{sim}}(x,t)=1-f_{\mathrm{sim}}(x,t)\), allowing us to estimate the interior near-completion time
\[
T_{\varepsilon,\mathrm{sim}}
=
\inf\left\{
t\geq 0:
\left\|s_{\mathrm{sim}}(\cdot,t)\right\|_{L^\infty(D_b )}
\leq \varepsilon
\right\}.
\]
Second, we record the replication time of each genomic bin across simulations and estimate the expected replication time by the empirical mean \(\widehat{\mathbb{E}}_{\mathrm{sim}}[T(x)]\) for \(x\in D_b \). All chromosome-scale simulations used to compute these means
completed before \(t_{\max}=2000\,\mathrm{min}\), and no
replication times were capped. The reported means are therefore
averages of the recorded replication times, with no contribution
beyond \(t_{\max}\) in the simulated sample. These quantities are then compared with the theoretical bounds obtained from the local initiation-mass function \(m_I\), computed by integrating the associated continuum density \(I(x)\) over genomic intervals.

The resulting comparisons are shown in Figure~\ref{fig:completion_bounds_simulations}. For the chromosome-scale analyses, we use wavelet-smoothed ENCODE Repli-seq profiles on chromosome~1 from selected cell lines, including BG02ES, BJ, HeLa, HEPG2, HUVEC, IMR90, K562, MCF7, NHEK and SK-N-SH~\cite{hansen2010sequencing,davis2018encyclopedia}. Across the datasets analysed, and after scaling to an approximately 8-hour S phase, the simulated timing profiles closely reproduced the Repli-seq-derived timing curves, with a maximum absolute discrepancy below $10$ min (Figure \hyperref[fig:completion_bounds_simulations]{2a}, illustrative for chromosome 1 of the embryonic stem cell BG02ES), and between the empirical and theoretical completion below $30$ min, over the full range of tolerances $\varepsilon$ considered (Figure \hyperref[fig:completion_bounds_simulations]{2b}). Across all cell lines, the empirical mean local replication
times lay below the corresponding analytical full-line
reference curves (Figure \hyperref[fig:completion_bounds_simulations]{2c}).

\begin{figure}[t]
    \centering
    \includegraphics[width=.82\textwidth]{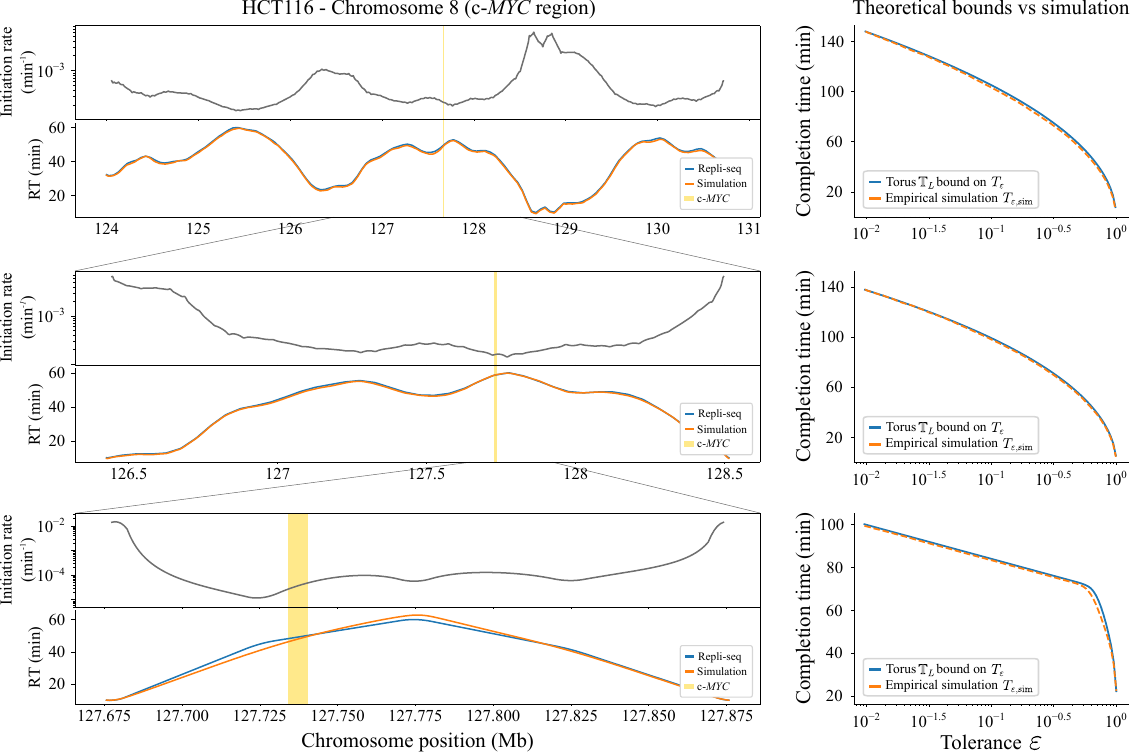}
    \caption{Periodic-domain simulations on three nested chromosome~8 intervals surrounding the c-\textit{MYC} locus. Left panels show the inferred per-bin firing-rate landscapes ($I_j$) and the corresponding Repli-seq and simulated mean replication-timing (RT) profiles; right panels compare empirical near-completion times with the theoretical estimates across tolerance levels. Each interval was evaluated using 10,000 independent simulations on a 1~kb analysis grid.}
\label{fig:completion_bounds_simulations_periodic}
\end{figure}

Beyond the chromosome-scale analyses on $\mathbb R$, the same framework can be tested on finite periodic geometries, such as $\mathbb T_L$. In the periodic formulation of~\eqref{eq:collapsed_system},
the light-cone representation formula
\eqref{eq:general_lightcone_representation} of
Lemma~\ref{lem:general-light-cone-representation} involves
the full cone interval on \(\mathbb R\), whereas for
continuum Poisson initiation on \(\mathbb T_L\), with the
same initiation rate, constant fork speed, and initially
unreplicated state, the integral is taken over its
projection onto \(\mathbb T_L\), with each position counted
once at each initiation time.
Although these integrals need not agree when the full
interval has length greater than \(L\), the non-negativity
of \(I\), \eqref{eq:cone_lbR}, and
\eqref{eq:def_domain_dependent_mI} yield the same upper
bound as in~\eqref{eq:geometry_master_completion_bound}
of Proposition~\ref{prop:geometry_master_completion_bound},
with the width function defined by~\eqref{eq:def_width_function}.
This setting is motivated by circular replication substrates, including oncogene-amplifying circular extrachromosomal DNA (ecDNA) and smaller extrachromosomal circular DNA, for which topology may influence how initiation-poor gaps are resolved~\cite{wu2019circular,yi2022extrachromosomal,tang2024transcription}. Jaworski et al.~\cite{jaworski2025ecdna} reconstructed a COLO~320DM ecDNA sequence mapping to chr8:126,425,747--127,997,820 and reported altered replication timing and origin usage relative to the corresponding chromosomal locus. Guided by this c-\textit{MYC}-associated region, and because the COLO~320DM Repli-seq profiles are not currently available through a traceable public accession, we use the independent HCT116 Repli-seq data of Zhao et al.~\cite{zhao2020high}, which Jaworski et al. also used, to define three nested test domains: a chr8 c-\textit{MYC}/ecDNA context window, chr8:124,000,000--130,775,000 ($L\simeq 6.8$~Mb); a dominant ecDNA-derived interval, chr8:126,425,747--128,512,820 ($L\simeq 2.1$~Mb); and a c-\textit{MYC}-centred origin-density zoom, chr8:127,675,747--127,875,747 ($L\simeq200$~kb) (Figure \hyperref[fig:completion_bounds_simulations_periodic]{3}). The HCT116 data have a source resolution of 50~kb and are interpolated to the 1~kb analysis grid used for fitting, simulation and evaluation of the local initiation mass. These intervals should be viewed as ecDNA-inspired modelling domains rather than reconstructions of the COLO~320DM circle. Their outer boundaries are selected from the HCT116 timing landscape to permit periodic identification without introducing a large artificial discontinuity, while the nested windows preserve the c-\textit{MYC}/ecDNA-associated initiation structure highlighted by Jaworski et al. Here, each bounded Repli-seq timing profile is rescaled to a nominal 60-minute window. This is not intended as an estimate of ecDNA completion time, but it serves as a convenient example for testing the robustness of the estimates to changes in time scale, in addition to domain geometry variation.

\section{Discussion}
\label{Sec5}

In this work, we develop and analyse a kinetic framework for DNA replication that links the evolution of unreplicated regions to fork densities, replication timing and near-completion dynamics. We show that reliable completion depends not only on the overall level of initiation, but also on how that activity is distributed over the genomic distances that forks can traverse. Once fork progression is specified, the lowest accumulated initiation across fork-accessible regions provides a sufficient, worst-case measure of how rapidly unreplicated DNA is resolved. This converts the well-established biological principle that broad initiation-poor regions hinder completion into a quantitative bound, showing that they can delay uniform locuswise near-completion even when initiation is abundant overall or most loci replicate early on average. Our framework thus provides a quantitative route from experimentally informed replication programmes to the identification of potential completion vulnerabilities.

To establish this result, we generalise earlier KJMA descriptions of replication by introducing a joint size-position-time density of unreplicated intervals, with initiation and fork speed allowed to vary across genomic position and time. Integrating over interval size recovers the spatially inhomogeneous mean-field fork-density system, thereby linking interval-based models of nucleation, growth and coalescence~\cite{Herrick2002,PhysRevE.71.011908,PhysRevE.71.011909} with the formulation of Gauthier et al.~\cite{gauthier2012modeling}, while extending previous well-posedness analysis for size-dependent replication models~\cite{Nieto2022}. This reduction retains the geometric interpretation of unreplicated intervals while producing a system suitable for analysing replication timing and completion. We establish existence, uniqueness and positivity (Theorem~\ref{thm:local-well-posedness}), together with continuous dependence on the initial data (Proposition~\ref{prop:continuous_dependence}), showing that prescribed initiation and fork-speed fields generate a stable and physically admissible replication programme. For data compatible with the underlying geometry (Lemma~\ref{lem:compatibility-propagation}), normalisation by the unreplicated fraction reveals a cancellation of the nonlinear coalescence term, reducing the system to linear transport along replication-fork characteristics (Proposition~\ref{prop:normalised-equations}). This yields global existence (Theorem~\ref{thm:global-compatible}) and a characteristic representation in which the backward causal region of a locus identifies the initiation events capable of reaching it by a given time (Lemma~\ref{lem:general-light-cone-representation}). From this representation, we derive bounds on $T_\varepsilon$, the time by which almost every locus has replicated in all but an $\varepsilon$-fraction of cells (Proposition \ref{prop:geometry_master_completion_bound}), and extend the same geometric approach to expected replication times (Corollary~\ref{cor:geometry_master_expectation_bound})
and their spatial variation on \(\mathbb R\)
(Proposition~\ref{prop:expected-replication-time}). The estimates recover the exact KJMA solution on \(\mathbb R\) in the homogeneous limit, while providing conservative guarantees
on the resolution of the final unreplicated regions in heterogeneous settings.

In addition, our results complement earlier work on the random completion problem. Previous studies have asked how stochastic origin firing can nevertheless support timely genome duplication, highlighting increasing initiation through S phase, excess licensed or dormant origins, and the spatial organisation of origins as possible solutions~\cite{hyrien2003paradoxes,jun2008just,yang2008xenopus,goldar2008dynamic,blow2011dormant,karschau2012optimal}. Yang and Bechhoefer characterised completion through the statistics of the final coalescence events, obtaining an extreme-value distribution for completion time~\cite{yang2008xenopus}, while Gauthier et al.\ estimated completion from fork-number statistics in a spatially heterogeneous mean-field model~\cite{gauthier2012modeling}. Our contribution is to recast the problem in terms of what can be guaranteed from a prescribed heterogeneous initiation and fork-speed programme. Rather than reconstructing the distribution of the final replication event, we identify sufficient conditions under which the unreplicated fraction becomes uniformly small across the genome. The local initiation mass (Definition~\ref{def:local_initiation_mass}) measures the lowest initiation accumulated over genomic intervals of a given size and brings several proposed completion mechanisms into a common quantitative framework. Dormant-origin activation raises the initiation available locally, more even origin placement limits broad initiation-poor regions, and faster fork progression increases the distance over which initiation elsewhere can rescue a late locus~\cite{ge2007dormant,ibarra2008excess,almamun2016inevitability,moreno2016unreplicated,karschau2012optimal}. In this way, the analysis clarifies how these mechanisms overcome stochastic origin firing by ensuring sufficient initiation across the regions accessible to replication forks.

To compare the resulting estimates with the corresponding stochastic forward model, we apply them to finite-grid Monte Carlo simulations driven by heterogeneous Repli-seq-derived initiation landscapes~\cite{hansen2010sequencing,zhao2020high}. This controlled comparison assesses whether the empirical completion-time estimates are consistent with the analytical curves and quantifies how conservative those curves remain under realistic spatial variation, different tolerances and replication programmes. It does not by itself validate the continuum theorem or provide independent biological validation, since the same fitted landscapes underlie both calculations. The intervals minimising local initiation mass also identify candidate completion impediments that may be obscured by population-averaged timing profiles, providing a possible basis for comparing cell types or perturbations~\cite{berkemeier2025dna,berners2025regulation,retkute2012mathematical,hawkins2013highresolution}. Finally, periodic-domain examples motivated by ecDNA illustrate that the same estimates extend naturally to circular replication geometries~\cite{jaworski2025ecdna}.

The main limitations of our approach arise from the scale at which replication is represented. The coalescence term is a mean-field closure and therefore neglects correlations among origins, forks and neighbouring gaps, while population-level fractions cannot resolve the rare single-cell outcomes that ultimately define failure of completion. Initiation and fork speed are prescribed effective fields, so the model does not itself predict checkpoint activation, dormant-origin firing or adaptive responses to replication stress. These fields are also not uniquely identifiable from Repli-seq alone, and the interval and half-line formulations still require a rigorous boundary theory. The resulting bounds should therefore be interpreted as conditional guarantees for a specified effective replication programme, rather than direct predictions of how individual cells respond to stress.

These limitations also point to several natural directions for further work. Connecting the mean-field closure to an underlying stochastic process would enable single-cell completion probabilities to be quantified, while integrating replication timing with origin maps, fork-directionality profiles and single-molecule fork-speed measurements would better constrain the inputs and propagate their uncertainty into the estimates. Extensions incorporating checkpoint feedback and stress-responsive initiation could examine how dormant origins preserve replication under perturbation and in therapeutic contexts~\cite{jones2025high}. Spatial completion estimates could then be mapped systematically and compared with stress-induced under-replication, fragile sites and double-strand breaks, although additional modelling would be needed to translate delayed replication into DNA-damage probabilities~\cite{brison2019transcription,berkemeier2025dna}. Across cell types, tumours and treatments, these measures could be evaluated as candidate biomarkers of replication stress and therapeutic sensitivity, subject to independent experimental and clinical validation~\cite{konstantinopoulos2021replication}. More broadly, locus-resolved estimates of incomplete replication could inform cell-based models of checkpoint activation, repair, arrest and death, linking genome-scale replication kinetics to cell fate and population response~\cite{hamis2021targeting}. Our work therefore provides both a rigorous analytical foundation for heterogeneous replication dynamics and a route from spatial replication programmes to testable predictions of genome instability, treatment vulnerability and cellular outcome.

\section*{Data and code availability}

Replication timing data used in this study were obtained from publicly available Repli-seq datasets from ENCODE~\cite{hansen2010sequencing}, as processed in~\cite{berkemeier2025dna}. Code for the fitting procedure, stochastic simulations, completion-time estimates and figure generation are available at \href{https://github.com/fberkemeier/dna-replication-completion}{https://github.com/fberkemeier/dna-replication-completion}.

\section*{Acknowledgements}

F.B. and M.A.B. were supported by the Leverhulme Trust (RPG-2022-028) and the Biotechnology and Biological Sciences Research Council (UKRI1912). Additional funding and career support were provided by a Rokos Postdoctoral Associate position at Queens’ College Cambridge and a College Research Associate position at Emmanuel College awarded to F.B. and a fellowship at St John's College Cambridge to M.A.B. F.B. also received bridging support from the Department of Pathology, University of Cambridge. This work was performed using resources provided by the Cambridge Service for Data Driven Discovery (CSD3), operated by the University of Cambridge Research Computing Service (\href{https://www.csd3.cam.ac.uk}{https://www.csd3.cam.ac.uk}), and supported by Dell EMC and Intel through Tier-2 funding from the Engineering and Physical Sciences Research Council (capital grant EP/T022159/1) and DiRAC funding from the Science and Technology Facilities Research Council (\href{https://dirac.ac.uk}{https://dirac.ac.uk}).

\bibliographystyle{unsrt}
\bibliography{DNALiterature}

\end{document}